\documentclass{amsart}
\usepackage[english]{babel}     
\usepackage{graphicx}           
\usepackage{amsmath,amssymb}    
\usepackage{amsthm}             
\usepackage{mathtools}          
\usepackage[shortlabels]{enumitem}           
\usepackage{listings}           
\usepackage{todonotes}          
\usepackage[hidelinks]{hyperref}           
\usepackage{subfiles}
\usepackage{tikz-cd}            
\tikzcdset{scale cd/.style={every label/.append style={scale=#1},
		cells={nodes={scale=#1}}}}
\usepackage{tikz}

\usepackage{pgfplots}
\usepgfplotslibrary{fillbetween}
\pgfplotsset{compat=1.16}
 \usepackage{caption}
\usepackage{adjustbox}
\usepackage{verbatim}
\usepackage[titletoc]{appendix}
\setlist[enumerate, 1]{label = \roman*), ref = \roman*),font = \textup}

\newtheorem{theorem}{Theorem}[section]
\newtheorem{corollary}[theorem]{Corollary}
\newtheorem{lemma}[theorem]{Lemma}
\newtheorem{proposition}[theorem]{Proposition}

\theoremstyle{definition}
\newtheorem{definition}[theorem]{Definition}
\newtheorem{example}[theorem]{Example}

\theoremstyle{remark}
\newtheorem{remark}[theorem]{Remark}

\numberwithin{equation}{section}

\newcommand{\R}{\mathbb{R}}
\newcommand{\N}{\mathbb{N}}
\newcommand{\Z}{\mathbb{Z}}
\newcommand{\C}{\mathbb{C}}

\newcommand{\FP}{\mathcal{FP}}
\newcommand{\LC}{\mathcal{LC}}
\newcommand{\PL}{\mathcal{PL}}
\newcommand{\SN}{\mathcal{SN}}
\newcommand{\ol}{\overline}
\newcommand\restr[2]{{
		\left.\kern-\nulldelimiterspace 
		#1 
		\vphantom{\big|} 
		\right|_{#2} 
}}
\newcommand\uprestr[3]{{
		\left.\kern-\nulldelimiterspace 
		#1 
		\vphantom{\big|} 
		\right|_{#2}^{#3} 
}}

\newcommand{\beq}[1]{\begin{equation} \label{#1}}
\newcommand{\eeq}{\end{equation}}

\DeclareRobustCommand{\SkipTocEntry}[5]{}

\DeclareMathOperator{\dom}{Dom}

\DeclareMathOperator{\id}{id}
\DeclareMathOperator{\supp}{supp}

\DeclareMathOperator{\End}{End}

\DeclareMathOperator{\Ind}{Ind}
\DeclareMathOperator{\Ad}{Ad}

\DeclareMathOperator{\diam}{diam}

\newcounter{char}
\loop\ifnum\value{char}<27
\expandafter\edef\csname c\Alph{char}\endcsname{\noexpand\mathcal{\Alph{char}}}  
\expandafter\edef\csname d\Alph{char}\endcsname{\noexpand\mathbb{\Alph{char}}}   
\expandafter\edef\csname f\Alph{char}\endcsname{\noexpand\mathfrak{\Alph{char}}} 
\expandafter\edef\csname e\Alph{char}\endcsname{\noexpand\mathsf{\Alph{char}}}   
\expandafter\edef\csname s\Alph{char}\endcsname{\noexpand\mathscr{\Alph{char}}}  
\addtocounter{char}{1}
\repeat

\begin{document}

\title[A partitioned index theorem for noncompact hypersurfaces]{A partitioned manifold index theorem for noncompact hypersurfaces}

\author{Peter Hochs and Thijs de Kok}
\address{Institute for Mathematics, Astrophysics and Particle Physics, Radboud University, \texttt{p.hochs@math.ru.nl}}
\address{Institute for Mathematics, Astrophysics and Particle Physics, Radboud University, \texttt{thijs.dekok@ru.nl}}

\date{\today}

\maketitle
\begin{abstract}
	Roe's partitioned manifold index theorem applies when a complete Riemannian manifold $M$ is cut into two pieces along a compact hypersurface $N$. It states that a version of the index of a Dirac operator on $M$ localized to $N$ equals the index of the corresponding Dirac operator on $N$. This yields obstructions to positive scalar curvature, and implies cobordism invariance of the index of Dirac operators on compact manifolds. We generalize this result to cases where $N$ may be noncompact, under assumptions on the way it is embedded into $M$. This results in an equality between two classes in the $K$-theory of the Roe algebra of $N$. Bunke and Ludewig, and Engel and Wulff, have recently obtained related results based on different approaches. 
\end{abstract}

\setcounter{tocdepth}{1}
\tableofcontents

\section*{Introduction}
\addtocontents{toc}{\SkipTocEntry}
\subsection*{Background}

The classical Atiyah--Singer index of elliptic differential operators on compact manifolds \cite{Atiyah1968indexI, Atiyah1968indexIII} was generalized to noncompact Riemannian manifolds $M$ by Roe  \cite{Roe1996indextheory}. This generalized, higher index takes values in $K_*(C^*(M))$,  the $K$-theory of the \emph{Roe algebra} of $M$. It has been applied to various problems, such as the Novikov conjecture and the study of Riemannian manifolds with positive scalar curvature. (See \cite{HR05a, HR05b, HR05c, WillettYu, Yu98, Yu00} for just a few examples.)

To compare indices of operators or to determine the vanishing of such an index, it is useful to convert the $K$-theory class to a number via a homomorphism $K_*(C^*(M))\rightarrow \C$. This is usually done by considering the pairing of $K$-theory with a cyclic cocycle. The easiest examples of such cyclic cocycles are the tracial functionals on the $C^*$-algebra $C^*(M)$. However, it turns out that if $M$ is a noncompact manifold, the only tracial functional on $C^*(M)$ is the trivial one \cite[Lem. 4.1]{Roe1996indextheory}. Therefore, to get an interesting pairing, a more sophisticated cyclic cocycle must be constructed. In \cite{Roe1988dualtoeplitz}, Roe showed that in the case of an odd-dimensional noncompact manifold, we can construct such cyclic cocycles by considering partitions of $M$ by a compact hypersurface $N$. Moreover, he showed that if the operator $D$ on $M$ under consideration is a Dirac operator, then the number resulting from pairing the index of $D$ with such cyclic cocycle is equal to the Fredholm index of a Dirac operator induced on the compact manifold $N$, which is a topological invariant by the Atiyah--Singer index theorem. This result is now known as Roe's partitioned manifold index theorem. After the first publication of the theorem, it was reformulated by Higson, avoiding the use of cyclic cocycles, and its proof was simplified considerably \cite[Thm. 4.4]{Roe1996indextheory}, \cite{higson1991cobordisminvariance}.


The partitioned manifold index theorem has several immediate implications to the index theory of Dirac operators on both compact and noncompact manifolds. To name a few: it gives several obstructions to the existence of metrics with uniformly positive scalar curvature on noncompact spin manifolds \cite[Thm. 9.1]{Roe1988dualtoeplitz}, \cite[Prop. 3.2]{Roe2016positivecurvature}, and new, more refined obstructions to the existence of metrics with positive scalar curvature on compact spin manifolds \cite[Prop. 4.7]{Roe1996indextheory}; it proves homotopy invariance of the higher signature  of a compact manifold \cite[prop. 4.8]{Roe1996indextheory}; and it shows that indices of Dirac operators on compact manifolds are cobordism-invariant \cite{higson1991cobordisminvariance}. It is therefore interesting to see if the partitioned manifold index theorem can be generalized.

Since its first publication, several generalizations into multiple directions have already been made. For example, in \cite{Zadeh2010indextheory}, a similar theorem is proven for Dirac operators on sections of a bundle of Hilbert $A$-modules and in \cite{Seto2018toeplitz}, a version of the partitioned manifold index theorem for even-dimensional manifolds is proven. In \cite{Schick2018largescale}, a multipartitioned manifold index theorem is proven for spin manifolds of arbitrary dimension with multiple partitioning hypersurfaces that intersect compactly. In \cite{Karami2019relative-partitioned}, a relative partitioned manifold index theorem is proven for a pair of partitioned manifolds that are isometric outside a closed set that intersects the partitioning hypersurface compactly and transversally. In this setting, the hypersurface is not assumed to be compact. In \cite{Ludewig2025large}, a generalized trace formula is proven that unifies Roe's partitioned manifold index theorem with the Kubo trace formula and the Kitaev trace formula, which are formulas from physics describing the quantization of the Hall conductance of a 2-dimensional sample.

Bunke and Ludewig recently obtained a version of the partitioned manifold index theorem  for Callias type Dirac operators, where the partitioning hypersurface $N$ need not be compact, see \cite[Theorem 4.44]{bunke2024coronas}. If one takes the Callias potential to be zero, then \cite[Corollary 4.50]{bunke2024coronas}  becomes a partitioned manifold index theorem for usual Dirac operators, which applies to noncompact hypersurfaces. This result requires weaker assumptions on the coarse structure on $N$ than ours, and yields extra information for Dirac operators that are invertible outside certain classes of subsets. On the other hand, the assumption is made that all data have product structures near $N$. In this paper, we do not make that assumption, and in fact, a large part of the work in Section \ref{section: proof general case} is done to handle the case where the relevant  metrics, Clifford actions etc.\ do not have such a product form.

Finally, in the last weeks of preparing this article, a partitioned manifold index theorem for noncompact hypersurfaces was published by Engel and Wulff \cite[Cor. 7.6]{Engel2025relative}, which follows as a direct consequence of their more general index theorem for multipartitioned manifolds \cite[Thm. 7.5]{Engel2025relative}. Apart from allowing for multipartitioned manifolds, their theorem also seems to accommodate the possibility for multiple connected components naturally. On the other hand, it is based on the cobordism invariance of the coarse index \cite{Wulff2012bordisminvariance}, which is not used in the proof of our main result. On the contrary, cobordism invariance of the coarse index follows as a corollary (Corollary \ref{cor: cobordism invariance of the index}). Moreover, their multipartitioned index theorem applies to spin-Dirac operators, whereas our result applies to more general Dirac operators.

The partitioned manifold indices to which the index of the Dirac operator on $N$ is related, are constructed in very different ways in \cite{bunke2024coronas}, \cite{Engel2025relative} and this paper (see Definition \ref{def: partitioned index} for our construction). A consequence of the partitioned manifold index theorems in \cite{bunke2024coronas}, \cite{Engel2025relative} and Theorem \ref{thm: PMT} in this paper, is that the three constructions yield the same result in cases where the conditions of all three results are satisfied.

\addtocontents{toc}{\SkipTocEntry}
\subsection*{Results}
In this article, we generalize
 Roe's partitioned manifold index theorem to the case where the partitioning hypersurface need not be compact (Theorem \ref{thm: PMT}). Our theorem differs from the theorems proven in \cite{Schick2018largescale} and \cite{Karami2019relative-partitioned} in that we do not use any construction to obtain a compact submanifold of our hypersurface such that the indices are just integers. We fully embrace the possibly noncompact hypersurface. Accordingly, we use the indices defined in the $K$-theory of Roe algebra $C^*(N)$ in full generality and do not appeal to any construction to convert these to integer indices. Our theorem also differs from \cite[Cor. 4.50]{bunke2024coronas} as we do not pose any restrictions on the Dirac operator. Instead, there are only a few restrictions on how the (noncompact) partitioning hypersurface is embedded in the ambient manifold.
Concretely, we prove that if the hypersurface $N\subseteq M$ is embedded in a way such that 
\begin{enumerate}
	\item its intrinsic coarse geometry is equivalent to the coarse geometry it inherits as metric subspace of $M$;
	\item $N$ admits a uniform tubular neighbourhood on which the Riemannian metric of $M$ can be compared with the product metric up to a uniform constant;
\end{enumerate}
then the equality
\begin{equation}
	\label{eq: equality PMIT}
	\Phi(\Ind(D_N)) = \Ind(D;N)
\end{equation} 
holds. In \eqref{eq: equality PMIT}, $\Ind(D_N)$ is the index of the Dirac operator $D_N$ induced on the hypersurface $N$ and $\Ind(D;N)$ is the partitioned index associated to the Dirac operator $D$ on $M$. The map $\Phi$ is a canonically defined map between the different $K$-theory groups in which these indices are defined. For the precise definitions and formulation of the theorem, see Section \ref{section: partitioned mfds and PMT}.

The extra assumptions imposed on how $N$ is embedded in $M$ are necessary to control the coarse geometry of our manifold when we make certain modifications to it. Based on the proofs of the partitioned manifold index theorem in \cite{Roe1988dualtoeplitz} and \cite{higson1991cobordisminvariance}, we first prove the theorem for a product manifold $\R\times N$, which is partitioned by the hypersurface $\{0\}\times N$ (Theorem \ref{thm: PMT for product}). In this step, the use of coarse geometry and its connections to $K$-theory and $K$-homology form an important new ingredient. This machinery allows us to deal with noncompact hypersurfaces, but completely changes the proof of the product case when compared to \cite{higson1991cobordisminvariance}. Then we derive the general case by modifying our original manifold to a product manifold. It is in this last step of modifying a general partitioned manifold to a product manifold, that the assumptions on the embedding of the hypersurface into the ambient manifold are needed to control the change of coarse geometry. If the embedded hypersurface is compact, these are automatically satisfied and we recover the original version of the partitioned manifold index theorem (Corollary \ref{cor: PMT for compact hypersurface}).

As applications of  our main result, we obtain
 an obstruction to the existence of uniformly positive scalar curvature metrics on spin manifolds (Corollary \ref{cor: obstruction UPSC metrics}, a version of \cite[Remark 4.51]{bunke2024coronas} and \cite[Theorem 7.2]{Engel2025relative} with different assumptions) and a generalized notion of cobordism invariance of the index of Dirac operators (Corollary \ref{cor: cobordism invariance of the index}), which was previously already proved in \cite{Wulff2012bordisminvariance} and \cite[Thm. 4.12]{Wulff2018twistedops}.

At the end of this paper, we describe two generalizations of the main result. The first (Theorem \ref{thm: PMT Gamma}) is an equivariant version, in cases where a discrete group $\Gamma$ acts properly and freely on $M$, preserving $N$. The proof of this version is entirely analogous to the proof of the non-equivariant version. An interesting special case (Corollary \ref{cor: PMT for cocompact hypersurface}) is the case where $N/\Gamma$  is compact. Then the conditions of the partitioned manifold index theorem are automatically satisfied. Furthermore, one may apply existing index theorems (such as the one in \cite{CM90}) to compute the pairing of the equivariant version of the left hand side of \eqref{eq: equality PMIT} with suitable cyclic cohomology classes, to obtain numerical obstructions to positive scalar curvature. As a consequence (Corollary \ref{cor higher PMIT}), we find that the higher index of the spin-Dirac operator on any compact, connected hypersurface $N$ of a spin manifold $M$ such that the induced map $\pi_1(N) \to \pi_1(M)$ is injective, is an obstruction to Riemannian metrics on $M$ of uniformly positive scalar curvature.

The second generalization, Theorem \ref{thm: PMT (mult comp)}, applies to cases where $N$ may be disconnected, but has finitely many connected components. We indicate how the proof of the connected case can be modified to obtain this generalization. This extension is relevant for example to cobordism invariance of the coarse index: it is desirable to allow cobordisms between possibly disconnected manifolds.

\addtocontents{toc}{\SkipTocEntry}
\subsection*{Acknowledgments}
We are grateful to Ulrich Bunke and Nigel Higson for helpful comments.
PH was supported by the Netherlands Organization for Scientific Research NWO, through grants OCENW.M.21.176 and OCENW.M.23.063.

\section{Preliminaries}
\label{section: preliminaries}
The first section will be devoted to the necessary preliminaries for the rest of this article. We fix notation and assumptions concerning Dirac operators, and briefly introduce Roe algebras and the higher index. We also review Paschke duality.

\subsection{Gradings, Dirac operators and Clifford actions}
\label{subsec: gradings, Dirac operators and Clifford actions}

In this entire section, $M$ will be a complete Riemannian manifold with a Hermitian bundle $S$ over $M$. If $M$ is \emph{even-dimensional}, we assume that $S$ is graded and if $M$ is \emph{odd-dimensional} we assume no grading. By a grading, we mean a $\Z/2\Z$-grading.

In the context of $K$-homology, we will also use the notion of a multigrading.
\begin{definition}
	\label{def: multigrading}
	\begin{enumerate}
		\item Let $p\in \N_0$. A $p$-\emph{multigraded} Hilbert space is a graded Hilbert space $H$ with $p$ anticommuting, odd, unitary operators $\varepsilon_1, \dots, \varepsilon_p$ such that $\varepsilon_i^2 = -1$ for all $1\leq i\leq p$.
		A $-1$-\emph{multigraded} Hilbert space is defined to be an ungraded Hilbert space.
		\item Let $p\in \N_0$. A $p$-\emph{multigraded} Hermitian vector bundle is a graded Hermitian vector bundle $S$ with $p$ anticommuting, odd, fiberwise unitary morphisms $\varepsilon_1, \dots, \varepsilon_p$ such that $\varepsilon_i^2 = -1$ for all $1\leq i\leq p$.
		A $-1$-\emph{multigraded} Hermitian vector bundle is defined to be an ungraded Hermitian vector bundle.
	\end{enumerate}
\end{definition}

\begin{remark}
	\label{rem: multigraded bundle gives multigraded Hilbert space}
	If $S$ is a $p$-multigraded vector bundle, then the (multi)grading morphisms define maps on $C_c^{\infty}(M;S)$, which extend to unitaries on the Hilbert space $L^2(M;S)$. This gives $L^2(M;S)$ the structure of a $p$-multigraded Hilbert space.
\end{remark}

\begin{definition}
	\label{def: multigraded operators}
	In either of the following cases where
	\begin{enumerate}
		\item $H$ is a $p$-multigraded Hilbert space and $T \in \cB(H)$;
		\item $H$ is a $p$-multigraded Hilbert space and $T$ is an unbounded operator on $H$ for which $\dom(T)$ is an invariant subspace for all multigrading operators;
		\item $S$ is a $p$-multigraded vector bundle and $T\colon C^{\infty}(M;S)\rightarrow C^{\infty}(M;S)$ is a $\C$-linear map. 
	\end{enumerate}
	Then $T$ is \emph{multigraded} if $[T, \varepsilon_i] = 0$ for $i = 1, \ldots, p$, where $\varepsilon_i$ denote the multigrading operators/morphisms.
\end{definition}

%
%
%
%

\begin{definition}
	\label{def: Dirac bundle}
	A \emph{Dirac bundle} is a Hermitian vector bundle $S$ over $M$ with a morphism of real vector bundles 
	\[c\colon T^*M \rightarrow \End(S),\quad\quad (p, \xi)\mapsto c(\xi)\in \End(S_p),\]
	such that $c(\xi)$ is skew-adjoint and satisfies $c(\xi)^2 = -g_p(\xi,\xi)\id_{S_p}$. If $M$ is even-dimensional, we assume that $S$ is graded and that $c(\xi)$ is odd for any $\xi \in T^*_pM$. The action of $T^*M$ on $S$ induced by this morphism is called \emph{the Clifford action} or \emph{Clifford multiplication} and is also denoted by $c(\xi)(s) = \xi\cdot s$ for $s\in S_p$ and $\xi\in T^*_pM$.
\end{definition}


Given a Dirac bundle, we would like to construct a Dirac operator that is associated to it. 

\begin{definition}
	\label{def: Dirac connection}
	A \emph{Dirac connection} on a Dirac bundle $S$ is a Hermitian connection such that
%
%
		for any $\omega \in \Omega^1(M)$, $u \in C^{\infty}(M;S)$ and $X \in \mathfrak{X}(M)$,
		\[
		\nabla_X(\omega\cdot u) = \nabla^{LC}_X\omega \cdot u + \omega \cdot \nabla_X u,
		\]
		where $\nabla^{LC}$ is the connection on $1$-forms induced by the Levi--Civita connection.

	If $S$ is graded, then $\nabla$ is required to be even for the grading. That is, for any $X\in \mathfrak{X}(M) $ the induced map $\nabla_X$ is even for the grading on $C^{\infty}(M;S)$.
\end{definition}

The existence of Dirac connections is guaranteed by \cite[Cor. 3.41]{HeatKernels1992Berline}.

\begin{definition}
	\label{def: associated Dirac operator}
	Let $S$ be a Dirac bundle over $M$ with Dirac connection $\nabla$ and consider its Clifford multiplication as a morphism  $c\colon T^*M\otimes S\rightarrow S$. Then the \emph{Dirac operator associated to $c$ and $\nabla$} is the composition 
	\[\begin{tikzcd}
		{C^{\infty}(M; S)} & {C^{\infty}(M;    T^*M \otimes S)} & {C^{\infty}(M; S).}
		\arrow["\nabla", from=1-1, to=1-2]
		\arrow["{c}", from=1-2, to=1-3]
	\end{tikzcd}\]
\end{definition}
%

Denote by $T^r_M$ the trivial rank $r$ complex vector bundle over $M$. 
\begin{example}
	\label{example: Dirac operator on R}
	The following example is a simple Dirac operator 
	that 
	 will be of great use in the rest of this article. Consider the manifold $M = \R$ with the Euclidean metric.
	Consider the complex vector bundle $T^1_{\R}$ over $\R$ with the canonical Hermitian metric.
	Then the vector bundle morphism
	\begin{align}
		\label{eq: clifford action on R}
		c\colon\ \R\times\R &\rightarrow \End(T^1_{\R}),\notag\\
		(x, v)&\mapsto [(x, \lambda)\mapsto (x, iv\lambda)],
	\end{align} 
	defines a Clifford action of $\R$ on $T^1_{\R}$. The connection $\nabla = d$ on $T^1_{\R}$
	is a Dirac connection for this action. 
	The associated Dirac operator is $D = i\frac{d}{dx}$.
	Alternatively, we could have taken $-c$ as Clifford action and obtain $-D$ as Dirac operator.
\end{example}

\begin{example}[{\cite[~p. 314]{Higson2000Khomology}}]
	\label{example: Dirac operator induced Dirac class}
	There is yet another natural Dirac operator on $\R$. With the same notation as in Example \ref{example: Dirac operator on R}, consider the algebra bundle $S_{\R} = \R\times\C_1$, where $\C_1$ is the complex Clifford algebra on one generator $e_1$. The map $c\colon \R\times \R\rightarrow \End(S_{\R})$ given by $\frac{d}{dx}\mapsto e_1$ defines a Clifford action on $S_{\R}$ via left multiplication.  After identifying $S_{\R}$ with $T^2_{\R}$ via the frame $(1, e_1)$ and considering the corresponding Hermitian metric, the flat connection on $S_{\R}$ defined by declaring the frame $(1, e_1)$ to be parallel is a Dirac connection. The induced Dirac operator is given by
	\begin{equation*}
		D_{\R} = \begin{pmatrix}
			0&-\frac{d}{dx}\\
			\frac{d}{dx}&0
		\end{pmatrix}.
	\end{equation*}
	This Dirac operator does come with more structure, namely a $1$-multigrading. Define 
	\begin{equation*}
		\gamma_{\R} = \begin{pmatrix}
			1&0\\
			0&-1
		\end{pmatrix}
		\quad\text{ and }\quad 
		\varepsilon_1 =\begin{pmatrix}
			0&-1\\
			1&0
		\end{pmatrix},
	\end{equation*}
	it is easy to check that this indeed provides $D_{\R}$ with the structure of a $1$-multigraded Dirac operator, using $\gamma_{\R}$ as grading and $\varepsilon_1$ as multigrading. 
\end{example}


\subsection{Roe algebras and related algebras}
\label{subsec: Roe algebras}

In this section the Roe algebra and some related $C^*$-algebras will be introduced. These will play a prominent role in this article as the indices of differential operators are elements of their $K$-theory. 
In this entire section, $(X, d_X)$ will be a second-countable metric space having the Heine--Borel property---that is, subsets are compact if and only if they are closed and bounded. In particular, such an $X$ is also separable, Hausdorff and locally compact. For more information on coarse geometry, the associated $C^*$-algebras and the properties of the $K$-theory of these $C^*$-algebras, we refer to \cite[Ch. 6]{Higson2000Khomology}, \cite[Ch. 2--3]{Roe1996indextheory}.


If $A,B\subseteq X$, put $d_X(A, B):= \inf\{d_X(a, b)\colon a\in A,\ b\in B\}$. Moreover, if $Z\subseteq X$ is a subset, then we define 
\[
B_r(Z) = \bigcup_{z\in Z}B_r(z) = \{x \in X\colon \text{there exists a $z \in Z$ such that $d_X(x, z) < r$}\}. 
\]
Note that for any $r, \delta >0$ we have that $\ol{B_r(Z)} \subseteq B_{r+\delta}(Z)$. 

\begin{definition}
	\label{def: nondegenerate representation}
	Let $A$ be a $C^*$-algebra and let $\rho\colon A\rightarrow \cB(H)$ be a representation of $A$ on some Hilbert space $H$. Then $\rho$ is said to be \emph{nondegenerate} if $\rho(A)H\subseteq H$ is dense.
\end{definition}

\begin{definition}
	\label{def: operator properties}
	Let $\rho\colon C_0(X) \rightarrow \cB(H)$ be a nondegenerate representation of $C_0(X)$ on a separable Hilbert space $H$. Then:
	\begin{enumerate}
		\item $T \in \cB(H)$ is \emph{locally compact} if for any $f \in C_0(X)$ the operators $\rho(f)T$ and $T\rho(f)$ are compact. Denote the set of locally compact operators in $\cB(H)$ by $\LC(X;H)$.
		\item $T \in \cB(H)$ is \emph{pseudolocal} if for any $f \in C_0(X)$ the operator $[T, \rho(f)]$ is compact. Denote the set of pseudolocal operators in $\cB(H)$ by $\PL(X;H)$.
		\item $T \in \cB(H)$ has \emph{finite propagation} if there exists a $P>0$ such that for any $f,g \in C_0(X)$ with $d_X(\supp(f), \supp(g)) > P$ the operator $\rho(f)T\rho(g)$ equals 0. Denote the set of operators with finite propagation in $\cB(H)$ by $\FP(X;H)$.
		\item Let $Z\subseteq X$ be a subset. Then $T \in \cB(H)$ is \emph{supported near $Z$} if there exists an $R>0$ such that for any $f \in C_0(X)$ with $d_X(\supp(f), Z)> R$ both $\rho(f)T$ and $T\rho(f)$ equal 0. Denote the set of operators in $\cB(H)$ that are supported near $Z$ by $\SN_Z(X;H)$.
	\end{enumerate}
\end{definition}

\begin{remark}\label{rem: FP and SN properties extend to bounded Borel functions}
In the definitions of finite propagation and being supported near $Z$, one may equivalently replace the representation $\rho$ by  its unique extension $\tilde{\rho}\colon B^{\infty}(X)\rightarrow \cB(H)$ to the bounded Borel functions (e.g. via the associated spectral measure, see for example \cite[Thm. IX.1.14]{conway1990functionalanalysis}). (Usually, we will consider the representation by multiplication operators of $B^{\infty}(M)$ on the Hilbert space $L^2(M;S)$, where $S$ is a Hermitian vector bundle over the Riemannian manifold $M$, and one may also use $L^{\infty}(M)$ instead.)
\end{remark}
%
%

\begin{proposition}[{\cite[Prop. 6.3.7]{Higson2000Khomology}}, {\cite[p. 155]{Higson2000Khomology}}]
	\label{prop: property induced sets are *algebras}
	Let $\rho
	\colon C_0(X) \rightarrow \cB(H)$ be a nondegenerate representation of $C_0(X)$ on a separable Hilbert space $H$ and let $Z \subseteq X$ be a subset. Then $\LC(X;H)$, $\PL(X;H)$, $\FP(X;H)$ and $\SN_Z(X;H)$ are $*$-subalgebras of $\cB(H)$. Moreover, 
	\[
	\SN_Z(X;H)\cap \LC(X;H) \cap \FP(X;H) \subseteq  \LC(X;H) \cap \FP(X;H)\subseteq \FP(X;H)
	\] 
	are inclusions of 2-sided $*$-ideals.
\end{proposition}

\begin{definition}
	\label{def: (localized) Roe algebra}
	Let $Z \subseteq X$ be a subset. Then we have the \emph{Roe algebra} 
	\[
	C^*(X; H) := \ol{\LC(X; H)\cap\FP(X; H)},
	\] 
	the \emph{localized Roe algebra}
	\[
	C^*(Z\subseteq X; H) := \ol{\LC(X; H)\cap\FP(X; H)\cap\SN_Z(X; H)}
	\] 
	and 
	\[
	D^*(X; H):= \ol{\PL(X; H)\cap\FP(X; H)}.
	\]
	When no confusion is possible, we will write $C^*(X)$, $D^*(X)$ and $C^*(Z\subseteq X)$ and thus leave out the chosen representation.
\end{definition}

\begin{remark}
	\label{rem: D^*, FP and PL are unital}
	Note that the algebras $\FP(X)$, $\PL(X)$ and $D^*(X)$ are always unital, with the unit being the identity $I\in \cB(H)$.
\end{remark}

\begin{corollary}[{\cite[p.~22]{Roe1996indextheory}}, {\cite[p.~2]{Roe2016positivecurvature}}]
	\label{cor: Roe C^* algebra is C^* alg}
	$C^*(X; H)$, $D^*(X; H)$ and $C^*(Z\subseteq X; H)$ are $C^*$-subalgebras of $\cB(H)$ and
	\[
	C^*(Z\subseteq X; H)\subseteq C^*(X; H)\subseteq D^*(X; H)
	\]
	are inclusions of $2$-sided $*$-ideals.
\end{corollary}

\begin{theorem}[{\cite[Prop. 10.5.6]{Higson2000Khomology}} {\cite[Lemma 2.1]{Roe2016positivecurvature}}]
	\label{thm: func calc thm for C^*M}
	Let $M$ be a connected, complete Riemannian manifold and let $D$ be a symmetric elliptic differential operator on $S$ with finite propagation speed. If $\varphi \in C_0(\R)$, then $\varphi(D) \in C^*(M)$.
\end{theorem}

\begin{theorem}[{\cite[Lem. 2.1]{Roe2016positivecurvature}}]
	\label{thm: func calc thm for D^*M}
	Let $M$ be a connected, complete Riemannian manifold and let $D$ be a symmetric elliptic differential operator on $S$ with finite propagation speed. If $\varphi \in C_b(\R)$ with finite limits at $\pm\infty$, then $\varphi(D) \in D^*(M)$.
\end{theorem}
\begin{remark}
	The assumptions in both theorems imply that the differential operator $D$ is essentially self-adjoint \cite[Prop. 10.2.11]{Higson2000Khomology}, so that the functional calculus is well-defined after passing to its operator-theoretic closure. 
\end{remark}

\begin{corollary}
	\label{cor: func calc thm for unitization C^*M}
	Let $M$ be a connected, complete Riemannian manifold and let $D$ be a symmetric elliptic differential operator on $S$ with finite propagation speed. If $\varphi \in C_b(\R)$ with finite and equal limits at $\pm\infty$,  then $\varphi(D) \in \widetilde{C^*(M)}$, the unitization of $C^*(M)$.
\end{corollary}
\begin{proof}
	This follows directly from Theorem \ref{thm: func calc thm for C^*M} as there exists a $\lambda \in \C$ such that $\varphi-\lambda \in C_0(\R)$.
\end{proof}

\subsection{\texorpdfstring{$K$}{TEX}-theory of Roe algebras and its functorial properties}
\label{subsubsec: K-theory of Roe algebras and functorial properties}

In this section, several useful results about the $K$-theory of the Roe algebra will be discussed. In particular, its functorial properties and its invariance under the chosen representation will be of great use.

\begin{definition}[{\cite[p.~9]{Roe1996indextheory}}, {\cite[Def. 6.1.14, 12.4.2]{Higson2000Khomology}}]
	\label{def: coarse maps and being close}
	Let $(X, d_X)$ and $(Y, d_Y)$ be metric spaces. A (not necessarily continuous) Borel measurable function $q\colon X\rightarrow Y$ is a \emph{coarse map} if it has the following two properties:
	\begin{enumerate}
		\item (Uniform expansiveness) For any $R>0$ there exists an $S>0$ such that for any $x, x' \in X$ \[d_X(x, x') \leq R \implies d_Y(q(x), q(x')) \leq S.\]
		\item (Metric boundedness) For any bounded $B \subseteq Y$, its inverse image $q^{-1}(B) \subseteq X$ is bounded.
	\end{enumerate}
	A \emph{uniform} map from $X$ to $Y$ is a continuous coarse map. 
	Two coarse maps $q_0, q_1 \in \mathcal{C}(X, Y)$ are \emph{close}, denoted by $q_0 \simeq q_1$, if there exists a $K>0$ such that $d_Y(q_0(x), q_1(x)) \leq K$ for all $x \in X$. Two metric spaces $(X, d_X)$ and $(Y, d_Y)$ are  \emph{coarsely equivalent} if there exist coarse maps $q\colon X \rightarrow Y$ and $r\colon Y \rightarrow X$ such that $r\circ q \simeq \id_X$ and $q\circ r\simeq \id_Y$ and are \emph{uniformly equivalent} if there exists a homeomorphism $q\colon X\rightarrow Y$ such that both $q$ and $q^{-1}$ are coarse maps.
\end{definition}

\begin{definition}
	\label{def: (very) ample representations}
	Let $A$ be a $C^*$-algebra and let $\rho\colon A\rightarrow \cB(H)$ be a representation of $A$ on a Hilbert space $H$. Then the representation $\rho$ is \emph{ample} if it is nondegenerate and if $\rho(a)$ is a noncompact operator for each nonzero $a \in A$. The representation $\rho$ is \emph{very ample} if it is unitarily equivalent to a direct sum of countably infinitely many copies of a fixed ample representation. 
\end{definition}

\begin{example}
	\label{ex: multiplication representation is ample}
	Let $M$ be a Riemannian manifold of dimension greater than 0 and let $S$ be a Hermitian vector bundle over $M$. Denote by $\rho\colon C_0(M)\rightarrow \cB(L^2(M; S))$ the representation by multiplication operators. Then $\rho$ is ample. By tensoring with $l^2(\N)$, and extending $\rho$ by the trivial representation on $l^2(\N)$, we obtain a very ample representation.
\end{example}


If $S, T \in \cB(H)$, then we write $S\sim T$ if $S-T$ is a compact operator.
 
\begin{definition}
	\label{def: (uniformly) covering isometry}  
	Let $q\colon X \rightarrow Y$ be a coarse map and let $\rho_X$ and $\rho_Y$ be nondegenerate representations of $C_0(X)$ and $C_0(Y)$ on Hilbert spaces $H_X$ and $H_Y$, respectively.
	An isometry $V\colon H_X\rightarrow H_Y$ \emph{covers} $q$ if there exists an $R>0$ such that for all $f\in C_0(Y)$ and $g \in C_0(X)$ with $d_Y(\supp(f), q(\supp(g)))>R$ the operator $\rho_Y(f)V\rho_X(g)$ equals 0.  
	
	If $q\colon X\rightarrow Y$ is a uniform map, then an isometry $V\colon H_X\rightarrow H_Y$ \emph{uniformly covers} $q$ if it covers $q$ as a coarse map and if 
	\[
	V^*\rho_Y(f)V \sim \rho_X(f\circ q)
	\]
	for all $f \in C_0(Y)$. 
\end{definition}
\begin{remark}
	In the definition of a uniformly covering isometry, we used that if $q\colon X\rightarrow Y$ is a uniform map between metric spaces with the Heine--Borel property, then metric boundedness of $q$ implies that $q$ is proper, and therefore $f\circ q \in C_0(X)$ for all $f \in C_0(Y)$.
\end{remark} 

If $H_1$ and $H_2$ are two Hilbert spaces and $V$ is an isometry from $H_1$ to $H_2$, then we define the $*$-homomorphism $\Ad_V\colon \cB(H_1)\rightarrow \cB(H_2)$ by $T\mapsto VTV^*$.

\begin{definition}
	\label{def: induced map by coarse map}
	Let $X$ and $Y$ be metric spaces and let $\rho_X$ and $\rho_Y$ be ample representations of $C_0(X)$ and $C_0(Y)$ on Hilbert spaces $H_X$ and $H_Y$, respectively. Let $q\colon X \rightarrow Y$ be a coarse map and let $V$ be any isometry covering $q$. Then we denote by
	\beq{eq funct Roe}
	q_*:= K_p(\Ad_V)\colon K_p(C^*(X;H_X))\rightarrow K_p(C^*(Y;H_Y))
	\eeq
	the induced map on $K$-theory. 
	
	Similarly, let $\rho_X$ and $\rho_Y$ be very ample representations and $q\colon X \rightarrow Y$ a uniform map. Let $V$ be any isometry uniformly covering $q$. Then we denote by
	\beq{eq funct Dstar}
	q_*:= K_p(\Ad_V)\colon K_p(D^*(X;H_X))\rightarrow K_p(D^*(Y;H_Y))
	\eeq
	the induced map on $K$-theory. 
\end{definition}
For the facts that $K_p(\Ad_V)$ maps $C^*(X;H_X)$ into $C^*(Y;H_Y)$ and $D^*(X;H_X)$ into $D^*(Y;H_Y)$ in the context of this definition, see \cite[Lem. 6.3.11, 12.4.4]{Higson2000Khomology}. 
For existence of covering isometries and independence of the maps \eqref{eq funct Roe} and \eqref{eq funct Dstar} of the choice of such an isometry, see {\cite[Prop. 6.3.12]{Higson2000Khomology}}, {\cite[Lem. 12.4.4]{Higson2000Khomology}}. 

By  \cite[Lem. 6.3.17]{Higson2000Khomology}, the constructions in Definition 
\ref{def: induced map by coarse map} are functorial. By  \cite[Prop. 6.3.16]{Higson2000Khomology}, coarse maps that are close induce the same maps on $K$-theory. These two facts have the following consequence.
%
%
%
\begin{corollary}[{\cite[Cor. 6.3.13]{Higson2000Khomology}} {\cite[p. 357--358]{Higson2000Khomology}}]
	\label{cor: K thoery independent of chosen representation}
	Let $X$ and $Y$ be metric spaces and let $\rho_X$ and $\rho_Y$ be ample representations of $C_0(X)$ and $C_0(Y)$ on Hilbert spaces $H_X$ and $H_Y$, respectively. 
	Then a coarse equivalence $q\colon X\rightarrow Y$  induces an isomorphism 
	\[
	q_*\colon K_p(C^*(X;H_X))\xrightarrow{\sim} K_p(C^*(Y;H_Y)).
	\]
	Similarly, if the representations are very ample and $q$ is a uniform equivalence, then it induces an isomorphism
	\[
	q_*\colon K_p(D^*(X;H_X))\xrightarrow{\sim} K_p(D^*(Y;H_Y)).
	\]
	In particular, taking $q = \id_X$ it follows that up to canonical isomorphism,  $K_p(C^*(X;H_X))$ does not depend on the choice of ample representation and $K_p(D^*(X;H_X))$ does not depend on the choice of very ample representation.
\end{corollary}
\begin{remark}
	Accordingly, if the specific representation used to define $C^*(X;H_X)$ or $D^*(X;H_X)$ is clear from the context, the Hilbert space $H_X$ will be dropped from the notation. 
\end{remark}

\begin{example}
	\label{ex: closed n-ball is coarsely equivalent}
	Let $(X, d_X)$ be a separable metric space with the Heine--Borel property and let $Y\subseteq X$ be a closed subspace, endowed with the subspace metric. Then for any $n\in \N$ the inclusion $Y \hookrightarrow \ol{B_n(Y)}=:Y_n$ is a coarse equivalence.
\end{example}

\begin{proposition}[{\cite[Prop. 6.4.7]{Higson2000Khomology}}]
	\label{prop: localized roe alg isomorphic to roe alg of subset}
	Let $X$ be a metric space and $Y\subseteq X$ a closed subspace. Then there is a unique isomorphism
	\[
	K_p(C^*(Y)) \xrightarrow{\sim} K_p(C^*(Y\subseteq X)) 
	\]
	such that for any $n \in \N$ the following diagram of isomorphisms commutes
	\begin{equation}
		\label{diag: commuting diagram localized iso}
		\begin{tikzcd}
			& {K_p(C^*(Y_n))} \\
			{K_p(C^*(Y))} && {K_p(C^*(Y\subseteq X))}
			\arrow["\sim", sloped, from=1-2, to=2-1]
			\arrow["\sim", sloped, from=1-2, to=2-3]
			\arrow["\sim", from=2-1, to=2-3]
		\end{tikzcd}
	\end{equation}
\end{proposition}



%

\subsection{Suspension in \texorpdfstring{$K$}{TEX}-homology}

In the following, let $A$ be a separable $C^*$-algebra. We denote the analytic $K$-homology \cite{Higson2000Khomology} of $A$ in degree $-p$ by $K^{-p}(A)$. (This involves the notion of $p$-multigradings, see Definition \ref{def: multigrading}.)

We will make use of the suspension isomorphism in $K$-homology. 
%
To construct the suspension isomorphism, a few remarks need to be made. First, it can be shown that if $A$ is a $C^*$-algebra and $[x] \in K^{-p}(A)$, then we can always choose to represent $[x]$ by a Fredholm module $x = (H, \rho, F)$ for which the representation $\rho$ is nondegenerate \cite[Lem. 8.3.8]{Higson2000Khomology}. Moreover, recall that if $\rho\colon A \rightarrow \cB(H)$ is a nondegenerate representation of $A$ on $H$ and $B$ is any $C^*$-algebra containing $A$ as ideal, then $\rho$ extends to a unique representation of $B$ on $H$ \cite[p. 41]{Higson2000Khomology}. Finally, any tensor product of $C^*$-algebras that will be used will be the spatial tensor product (but usually, if not always, at least one of the tensored $C^*$-algebras is Abelian and hence nuclear anyway). 

Let $(H, \rho, F)$ be a nondegenerate $(p+1)$-multigraded Fredholm module over the $C^*$-algebra $C_0(-1, 1)\otimes A$. Denote by $\rho'$ the unique extension of $\rho$ to $C([-1, 1])\otimes \tilde{A}$ and let $X_0 = \rho'(\id\otimes 1) \in \cB(H)$. Define 
\begin{equation}
	\label{eq: suspension operator}
	X = \gamma\varepsilon_1X_0,
\end{equation}
where $\gamma$ denotes the grading operator and $\varepsilon_1$ the first multigrading operator. Also note that via the extension of $\rho$ to $C([-1, 1])\otimes \tilde{A}$, the Hilbert space $H$ can be equipped with a representation $\rho_A$ of $A$ via $\rho_A(a) = \rho'(1\otimes a)$.

It can be shown that the triple $(H, \rho_A, V)$, where 
\begin{equation}
	\label{eq: schrodinger operator}
	V = X + (1 - X^2)^{\frac{1}{2}}F,
\end{equation}
with the multigrading operators $(\varepsilon_2, \dots, \varepsilon_{p+1})$ is a $p$-multigraded Fredholm module over $A$ \cite[p. 255]{Higson2000Khomology}. Moreover, the induced class $[H, \rho_A, V] \in K^{-p}(A)$ does not depend on the chosen nondegenerate representative \cite[p. 256]{Higson2000Khomology}.

\begin{definition}
	\label{def: suspension}
	Let $A$ be a $C^*$-algebra and let $x=(H, \rho, F)$ be a nondegenerate Fredholm module over $C_0(-1, 1)\otimes A$. The \emph{suspension} of $[x]$ is defined as $s([x]) = [H, \rho_A, V]$, where $\rho_A$ is the induced representation of $A$ on $H$ and $V$ is the operator in Equation \eqref{eq: schrodinger operator}.
\end{definition}

\begin{proposition}[{\cite[Thm. 9.5.2]{Higson2000Khomology}}]
	\label{prop: suspension is isomorphism}
	The resulting map 
	\[
	s\colon K^{-p-1}(C_0(-1, 1)\otimes A) \rightarrow K^{-p}(A)
	\]
	is an isomorphism.
\end{proposition}

\begin{example}[The Dirac class {\cite[Def. 9.5.1]{Higson2000Khomology}}]
	\label{example: Dirac class}
	Consider the Hilbert space $H= L^2([-1, 1])$ with orthonormal basis $e_n(x) = 2^{-\frac{1}{2}}e^{\pi i nx}$, and define the bounded operator $Y\in \cB(H)$ as:
	\begin{equation}
		Ye_n = \begin{cases}
			e_n &\text{ if $n\geq 0$,}\\
			-e_n &\text{ if $n<0$}.
		\end{cases}
	\end{equation}
	The Hilbert space $H\oplus H$ is 1-multigraded by the grading en multigrading operators
	\begin{equation*}
		\gamma = \begin{pmatrix}
			1&0\\
			0&-1
		\end{pmatrix}
		\quad\text{ and }\quad
		\varepsilon = \begin{pmatrix}
			0&-1\\
			1&0
		\end{pmatrix},
	\end{equation*}
	for which the natural representation $\rho$ of $C_0(-1, 1)$ on $H\oplus H$ by multiplication operators is even multigraded. On $H\oplus H$ consider the operator 
	\[
	F = \begin{pmatrix}
		0&-iY\\
		iY&0
	\end{pmatrix},
	\]
	which is obviously odd multigraded and satisfies $F^* = F$ and $F^2 = 1$ (as $Y^2 = 1$). Finally, if $f \in C_0(-1, 1)$, then $[F, \rho(f)]$ is compact, which follows by uniformly approximating $f$ by trigonometric polynomials using the Weierstrass approximation theorem.
		
	It follows that 
	\[
	\left(H\oplus H, \rho, F\right)
	\]
	is a 1-multigraded Fredholm module over $C_0(-1, 1)$ which therefore defines an element $d \in K^{-1}(C_0(-1,1))$.
\end{example}

\begin{definition}
	\label{def: Dirac class}
	The element $d \in K^{-1}(C_0(-1,1))$ constructed in Example \ref{example: Dirac class} is the \emph{Dirac class}.
\end{definition}

The suspension map for $A = \C$ under the identification $C_0(-1,1)\otimes \C \cong C_0(-1,1)$ gives an isomorphism $s\colon K^{-1}(C_0(-1,1)) \rightarrow K^0(\C) \cong \Z$.

\begin{lemma}
	\label{lem: suspension maps Dirac class to -1}
	Let $s$ be the suspension isomorphism for $\C$ and let $d$ be the Dirac class. Under the identification $K^0(\C)\cong \Z$, it follows that 
	\begin{equation}
		\label{eq: suspension sends dirac to -1}
		s(d) = -1.
	\end{equation}
\end{lemma}
\begin{proof}
In \cite[Lem. 9.5.7]{Higson2000Khomology}, it is argued that $s(d) = 1$.  There is  a sign error in this computation, however. The operator denoted by $W_1(X_g, Y)$ in the proof of  \cite[Lem. 9.3.6]{Higson2000Khomology}
	is not the same as the operator $W_1(X, Y)$ in  \cite[Lem. 8.6.12]{Higson2000Khomology}, but it is a compact perturbation of its adjoint, which changes the index by a minus sign. With this change, the proof of  \cite[Lem. 9.5.7]{Higson2000Khomology} shows that $s(d) = -1$. 
\end{proof}

\subsection{Paschke duality and the coarse index}

\begin{definition}
	\label{def: K homology of a space}
	Let $X$ be a separable metric space with the Heine--Borel property.  The \emph{$K$-homology} of $X$ is defined as $K_p(X) = K^{-p}(C_0(X))$.
\end{definition}

If $X$ is a separable metric space with the Heine--Borel property, then the $K$-homology of $X$ is isomorphic to the $K$-theory of the quotient algebra $Q^*(X):=D^*(X)/C^*(X)$. This is known as \emph{Paschke duality} after Paschke who first proved a related result \cite{Paschke1981Ktheoryofcommutatants}.

\begin{proposition}
	\label{prop: Paschke duality}
	Let $X$ be a separable metric space with the Heine--Borel property and let $\rho\colon C_0(X)\rightarrow \cB(H)$ be an ample representation on a Hilbert space $H$ on which we form the $C^*$-algebras $C^*(X)$ and $D^*(X)$. Then for $p = -1, 0$, there is an isomorphism $K_p(X)\cong K_{1+p}(D^*(X)/C^*(X))$ satisfying the following identities:
	\begin{enumerate}
		\item \emph{For $p=-1$}: If $T \in M_m(D^*(X))$ is such that $q(T) \in M_m(D^*(X)/C^*(X))$ is a projection, then $[H^{\oplus m}, \rho^{\oplus m}, 2T-1] \in K_{-1}(X)$ is the $K$-homology class corresponding to $[q(T)]_0 \in K_0(D^*(X)/C^*(X))$ under Paschke duality.
		\item \emph{For $p=0$}:	If $T \in M_m( D^*(X))$ is such that $q(T) \in M_m(D^*(X)/C^*(X))$ is a unitary, then $\left[H^{\oplus 2m}, \rho^{\oplus 2m}, 
		\begin{pmatrix}
			0&T^*\\
			T&0
		\end{pmatrix}
		\right] \in K_{0}(X)$ is the $K$-homology class corresponding to $[q(T)]_1 \in K_1(D^*(X)/C^*(X))$ under Paschke duality.
	\end{enumerate}
\end{proposition}
\begin{proof}
	A proof of the fact that the Paschke duality map is an isomorphism follows from combining the following results in \cite{Higson2000Khomology}: Exercise 4.10.9; Remark 5.4.2 and 5.4.4; Theorem 8.4.3 and Lemma 12.3.2 (part $b)$ is precisely Paschke duality).
\end{proof}

\begin{remark}
	\label{rem: Paschke duality}
	It should be noted that in general there are $4$ ``different'' Paschke duality maps employing the isomorphism. Of course the isomorphisms are different for $p = -1$ and $p = 0$, but also for compact and noncompact $X$. The reason for this is that for compact $X$ the $C^*$-algebra $C_0(X)$ is already unital such that an ample representation of $C_0(X)$ does not extend to an ample representation of $\widetilde{C_0(X)}$ (it is not even injective).
	
	In general, any of the $4$ Paschke duality maps can be decomposed into the following maps:
	\begin{equation}\label{eq: Paschke alpha j}
    \begin{tikzcd}
		{K_{1+p}(\PL(\widetilde{C_0(X)}))} & {K_{1+p}(\PL(X)/\LC(X))} \\
		{K_p(X)} & {K_{1+p}(D^*(X)/C^*(X))}
		\arrow["{\alpha_2}", from=1-1, to=1-2]
		\arrow["{\alpha_1}"', from=1-1, to=2-1]
		\arrow[dotted, from=2-1, to=2-2]
		\arrow["{\alpha_3}"', from=2-2, to=1-2]
	\end{tikzcd}
    \end{equation}
	
	In the above diagram, all maps $\alpha_i$ are isomorphisms and $\PL(\widetilde{C_0(X)})$ are the pseudolocal operators with respect to an ample representation of the unitization $\widetilde{C_0(X)}$. If $X$ is noncompact, then the unique extension of $\rho$ can be used as an ample representation. However, if $X$ is compact the unique extension of the representation $\rho' = \rho\oplus 0\colon C_0(X)\rightarrow \cB(H\oplus H)$ is used \cite[p. 128]{Higson2000Khomology}. With this in place, the map $\alpha_1$ is the map from \cite[Thm. 8.4.3]{Higson2000Khomology} and $\alpha_3$ is the map on $K$-theory induced by the natural map on these quotients that arises from the inclusion $D^*(X)\hookrightarrow \PL(X)$ \cite[Lem. 12.3.2]{Higson2000Khomology}.
	
	The map $\alpha_2$ depends on the compactness of $X$. If $X$ is noncompact, then it follows that $\PL(\widetilde{C_0(X)}) = \PL(X)$ (with respect to the just specified representation) and $\alpha_2$ is just the map induced by the quotient map. If $X$ is compact then $\PL(\widetilde{C_0(X)}) \subseteq \cB(H\oplus H)$ and $\alpha_2$ is the map induced by the homomorphism 
	\[
	\begin{pmatrix}
		a&*\\
		*&*
	\end{pmatrix}\mapsto q(a),
	\] 
	where $q\colon \PL(X)\rightarrow \PL(X)/\LC(X)$ is the quotient map (this is well-defined as $\LC(X) = \cK(H)$ is this case and the anti-diagonal terms of elements in $\PL(\widetilde{C_0(X)})$ are necessarily compact \cite[p. 128]{Higson2000Khomology}).
\end{remark}

\begin{definition}
	\label{def: assembly map}
	Let $X$ be a separable metric space with the Heine--Borel property. The \emph{coarse assembly maps} associated to $X$ are the maps 
	\[
	A_p\colon K_p(X)\rightarrow K_p(C^*(X)) 
	\]
	defined as the composition 
	\[\begin{tikzcd}
		{K_p(X)} & {K_{1+p}(D^*(X)/C^*(X))} & {K_p(C^*(X))}
		\arrow["\sim", from=1-1, to=1-2]
		\arrow["\partial", from=1-2, to=1-3]
	\end{tikzcd}\]
	where the first map is Paschke duality and  $\partial$ is the boundary map from the 6-term exact associated to the inclusion of the ideal $C^*(X)\hookrightarrow D^*(X)$. The $K$-theory indices should be interpreted as elements of $\Z/2\Z$---that is, for $p= 0, -1$, the assembly map $A_p$ takes the following forms, respectively:
	\[
	A_0\colon K_0(X)\rightarrow K_0(C^*(X)) \qquad \text{ and } \qquad A_{-1}\colon K_{-1}(X)\rightarrow K_1(C^*(X)).
	\]
\end{definition}
Recall that we assumed that $S$ was graded when $M$ is even-dimensional and that no grading on $S$ was assumed when $M$ is odd-dimensional. Let $\chi\colon \R\rightarrow [-1, 1]$ be a \emph{normalizing function}. That is, $\chi$ is smooth and odd, $\chi(x)>0$ for $x >0$ and $\lim_{x\rightarrow \pm\infty}\chi(x) = \pm 1$. Recall that if $D$ is a symmetric, elliptic differential operator on $S$ with finite propagation speed, then $D$ defines a $K$-homology class $[D]_p \in K_p(M)$, where $p = 0$ when $S$ is graded and $p = -1$ otherwise. This $K$-homology class is given by 
\begin{equation}
	\label{eq: Khom class D}
	[D]_p = \left[L^2(M;S), \rho, \chi(D)\right],
\end{equation}
which indeed defines a $K$-homology class by Theorems \ref{thm: func calc thm for C^*M} and \ref{thm: func calc thm for D^*M}, see also \cite[Thm. 10.6.5]{Higson2000Khomology}.

\begin{definition}
	\label{def: index D}
	Let $M$ be a complete Riemannian manifold and let $S$ be a Hermitian vector bundle over $M$. For any symmetric elliptic differential operator $D$ on $S$ with finite propagation speed, we define its \emph{coarse index} to be the element $A_p[D]_{p}\in K_p(C^*(M))$, where $p = -1$ if $M$ is odd-dimensional and $p = 0$ if $M$ is even-dimensional.
\end{definition}
\begin{remark}
 Later (see Proposition \ref{prop: Roe homomorphism}), it will be shown that for odd-dimensional manifolds the coarse index defined above is nothing else than the $K_1$-class determined by the Cayley transform of the differential operator $D$.
\end{remark}

%
%
%

\section{Partitioned manifolds, the generalized partitioned manifold index theorem and its consequences}
\label{section: partitioned mfds and PMT}

In this section, we will first introduce partitioned manifolds and the partitioned index that we can associate to an elliptic, symmetric differential operator with finite propagation speed on such partitioned manifold. In Section \ref{subsec: Roe homomorphism}, we will introduce the Roe homomorphism, which relates the partitioned index of an operator to its original index. Then in Section \ref{subsec: generalized pmt} our generalized version of the partitioned manifold index theorem will be stated. We will conclude in Section \ref{subsec: corollaries to pmt} by giving some immediate corollaries to our version of the partitioned manifold index theorem. This section is based on \cite{higson1991cobordisminvariance} and \cite[Ch. 4]{Roe1996indextheory}, but most of the concepts and objects are adapted to be compatible with our new setting, which relies more heavily on $K$-theory to allow for noncompact hypersurfaces.

\subsection{Partitioned manifolds}
\label{subsec: partitioned manifolds}
In this subsection, we will discuss the concept of a partitioned manifold and some of its properties that we will need. The following definition is a variation of \cite[Def. 4.2]{Roe1996indextheory} as we \emph{do not} assume compactness of the hypersurface.

\begin{definition}
	\label{def: partitioning hypersurface and partition}
	Let $M$ be a manifold. An embedded submanifold $N \subseteq M$ of codimension $1$ is said to be a \emph{partitioning hypersurface} if we can write $M = M_+\cup M_-$, where $M_{\pm} \subseteq M$ are submanifolds with boundary such that $N = M_+\cap M_- = \partial M_- = \partial M_+$. The triple $(N, M_+, M_-)$ is the \emph{partition} of $M$, i.e.\ this is a partitioning hypersurface with a corresponding choice of labeling $M_+$ and $M_-$. 
\end{definition}
 
The following lemma gives some easy topological consequences of the definition of a partitioned manifold. 

\begin{lemma}
	\label{lem: topological properties of partitions}
	Let $(N, M_+, M_-)$ be a partition of the manifold $M$. Then the following are true:
	\begin{enumerate}
		\item $N , M_{\pm} \subseteq M$ are topologically closed.
		\item $M_{\pm} = N\cup M_{\pm}^{\circ}$ where the union is disjoint.
		\item $M = M_-^{\circ}\cup N\cup M_+^{\circ}$ where the union is disjoint.
		\item $M\setminus N = M_-^{\circ} \cup M_+^{\circ}$ is disconnected.
	\end{enumerate}
\end{lemma}
\begin{proof}
	By \cite[Prop. 5.46]{Lee2012Smooth} it follows that the manifold interior and -boundary of $M_{\pm}$ agree with the topological interior and -boundary as subset of $M$. It follows that $N$ is closed and $M_{\pm}^{\circ}$ are open. This also implies that $N\cap M_{\pm}^{\circ} = \emptyset$. The other statements are easy consequences.
\end{proof}

	Let $(N, M_+, M_-)$ be a partition of the Riemannian manifold $(M, g)$. Then there exists a unit normal vector field $n$ on $N$ pointing into $M_+$. Moreover, if $M$ is oriented, then $N$ is orientable and $N$ has a unique orientation such that for any oriented orthonormal frame $(e_1, \ldots, e_d)$ of $TN$ the frame $(n, e_1, \ldots, e_d)$ is an oriented orthonormal frame for $\restr{TM}{N}$ with $\dim(M) = d+1$.
%

If $N$ is a connected partitioning hypersurface of a manifold $M$ with Riemannian metric $g$, then there are two ways to give $N$ the structure of a metric space. In the first place, $N$ can be equipped with the distance function $\restr{d_M}{N}$ it inherits as subset of $M$. Alternatively, $N$ can also be equipped with the Riemannian distance $d_{N}$ that is induced by the metric $g|_N$.

\begin{lemma}
	\label{lem: metric props hypersurfaces}
	Let $N \subseteq M$ be a topologically closed, connected hypersurface in a complete manifold $M$. Then $N$ is a complete Riemannian manifold. Moreover, the identity map 
	\[\id_N\colon (N, d_{N})\rightarrow (N, \restr{d_M}{N})\]
	is a uniform map. 
\end{lemma}
\begin{proof}
	As any curve through $N$ is automatically a curve through $M$. It follows that $\id_N$ is a Lipschitz continuous map and therefore it is uniformly expansive. The rest of the result now immediately follows from the fact that $d_N$ and $\restr{d_M}{N}$ induce the same topology on $N$.
%
%
\end{proof}

Finally, let us introduce the following notion of two hypersurfaces or partitions being ``near'' each other. This is a variation of the notion of cobordism that is introduced in \cite[p. 440]{higson1991cobordisminvariance} and \cite[p. 30]{Roe1996indextheory} that is adapted to our new setting. Recall that if $A\subseteq C$ and $B\subseteq C$ are two subsets of some set $C$, then we denote their \emph{symmetric difference} by 
\[
A\Delta B = (A\cup B)\setminus (A\cap B).
\]

\begin{definition}
	\label{def: near hypersurfaces}
	Let $(N, M_+, M_-)$ and $(N', M'_+, M'_-)$ be two partitions of $M$.
	\begin{enumerate} 
		\item We say that $N'$ is \emph{near} $N$ if there exists an $R>0$ such that $N'\subseteq B_R(N)$.
		\item We say that $(N', M'_+, M'_-)$ is \emph{near} $(N, M_+, M_-)$ if there exist an $R>0$ such that $M_{\pm}\Delta M_{\pm}'\subseteq B_R(N)$. 
	\end{enumerate}
\end{definition}

\begin{example}
	\label{ex: compact hypersurface is near all other hypersurfaces}
	Let $N, N'\subseteq M$ be two hypersurfaces and suppose that $N$ is compact. Then $N$ is near $N'$. Indeed, let $p\in N$ be arbitrary and choose $R>0$ such that $p \in B_R(N')$. It follows that $N\subseteq B_{R+\diam(N)}(N')$. Thus a compact hypersurface is near any other hypersurface.
\end{example}

\begin{remark}
	\label{rem: nearness is not equivalence relation}
	The relations of ``being near'' on hypersurfaces in and partitions of $M$ in Definition \ref{def: near hypersurfaces} are not equivalence relations. They are obviously reflexive and it can also be shown that they are transitive. For hypersurfaces, this follows easily; if $N\subseteq B_{R}(N')$ and $N'\subseteq B_{R'}(N'')$, then $N \subseteq B_{R}(B_{R'}(N'')) = B_{R+R'}(N'')$. For partitions this is a little more work. Let $(N, M_+, M_-)$, $(N', M'_+, M'_-)$ and $(N'', M''_+, M''_-)$ be partitions such that $M_{\pm}\Delta M'_{\pm} \subseteq B_R(N')$ and
	$M'_{\pm}\Delta M''_{\pm} \subseteq B_{R'}(N'')$. Then it follows that
	\begin{align*}
		M_{\pm}\Delta M''_{\pm} &\subseteq (M_{\pm}\Delta M'_{\pm}) \cup (M'_{\pm}\Delta M''_{\pm}) \subseteq B_R(N') \cup B_{R'}(N'')\\
		&\subseteq B_R(B_{R'}(N'')) \cup B_{R'}(N'') = B_{R+R'}(N''),
	\end{align*}
	where the third inclusion follows from Lemma \ref{lem: near partitions have near hypersurfaces}. We conclude that $(N, M_+, M_-)$ is near $(N'', M''_+, M''_-)$.
	
	However, the nearness relations are not symmetric. Consider $\R^2$ as Riemannian manifold with the Euclidean metric. The partitions of $\R^2$ given by 
	\[
	N = \{1\}\times \R \qquad \text{and} \qquad N' = \{(x, y)\in \R^2\colon y = \mp\frac{1}{x}\pm 1\}
	\]
	form a counterexample. 
\end{remark}

We write $\chi_A$ for the indicator function of a set $A$.
\begin{lemma}
	\label{lem: near partitions have near hypersurfaces}
	Let $(N, M_+, M_-)$ and  $(N', M_+', M_-')$ be two partitions of $M$ such that  $(N', M_+', M_-')$ is near $(N, M_+, M_-)$. Then $N'$ is near $N$ and $\chi_{M_+}-\chi_{M_+'}$ is supported near $N$. In particular, if $R>0$ is such that $M_{\pm}\Delta M_{\pm}'\subseteq B_R(N)$, then $N' \subseteq B_R(N)$ and $\supp(\chi_{M_{\pm}}-\chi_{M_{\pm}'}) \subseteq \ol{B_{R}(N)}$.
\end{lemma}
\begin{proof}
	Take $R>0$ such that $M_{\pm}\Delta M_{\pm}'\subseteq B_R(N)$ and let $p \in N' = M_+'\cap M_-'$ be arbitrary. Suppose that $p \notin M_+\Delta M_+'$ and $p \notin M_-\Delta M_-'$, then 
	\[
	p \in M_+\cap M_+'\cap M_-\cap M_-' = N\cap N' \subseteq B_R(N).
	\] 
	If either $p \in M_+\Delta M_+'$ or $p \in M_-\Delta M_-'$, then $p \in B_R(N)$ trivially.
	
	For the last statement, observe that 
	\begin{equation*}
		|\chi_{M_{\pm}}-\chi_{M_{\pm}'}|(p) = 
		\begin{cases}
			1 &\text{ if $p \in M_{\pm}\setminus M_{\pm}' \cup M_{\pm}'\setminus M_{\pm}$,}\\
			0 &\text{ if $p \notin M_{\pm}\setminus M_{\pm}' \cup M_{\pm}'\setminus M_{\pm}$,}
		\end{cases}
	\end{equation*}
	which is equal to $\chi_{M_{\pm}\Delta M_{\pm}'}(p)$. So $\supp(\chi_{M_{\pm}}-\chi_{M_{\pm}'}) \subseteq \ol{M_{\pm}\Delta M_{\pm}'} \subseteq \ol{B_R(N)}$.
\end{proof}

\subsection{The partitioned index}
\label{subsec: the partitioned index}
In his proof of Roe's partitioned manifold index theorem, Higson constructs a ``partitioned index'' $\Ind(D;N)$, where $N\subseteq M$ is a compact partitioning hypersurface \cite[p. 440]{higson1991cobordisminvariance}. In this subsection, our goal is to show that this partitioned index can be constructed in a similar way when the partitioning hypersurface is possibly noncompact. Again, $M$ will always be a complete Riemannian manifold with a Hermitian vector bundle $S$, and $D$ will be an elliptic, symmetric operator on $S$ that has finite propagation speed, such that $D$ is an essentially self-adjoint operator on $L^2(M;S)$. We will denote the unique self-adjoint extension also by $D$.
%

Suppose that we have a fixed partitioning hypersurface $N \subseteq M$. We write $Q^*(N\subseteq M):= D^*(M)/ C^*(N\subseteq M)$. If $S, T \in D^*(M)$, we write $S \sim T $ if $S-T \in C^*(N\subseteq M)$, i.e.\ when $S$ and $T$ represent the same equivalence class in $Q^*(N\subseteq M)$.  The following result is based on \cite[Thm. 1.2]{higson1991cobordisminvariance}.

\begin{lemma}
	\label{lem: Higsons 1.2 general}
	Let $\theta \in B^{\infty}(M)$ be such that $\supp(\theta) \subseteq B_R(N)$ for some $R>0$. Then $\rho(\theta)T \sim 0$ for any $T \in C^*(M)$.
\end{lemma}
\begin{proof}
	By definition of $C^*(M)$, there exists a sequence $\{T_n\}_{n\in\N}$ such that $T_n \rightarrow T$ in norm and $T_n \in \LC(M)\cap\FP(M)$. 
We have	 $\rho(\theta)T_n \in \LC(M)\cap\FP(M)$ by Proposition \ref{prop: property induced sets are *algebras}. For each $n \in \N$, choose $P_n>0$ such that for any $g, h \in B^{\infty}(M)$ with $d_M(\supp(g), \supp(h))>P_n$ the operator $\rho(g)T_n\rho(h)$ equals 0. If $f \in B^{\infty}(M)$ is such that $d_M(\supp(f), N)> R+P_n+1$, then $d_M(\supp(f), \supp(\theta))>P_n$ and therefore
	\[
	\rho(f)\rho(\theta)T_n = \rho(\theta)T_n\rho(f) = 0, 
	\]
	which implies that $\rho(\theta)T_n \in  \LC(M)\cap\FP(M)\cap\SN_N(M)$. Since $\rho(\theta)T_n \rightarrow \rho(\theta)T$, it follows that $\rho(\theta)T \in C^*(N\subseteq M)$.
\end{proof}

\begin{corollary}
	\label{cor: Higsons 1.2}
	Let $\theta \in B^{\infty}(M)$ be such that $\supp(\theta) \subseteq B_R(N)$ for some $R>0$. Then $\rho(\theta)(D\pm i)^{-1} \sim 0$.
\end{corollary}
\begin{proof}
	 Since $\varphi_{\pm}(x) = (x\pm i)^{-1} \in C_0(\R)$, it follows from Theorem \ref{thm: func calc thm for C^*M} that $(D\pm i)^{-1} \in C^*(M)$. Therefore, the result follows directly from Lemma \ref{lem: Higsons 1.2 general}
\end{proof}

We also have the following generalization of \cite[Lem. 1.3]{higson1991cobordisminvariance}.

\begin{lemma}
	\label{lem: Higsons 1.3}
	Let $\phi \in C^{\infty}(M)$ be such that $\phi$ and $\|d\phi\|$ are uniformly bounded and suppose there exists an $R>0$ such that $\phi$ is locally constant on $M\setminus B_R(N)$. Then $[(D\pm i)^{-1}, \rho(\phi)] \sim 0$.
\end{lemma}
\begin{proof}
	Note that $\rho(\phi)(\dom(D^*))\subseteq \dom(D^*)$. Indeed, take $u \in \dom(D^*)$ so that the functional $s \mapsto \langle Ds, u\rangle$ is bounded on $C_c^{\infty}(M;S)$. Then it follows that
	\begin{align*}
		|\langle Ds, \phi u\rangle| &= |\langle \phi Ds, u\rangle| \leq |\langle D(\phi s), u \rangle| + |\langle [D, \rho(\phi)]s, u\rangle| \\
		&\leq \|Du\|\|\phi\|_{\infty}\|s\| + \|u\|\|[D, \rho(\phi)]\|\|s\| = C\|s\|,
	\end{align*}
	where $\|[D, \rho(\phi)]\| < \infty$ as $D$ has finite propagation speed and we assumed $\|d\phi\|$ to be bounded on $B_R(N)$ (note that $[D, \rho(\phi)]u(p) = \sigma_D(p, d\phi_p)u(p)$). From this, we conclude that the operator $[D, \rho(\phi)]$ is well-defined on $\dom(D^*)$. Thus, it follows that
	\begin{align*}
		[(D\pm i)^{-1}, \rho(\phi)] &= (D\pm i)^{-1}\rho(\phi) -\rho(\phi)(D\pm i)^{-1} \\
		&=(D\pm i)^{-1}\left[\rho(\phi) -(D\pm i)\rho(\phi)(D\pm i)^{-1}\right] \\
		&=(D\pm i)^{-1}\left[\rho(\phi)(D\pm i) -(D\pm i)\rho(\phi)\right](D\pm i)^{-1} \\
		&=(D\pm i)^{-1}[\rho(\phi),D](D\pm i)^{-1}, 
	\end{align*}
	where for the second equality we used that $\rho(\phi)(\dom(D^*))\subseteq \dom(D^*)$ to conclude that 
	\[
	(D\pm i)^{-1}(D\pm i)\rho(\phi)(D\pm i)^{-1} =\rho(\phi)(D\pm i)^{-1}.
	\]
	
	Since $[D, \rho(\phi)] = \sigma_D(d\phi)$ is a vector bundle morphism of $S$ (it is a smooth section of $\End(S)$) such that $[D, \rho(\phi)](p) = \sigma_D(p, d\phi_p) = 0$ on $M\setminus B_R(N)$, it follows that $[D, \rho(\phi)] \in \SN_N(M)$. Suppose that $\{\psi, 1-\psi\}$ is a partition of unity subordinate to the open cover $\{B_{2R}(N), M\setminus\ol{B_R(N)}\}$ of $M$. Then $\rho(\psi)[D, \rho(\phi)] = [D, \rho(\phi)]$ and thus $ [(D\pm i)^{-1}, \rho(\phi)]=(D\pm i)^{-1}\rho(\psi)[\rho(\phi),D](D\pm i)^{-1}$. By Corollary \ref{cor: Higsons 1.2}, we conclude that $(D\pm i)^{-1}\rho(\psi) \in C^*(N\subseteq M)$. Moreover, by Theorem \ref{thm: func calc thm for C^*M} and the fact that $[D, \rho(\phi)] \in D^*(M)$, it follows that $[D, \rho(\phi)](D\pm i)^{-1} \in C^*(M)$. Since $C^*(N\subseteq M) \subseteq C^*(M)$ is an ideal, we conclude that $[(D\pm i)^{-1}, \rho(\phi)] \in  C^*(N\subseteq M)$.
\end{proof}

\begin{definition}
	\label{def: simple hypersurface}
	A partitioning hypersurface $N \subseteq M$ is \emph{simple} if there exists an $R>0$ and a $\phi_+ \in C^{\infty}(M)$ such that $\phi_+$ and $\|d\phi_+\|$ are uniformly bounded and such that $\restr{\phi_+}{M\setminus B_R(N)} = \restr{\chi_{M_+}}{M\setminus B_R(N)}$.
\end{definition}
\begin{remark}
Every compact partitioning hypersurface is simple; see Remark \ref{rem: hypersurfaces as in PMT are simple} and Corollary \ref{cor: PMT for compact hypersurface}. In fact, it may be true that every partitioning hypersurface is simple; we have not found any counterexamples. In any case, Remark \ref{rem: hypersurfaces as in PMT are simple} ensures that this has no influence on our main result; the additional assumptions we require to hold imply that the hypersurface is simple anyway.
\end{remark}

If $A$ is a unital $C^*$-algebra,  then we denote the $K$-theory class of a unitary $U \in M_{\infty}(A)$ by $[U]_1 \in K_1(A)$, and the class of a projection $P \in M_{\infty}(A)$ by $[P]_0 \in K_0(A)$. Denote by $\partial$ the index map induced by the inclusion of the ideal $C^*(N \subseteq M) \subset D^*(M)$. Then the following result is a generalization of \cite[Lem. 1.4]{higson1991cobordisminvariance}.

\begin{proposition}
	\label{prop: Higson's 1.4}
	Let $N$ be a simple partitioning hypersurface and let $\phi_+$ be as in Definition \ref{def: simple hypersurface}. Put $\phi_- = 1-\phi_+$ and let $U$ be the Cayley transform of $D$. Then:
	\begin{enumerate}
		\item The operators $U_{\pm}= \rho(\phi_{\mp}) + \rho(\phi_{\pm})U$ are elements of $D^*(M)$ which represent unitary equivalence classes in $Q^*(N\subseteq M)$.
		\item $\partial[q_N(U_+)]_1 = -\partial[q_N(U_-)]_1$.
		\item $\partial[q_N(U_+)]_1$ is independent of the choice of $\phi_+$.
	\end{enumerate}
\end{proposition}
\begin{proof}
	\begin{enumerate}
		\item It is obvious that $U_{\pm} \in D^*(M)$ as $U \in D^*(M)$ by Theorem \ref{thm: func calc thm for D^*M} and $\rho(\phi_{\pm})\in D^*(M)$. 
		Writing $U = 1 - 2i(D+i)^{-1}$ and using Lemma \ref{lem: Higsons 1.3}, we see that
		\begin{equation}
			\label{eq: equivalence 1}
			[\rho(\phi_{\pm}), U] = [\rho(\phi_{\pm}),  1 - 2i(D+i)^{-1}] = -2i[ \rho(\phi_{\pm}), (D+i)^{-1}] \sim 0.
		\end{equation}
		Moreover, it follows from Corollary \ref{cor: Higsons 1.2} that
		\begin{align}
			\label{eq: equivalence 2}
			\rho(\phi_+)\rho(\phi_-)(U - 1) &= \rho(\phi_+)\rho(\phi_-)(-2i(D+i)^{-1})\nonumber\\
			 &= -2i\rho(\phi_+\phi_-)(D+i)^{-1} \sim 0.
		\end{align}
		Using the equivalences in \eqref{eq: equivalence 1} and \eqref{eq: equivalence 2}, we see that
		\begin{align*}
			U_{\pm}U^*_{\pm} &= (\rho(\phi_{\mp}) + \rho(\phi_{\pm})U)(\rho(\phi_{\mp}) + \rho(\phi_{\pm})U)^* \\
			&= (\rho(\phi_{\mp}) + \rho(\phi_{\pm})U)(\rho(\phi_{\mp}) + U^*\rho(\phi_{\pm})) \\
			&= \rho(\phi_{\mp}^2) +\rho(\phi_{\pm}^2) + \rho(\phi_{\pm})U\rho(\phi_{\mp}) + \rho(\phi_{\mp})U^*\rho(\phi_{\pm}) \\
			&\sim \rho(\phi_{\mp}^2) +\rho(\phi_{\pm}^2) + \rho(\phi_{\pm})\rho(\phi_{\mp})U + \rho(\phi_{\mp})\rho(\phi_{\pm})U^* \\
			&\sim \rho(\phi_{\mp}^2) +\rho(\phi_{\pm}^2) + \rho(\phi_{\pm})\rho(\phi_{\mp}) + \rho(\phi_{\mp})\rho(\phi_{\pm}) = 1.
		\end{align*}
		Similarly, we find that $U^*_{\pm}U_{\pm} \sim 1$. It follows that $q_N(U_{\pm}) \in Q^*(N\subseteq M)$ are indeed unitaries.
		\item Again using the equivalences in \eqref{eq: equivalence 1} and \eqref{eq: equivalence 2}, we see that 
		\begin{align*}
			U_+U_- &= (\rho(\phi_-)+\rho(\phi_+)U)(\rho(\phi_+)+\rho(\phi_-)U) \\
			&= \rho(\phi_-)\rho(\phi_+) + \rho(\phi_+)U\rho(\phi_+) + \rho(\phi_-)\rho(\phi_-)U + \rho(\phi_+)U\rho(\phi_-)U \\
			&\sim \rho(\phi_-)\rho(\phi_+)U +\rho(\phi_+)\rho(\phi_+)U + \rho(\phi_-)\rho(\phi_-)U + \rho(\phi_+)\rho(\phi_-)U = U
		\end{align*}
		It follows that $q_N(U_+)q_N(U_-) = q_N(U)$ and thus  
		\begin{align*}
			\partial([q_N(U_+)]_1) + \partial([q_N(U_-)]_1) &= \partial([q_N(U_+)]_1+[q_N(U_-)]_1) = \partial([q_N(U_+)q_N(U_-)]_1) \\
			&= \partial([q_N(U)]_1) = \partial\circ K_1(q_N)([U]_1) = 0,
		\end{align*}
		where the last equality follows since $\partial\circ K_1(q_N) = 0$. So indeed $\partial([q_N(U_+)]_1) = -\partial([q_N(U_-)]_1)$.
		\item Suppose $\phi_+'$ is a second function as in Definition \ref{def: simple hypersurface} and put $\phi_-' = 1 - \phi_+'$. This gives an operator $U'_+ = \rho(\phi_-')+\rho(\phi_+')U$. We claim that $U_+ \sim U'_+$. Indeed, since there exist $R> 0$ and $R' > 0$ such that $\restr{\phi_+}{M\setminus B_R(N)} = \restr{\chi_{M_+}}{M\setminus B_R(N)}$ and $\restr{\phi'_+}{M\setminus B_{R'}(N)} = \restr{\chi_{M_+}}{M\setminus B_{R'}(N)}$, it follows that $\phi_+$ and $\phi'_+$ are equal on $M\setminus B_{R_{\max}}(N)$ with $R_{\max} = \max(R, R')$. So $\supp(\phi_+ - \phi_+') \subseteq \ol{B_{R_{\max}}(N)} \subseteq B_{R_{\max}+1}(N)$. So, it follows from Corollary \ref{cor: Higsons 1.2} that
		\begin{align*}
			U_+ - U_+' &= -\rho(\phi_+ - \phi_+')+\rho(\phi_+ - \phi_+')U = \rho(\phi_+ - \phi_+')(U-1) \\
			&= -2i\rho(\phi_+ - \phi_+')(D+i)^{-1} \sim 0.
		\end{align*}
		We conclude that $\partial[q_N(U_+)]_1 = \partial[q_N(U_+')]_1$, which therefore does not depend on the choice of $\phi_+$. \qedhere
	\end{enumerate}
\end{proof}

Proposition \ref{prop: Higson's 1.4} allows us to give the following generalized definition of the partitioned index introduced in \cite[p. 440]{higson1991cobordisminvariance}.

\begin{definition}
	\label{def: partitioned index}
	Let $N$ be a simple partitioning hypersurface of $M$. We define the \emph{partitioned index} of $D$ to be the element 
	\[
	\Ind(D;N):= \partial[q_N(U_+)]_1 \in K_0(C^*(N \subseteq M)).
	\]
\end{definition}

Note that if $(N', M_+', M_-')$ is a second partition of $M$ that is near the original one, then by Lemma \ref{lem: near partitions have near hypersurfaces} it follows that $N'$ is near $N$. This implies that $C^*(N'\subseteq M)\subseteq C^*(N\subseteq M)$. Denote the resulting inclusion map by $i$. The following result shows that the partitioned index is well-behaved with respect to this way of comparing partitioning hypersurfaces.

\begin{lemma}\label{lem near part index}
	 Suppose that $N'$ is a second simple partitioning hypersurface of $M$ such that $(N', M_+', M_-')$ is near $(N, M_+, M_-)$. Then
	\begin{equation}
		\label{eq: naturality of partitioned index}
		\Ind(D;N) = K_0(i)( \Ind(D;N')).
	\end{equation}
\end{lemma}
\begin{proof}
	Choose a $\phi_+$ as in Definition \ref{def: simple hypersurface} that works for $N'$. Then this $\phi_+$ also works for $N$. 	
	This means both partitioned indices $\Ind(D;N)$ and $\Ind(D;N')$ can be computed using the same operator $U_+ = \rho(\phi_-) + \rho(\phi_+)U$. Let $\partial$ be the index map in the $6$-term exact sequence associated with $N$ and $\partial'$ the index map in the $6$-term exact sequence associated with $N'$. Then $\Ind(D;N) = \partial([q_N(U_+)]_1)$ and $\Ind(D;N') = \partial'([q_{N'}(U_+)]_1)$. 
	
	Note that the two short exact sequences associated with $N$ and $N'$ fit together in the following commutative diagram:
	\[\begin{tikzcd}
		0 & {C^*(N\subseteq M)} & {D^*(M)} & {Q^*(N\subseteq M)} & 0 \\
		0 & {C^*(N'\subseteq M)} & {D^*(M)} & {Q^*(N'\subseteq M)} & 0
		\arrow[from=1-1, to=1-2]
		\arrow["{i_N}", hook, from=1-2, to=1-3]
		\arrow["{q_N}", from=1-3, to=1-4]
		\arrow[from=1-4, to=1-5]
		\arrow[from=2-1, to=2-2]
		\arrow["i"', hook, from=2-2, to=1-2]
		\arrow["{i_{N'}}", hook, from=2-2, to=2-3]
		\arrow[equals, from=2-3, to=1-3]
		\arrow["{q_{N'}}", from=2-3, to=2-4]
		\arrow["{\exists ! \alpha_{N',N}}"', dotted, from=2-4, to=1-4]
		\arrow[from=2-4, to=2-5]
	\end{tikzcd}\] 
	Therefore, we conclude that
	\begin{align*}
		K_0(i)(\Ind(D; N')) &= K_0(i)\circ\partial'([q_{N'}(U_+)]_1) = \partial\circ K_1(\alpha_{N',N})([q_{N'}(U_+)]_1) \\
		&= \partial([\alpha_{N',N}\circ q_{N'}(U_+)]_1) = \partial([q_N(U_+)]_1) = \Ind(D;N),
	\end{align*}
	where the second equality follows by naturality of the index map.
\end{proof}
\begin{corollary}
	\label{cor: mutually near hypersurfaces induce same partitioned index}
	Let  $(N, M_+, M_-)$ and $(N', M_+', M_-')$ be mutually near simple partitions of $M$---that is, $(N, M_+, M_-)$ is near $(N', M_+', M_-')$ and $(N', M_+', M_-')$ is near $(N, M_+, M_-)$. Then $\Ind(D; N) = \Ind(D; N')$.
\end{corollary}

\subsection{The Roe homomorphism}
\label{subsec: Roe homomorphism}
A natural question to ask is what relation there is between the partitioned index $\Ind(D; N)$
and the index $\Ind(D)$ of the differential operator $D$. In the spirit of Roe's own formulation of the partitioned manifold index theorem in \cite[Thm. 4.4]{Roe1996indextheory}, we will show that there exists a map 
\[\varphi_N\colon K_1(C^*(M)) \rightarrow K_0(C^*(N\subseteq M)),\] 
such that $\varphi_N(\Ind(D)) = \Ind(D;N)$, see Proposition \ref{prop: Roe homomorphism}. In particular, $\Ind(D;N)=0$ if $\Ind(D) = 0$.

Let $P_+:= \rho(\chi_{M_+})$ be the projection in $\cB(L^2(M;S))$ obtained by multiplying with the characteristic function of $M_+$. Consider the following variation of Lemma \ref{lem: Higsons 1.3}, which is based on \cite[Lem. 4.3]{Roe1996indextheory}.

\begin{lemma}
	\label{lem: commutator with projection lands in localized algebra}
	Let $T \in C^*(M)$. Then $[T, P_+] \in C^*(N \subseteq M)$. 
\end{lemma}
\begin{proof}
	Suppose that $T \in \FP(M)\cap\LC(M)$. Since $P_+ \in \FP(M)$, it follows by Proposition \ref{prop: property induced sets are *algebras} that $[T, P_+] \in \FP(M)\cap\LC(M)$. We show that $[T, P_+] \in \SN_N(M)$. Choose $R>0$ such that for all $g, h \in C_0(M)$ with $d_M(\supp(g), \supp(h))>R$ the operator $\rho(g)T\rho(h)$ equals 0 and suppose that $f \in C_0(M)$ is such that $d_M(\supp(f), N) > 2R$. We claim that $[T, P_+]\rho(f) = 0$. For this, write $f_+ = \chi_{M_+}f$ and $f_- = f - f_+$. It follows that
	\begin{align*}
		[T, P_+]\rho(f) &= TP_+\rho(f) - P_+T\rho(f)= T\rho(f_+) - P_+T\rho(f_+) - P_+T\rho(f_-) \\
		&= \rho(1 - \chi_{M_+})T\rho(f_+) - \rho(\chi_{M_+})T\rho(f_-).
	\end{align*}
	%
	%
	Now $d_M(\supp(1 - \chi_{M_+}), \supp(f_+) ) \geq 2R>R$ 
	 and thus $\rho(1 - \chi_{M_+})T\rho(f_+) = 0$ by Remark \ref{rem: FP and SN properties extend to bounded Borel functions}. Similarly, $\rho(\chi_{M_+})T\rho(f_-) = 0$,
so $[T, P_+]\rho(f) = 0$. By an analogous argument, 
$
		\rho(f)[T, P_+] 
		= 0
$. 
	Thus $[T, P_+]$ is supported within a distance $2R$ from $N$ and therefore $[T, P_+] \in \SN_N(M)$. We conclude that $[T, P_+] \in \FP(M)\cap\LC(M)\cap\SN_N(M) \subseteq C^*(N\subseteq M)$. The general result now follows since $\FP(M)\cap\LC(M) \subseteq C^*(M)$ is dense. 
\end{proof}

\begin{lemma}
	\label{lem: construction gamma}
	Let $q\colon C^*(M) \rightarrow C^*(M)/C^*(N\subseteq M)$ denote the quotient map. Then the map $\gamma\colon C^*(M) \rightarrow C^*(M)/C^*(N\subseteq M)$ defined by $\gamma(T) = q(P_+TP_+) $ is a $*$-homomorphism.
\end{lemma}
\begin{proof}
	Note that $P_+TP_+ \in C^*(M)$ as $P_+ \in D^*(M)$
	and $C^*(M)\subseteq D^*(M)$ is an ideal. It follows that $q(P_+TP_+)\in C^*(M)/C^*(N\subseteq M)$ is a well-defined class in the quotient. An easy computation using Lemma \ref{lem: commutator with projection lands in localized algebra} directly shows that $\gamma$ is a homomorphism.
\end{proof}

To construct the map $\varphi_N$, consider the following extension of $C^*$-algebras:
\begin{equation}
	\label{eq: C* extension for Roe homomorphism}
	\begin{tikzcd}
		0 & {C^*(N\subseteq M)} & {C^*(M)} & {C^*(M)/C^*(N\subseteq M)} & 0.
		\arrow[from=1-1, to=1-2]
		\arrow[hook, from=1-2, to=1-3]
		\arrow["{q}", from=1-3, to=1-4]
		\arrow[from=1-4, to=1-5]
	\end{tikzcd}
\end{equation}
As a variation of the map constructed in \cite[p. 29]{Roe1996indextheory}, we give the following definition.

\begin{definition}
	\label{Def: Roe homomorphism}
	Let $(N, M_+, M_-)$ be a partition of the manifold $M$. The \emph{Roe homomorphism} associated to this partition is the map 
	\[
	\varphi_N\colon K_1(C^*(M))\xrightarrow{K_1(\gamma)} K_1(C^*(M)/C^*(N\subseteq M)) \xrightarrow{\partial} K_0(C^*(N\subseteq M)),
	\]
	where $\gamma$ is the $*$-homomorphism from Lemma \ref{lem: construction gamma} and $\partial$ is the index map associated to the extension in \eqref{eq: C* extension for Roe homomorphism}.
\end{definition}

\begin{proposition}
	\label{prop: Roe homomorphism}
	Let $(N, M_+, M_-)$ be a simple partition of the manifold $M$ and let $\varphi_N$ be its associated Roe homomorphism. Let $D$ be an elliptic symmetric differential operator with finite propagation speed and let $U$ be its Cayley transform. Let $\chi$ be any normalizing function. Then the following hold:
	\begin{enumerate}
		\item \label{item: ind as cayley trafo} $\Ind(D) = [-\exp(\pi i\chi(D))]_1 = [U]_1$.
		\item \label{item: partitioned index as image roe homom} $\varphi_N(\Ind(D)) = \Ind(D;N)$.
		\item \label{item: compatibilty Roe homom} For any other simple partitioning hypersurface $N'$ of $M$ such that $(N', M_+', M_-')$ is near $(N, M_+, M_-)$, the associated Roe homomorphisms are related via 
		\begin{equation}
			\label{eq: naturality of Roe homomorphism}
			\varphi_N = K_0(i)\circ\varphi_{N'},
		\end{equation}
		where $i$ denotes the inclusion map $i\colon C^*(N'\subseteq M)\hookrightarrow C^*(N \subseteq M)$.
	\end{enumerate}
\end{proposition}
\begin{proof}
	\begin{enumerate}
		\item By Definition \ref{def: index D} and Theorem \ref{prop: Paschke duality}, it follows that $\Ind(D) = \partial\left(\left[\frac{1 + q(\chi(D))}{2}\right]_0\right) \in K_1(C^*(M))$, where $\partial$ is the exponential map associated with the short exact sequence that arises from the inclusion $C^*(M) \hookrightarrow D^*(M)$. We can explicitly calculate the action of the exponential map, using \cite[Remark 4.9.3]{Higson2000Khomology}. Since the projection $\frac{1 + q(\chi(D))}{2} \in Q^*(M)$ lifts to the self-adjoint element $\frac{1 + \chi(D)}{2} \in D^*(M)$, it follows that
		\begin{align*}
			\Ind(D) &= \left[\exp\left(2\pi i\left(\frac{1 + \chi(D)}{2}\right)\right)\right]_1 = \left[-\exp\left(\pi i\chi(D)\right)\right]_1.
		\end{align*}
		Note that this makes sense, as $x\mapsto -\exp\left(\pi i\chi(x)\right)$ has limits equal to $1$ at $\pm\infty$, so that functional calculus applied to this function indeed gives a unitary in $\widetilde{C^*(M)}$ by Corollary \ref{cor: func calc thm for unitization C^*M}. By construction, the element $\Ind(D)$ does not depend on the chosen normalizing function $\chi$. Indeed, if $\chi'$ is a second normalizing function, then $\chi- \chi' \in C_0(\R)$ so that $\chi(D) - \chi(D') \in C^*(M)$  by Theorem \ref{thm: func calc thm for C^*M}. Therefore $q(\chi(D)) = q(\chi'(D))$.
		
		Consider the normalizing function $\theta(x) = \frac{2}{\pi}\arctan(x)$.
		Note that for nonzero $x$, we have $\arctan(\frac{1}{x}) = \pm\frac{\pi}{2}-\arctan(x)$, where the sign is equal to the sign of $x$. So for nonzero $x$, it follows that
		\begin{align}
			\label{eq: Cayley transform as exponential}
			-\exp(\pi i \theta(x)) &=  -\exp(2i\arctan(x)) = -\exp\left(2i\left(\pm\frac{\pi}{2}-\arctan\left(\frac{1}{x}\right)\right)\right) \notag\\
			&= \exp\left(-2i\arctan\left(\frac{1}{x}\right)\right) = \exp\left(i\arctan\left(-\frac{1}{x} \right)-i\arctan\left(\frac{1}{x}\right)\right)\notag\\
			&= \frac{x - i}{x + i},
		\end{align}
		where the last equality follows as the two terms in the exponential are precisely the arguments of the numerator and denominator. It is easy to see that Equation \eqref{eq: Cayley transform as exponential} also holds when $x = 0$ as both sides evaluate to $-1$. Therefore, choosing $\theta$ as normalizing function gives 
		\begin{equation}
			\label{eq: Ind(D) = [U]}
			\Ind(D) = \left[-\exp\left(\pi i\theta(D)\right)\right]_1 = \left[\frac{D - i}{D + i}\right]_1= [U]_1.
		\end{equation}
		\item By definition of $\varphi_N$ and by part \ref{item: ind as cayley trafo} of the theorem, it follows that $\varphi_N(\Ind(D)) = \partial\circ  K_1(\gamma)([U]_1)$. Note that since $\gamma$ is a $*$-homomorphism between nonunital $C^*$-algebras, its action on $K$-theory will be given by the action of the corresponding unitized homomorphism.
		Define $P_- = 1-P_+$, choose $\phi_+$ as in Definition \ref{def: simple hypersurface}
		and consider $U_+ = \rho(\phi_-) +\rho(\phi_+)U$. Note that $U-1 = -2i(D+i)^{-1} \in C^*(M)$ and therefore $[U, P_+] = [U-1, P_+] \in C^*(N\subseteq M)$ by Lemma \ref{lem: commutator with projection lands in localized algebra}. It follows that 
		\begin{align*}
			P_- + P_+UP_+ &\sim  P_- + P_+P_+U =  P_- + P_+U \\
			&= \rho(\phi_-) +\rho(\phi_+)U + (P_--\rho(\phi_-)) + (P_+-\rho(\phi_+))U \\
			&= U_+ + (P_+-\rho(\phi_+))(U-1) \sim U_+,
		\end{align*}
		where the last equivalence follows from Lemma \ref{lem: Higsons 1.2 general}.  So $q_N(P_- + P_+UP_+ ) = q_N(U_+)$. Consider the following commutative diagram:
		\[\begin{tikzcd}
			0 & {C^*(N\subseteq M)} & {D^*(M)} & {Q^*(N\subseteq M)} & 0 \\
			0 & {C^*(N\subseteq M)} & {C^*(M)} & {C^*(M)/C^*(N\subseteq M)} & 0
			\arrow[from=1-1, to=1-2]
			\arrow[hook, from=1-2, to=1-3]
			\arrow["{q_N}", from=1-3, to=1-4]
			\arrow[from=1-4, to=1-5]
			\arrow[from=2-1, to=2-2]
			\arrow[equals, from=2-2, to=1-2]
			\arrow[hook, from=2-2, to=2-3]
			\arrow[hook, from=2-3, to=1-3]
			\arrow["q", from=2-3, to=2-4]
			\arrow["{\exists ! \beta}"', dotted, from=2-4, to=1-4]
			\arrow[from=2-4, to=2-5]
		\end{tikzcd}\]
		Denote the index map associated to the upper short exact sequence by $\partial^D$ and the index map associated to the lower short exact sequence by $\partial^C$. Then by naturality of the index map and by writing $U = (U-1) + 1$ with $U-1 = -2i(D+i)^{-1} \in C^*(M)$, we conclude that
		\begin{align*}
			\label{eq: Roe map maps index to partitioned index}
			&\varphi_N(\Ind(D))\\ 
			&\qquad= \partial^C\circ K_1(\gamma)([U]_1) = \partial^C([q(P_+(U-1)P_+) + 1]_1) \\ 
			&\qquad=\partial^D\circ K_1(\beta)([q(P_+(U-1)P_+) + 1]_1) =\partial^D([q_N(P_+(U-1)P_+) + 1]_1)\\
			&\qquad=\partial^D([q_N(P_- + P_+UP_+)]_1)= \partial^D([q_N(U_+)]_1) = \Ind(D;N).
		\end{align*}
		\item Consider the following commutative diagram:
		\[\begin{tikzcd}
			0 & {C^*(N\subseteq M)} & {C^*(M)} & {C^*(M)/C^*(N\subseteq M)} & 0 \\
			0 & {C^*(N'\subseteq M)} & {C^*(M)} & {C^*(M)/C^*(N'\subseteq M)} & 0
			\arrow[from=1-1, to=1-2]
			\arrow[hook, from=1-2, to=1-3]
			\arrow["q", from=1-3, to=1-4]
			\arrow[from=1-4, to=1-5]
			\arrow[from=2-1, to=2-2]
			\arrow["i"', hook, from=2-2, to=1-2]
			\arrow[hook, from=2-2, to=2-3]
			\arrow[equals, from=2-3, to=1-3]
			\arrow["{q'}", from=2-3, to=2-4]
			\arrow["{\exists ! \eta_{N', N}}"', dotted, from=2-4, to=1-4]
			\arrow[from=2-4, to=2-5]
		\end{tikzcd}\]
		Let $\partial$ and $\partial'$ denote the index maps of associated $6$-term exact sequences of the upper and lower row, respectively. Denote by $\gamma'$ the $*$-homomorphism from $C^*(M)$ to $C^*(M)/C^*(N'\subseteq M)$ obtained by first compressing with $P_+' = \rho(\chi_{M_+'})$ and then applying $q'$. Note that for any $T \in C^*(M)$, we have
		\begin{align*}
			\eta_{N',N}\circ\gamma'(T) &= \eta_{N',N}\circ q'(P_+'TP_+') = q(P_+'TP_+') \\
			&= q(P_+TP_+) + q(P_+'TP_+' -P_+TP_+ ) \\
			&= q(P_+TP_+) + q((P_+'-P_+)TP_+' +P_+T(P_+'-P_+ )) \\
			&= q(P_+TP_+) + 0 = \gamma(T),
		\end{align*}
		where the second-to-last equality follows from the fact that $\supp(\chi_{M_+}-\chi_{M_+'}) \subseteq \ol{B_{R}(N)}$ by Lemma \ref{lem: near partitions have near hypersurfaces} and thus $(P_+'-P_+)TP_+' +P_+T(P_+'-P_+ ) \in C^*(N\subseteq M)$ by Lemma \ref{lem: Higsons 1.2 general}. So $\eta_{N',N}\circ\gamma'=\gamma$.
		It follows that 
		\begin{align*}
			K_0(i)\circ \varphi_{N'} &=  K_0(i)\circ \partial' \circ  K_1(\gamma') = \partial \circ K_1(\eta_{N',N}) \circ K_1(\gamma') \\
			&= \partial \circ K_1(\eta_{N',N} \circ \gamma')  = \partial \circ K_1(\gamma) = \varphi_N.\qedhere
		\end{align*}
	\end{enumerate}
\end{proof}

\subsection{The generalized partitioned manifold index theorem}
\label{subsec: generalized pmt}
Let $(M, g)$ be a oriented complete Riemannian manifold of odd dimension at least 3 with a connected partitioning hypersurface $N\subseteq M$ and partition $(N, M_+, M_-)$. Recall from Lemma \ref{lem: metric props hypersurfaces} that $N$ is complete when equipped with the induced metric $\restr{g}{N}$. Also, recall from 
Subsection \ref{subsec: partitioned manifolds}
that $N$ admits a unit normal vector field $n$ pointing into $M_+$. This $n$ induces an orientation on $N$ such that $n$ completes oriented frames for $TN$ to oriented frames for $\restr{TM}{N}$. Let $S$ be a Dirac bundle over $M$ with Clifford action $c\colon TM \rightarrow \End(S)$ and a Dirac connection $\nabla$ and let $D$ be the associated Dirac operator. Let $d = \dim(N)$, which is even. We will make the simplifying assumption that for every oriented orthonormal frame $(e_0, \ldots, e_d)$ of $T_pM$, we have that 
\begin{equation}
	\label{eq: orientation convention Clifford action}
	i^{1+d/2}c(e_0)\cdot\ldots\cdot c(e_d) = \id_{S_p}.
\end{equation}
This is not really a restriction, but makes the computations easier, see \cite[p. 440--441]{higson1991cobordisminvariance}.

On $N$, we consider the bundle $S_N = \restr{S}{N}$. This bundle has a Hermitian metric $h_N$ by pulling back the Hermitian metric $h$ to $S_N$. Consider the vector bundle morphism $c_N\colon TN\rightarrow \End(S_N)$ defined as the composition
\[
TN\xrightarrow{d(i_N)}\restr{TM}{N}\xrightarrow{c|_N}\restr{\End(S)}{N} \cong \End(S_N).
\]
Note that if $(p, X) \in TN$ and $s \in S_p$ so that also $s \in (S_N)_p$, then 
\begin{equation*}
	c_N(p, X)s = c\circ d(i_N)(p, X)s.
\end{equation*}
From this, it is clear that the vector bundle morphism $c_N$ still has the necessary properties to be a Clifford action, except that there is no grading yet. Since this Clifford action arises by identifying the tangent vectors of $N$ just as tangent vectors on $M$ via the inclusion, we will also denote $c_N$ by $c$ when this is convenient. Finally, choose a Dirac connection $\nabla^N$ on $S_N$, which exists by \cite[Cor. 3.41]{HeatKernels1992Berline} and define 
\begin{equation}
	\label{eq: grading induced on restricted bundle}
	\gamma_N = ic(n),
\end{equation}
which is a smooth vector bundle endomorphism of $S_N$.
\begin{remark}
	\label{rem: alternative expression grading}
	Note that due to Equation \eqref{eq: orientation convention Clifford action}, this definition of $\gamma_N$ agrees with the definition used in \cite{higson1991cobordisminvariance} where $\gamma_N$ is defined as
\[
	\gamma_N = i^{d/2}c(e_1)\cdot\ldots\cdot c(e_d),
\]
	where $(e_1,\ldots, e_d)$ is a local oriented orthonormal frame of $TN$.
\end{remark}
The following fact follows from the definitions.
\begin{lemma}
	\label{lem: grading operator}
	The endomorphism $\gamma_N$ is a self-adjoint involution. Moreover, when $S_N$ is endowed with the grading determined by $\gamma_N$, the Clifford action $c_N$ is odd and any Dirac connection on $S_N$ is even.
\end{lemma}

Let $\cN N \subset TM|_N$ be the normal bundle to $N$.
A \emph{uniform tubular neighbourhood} of $N$ is a diffeomorphic image under the exponential map of a set of the form
\[
\{(p, X)\in \cN N\colon g_p(X, X) < \varepsilon\},
\]
for some $\varepsilon>0$.  
Note that in the case where $N$ is a partitioning hypersurface and therefore comes with a unit normal vector field $n$, a uniform tubular neighbourhood $U$ of $N$ with radius $\varepsilon$ (if it exists) is diffeomorphic to $(-\varepsilon, \varepsilon)\times N$ via the diffeomorphism
\begin{equation}
	\label{eq: tub nbhd map}
	F\colon (-\varepsilon, \varepsilon)\times N \rightarrow U, \qquad (r, p)\mapsto \exp_p(rn_p).
\end{equation}

Let $D_N$ denote the Dirac operator on $S_N$ associated to the induced Clifford action $c_N$ and the chosen Dirac connection $\nabla^N$. 
Since $N$ is connected and is therefore naturally a metric space $(N, d_{N})$ using the induced Riemannian metric on $N$, it follows that $D_{N}$ has an index $\Ind(D_{N}) \in K_0(C^*(N))$ according to Definition \ref{def: index D}. 	Define $\Phi$ to be the composition
\begin{equation}
	\label{diag: identification maps}
	\begin{tikzcd}
		{K_0(C^*(N, d_{N}))} & {K_0(C^*(N, \restr{d_M}{N}))} & {K_0(C^*(N\subseteq M)),}
		\arrow["{(\id_N)_*}", from=1-1, to=1-2]
		\arrow["\sim", from=1-2, to=1-3]
	\end{tikzcd}
\end{equation}
where the first map is well-defined by Lemma \ref{lem: metric props hypersurfaces} and the second map is the isomorphism from Proposition \ref{prop: localized roe alg isomorphic to roe alg of subset}.

Our generalized version of the partitioned manifold index theorem is the following.
\begin{theorem}[Partitioned manifold index theorem]
	\label{thm: PMT}
	Let $M$ be an oriented, complete Riemannian manifold of odd dimension at least 3, partitioned by a connected hypersurface $N$. Let $S$ be a Dirac bundle over $M$ with Dirac connection $\nabla$ and let $D$ be the associated Dirac operator on $S$. Let $D_N$ be the Dirac operator on $S_N$ associated to the restricted Clifford action and an arbitrary Dirac connection on $S_N$. Suppose moreover that the following hold:
	\begin{enumerate}
		\item The map $\id_{N}\colon (N, d_{N})\rightarrow (N, \restr{d_M}{N})$ is a uniform equivalence.
		\item The submanifold $N\subseteq M$ has a uniform tubular neighbourhood $U$ of radius $\varepsilon >0$. 
		\item Let $F\colon (-\varepsilon, \varepsilon)\times N \rightarrow U$ be defined as in \eqref{eq: tub nbhd map} and let $g_{(-\varepsilon, \varepsilon)\times N}$ be the product metric on the manifold $(-\varepsilon, \varepsilon)\times N$. Then there exists a constant $C \geq 1$ such that on $U \cong (-\varepsilon, \varepsilon)\times N$ the estimates
		\begin{equation}
			\label{eq: estimate on Riemannian metric}
			\frac{1}{C}g_{(-\varepsilon, \varepsilon)\times N}\leq F^*(\restr{g_M}{U}) \leq Cg_{(-\varepsilon, \varepsilon)\times N}
		\end{equation} 
		hold as quadratic forms on vectors.
	\end{enumerate}
	
	Then
	\[
	\Phi(\Ind(D_N)) = \Ind(D;N),
	\]
	with $\Phi$ defined as in \eqref{diag: identification maps}.
\end{theorem}

The proof of this theorem will be given in Sections \ref{section: proof special case} and \ref{section: proof general case}.

\begin{remark}
	\label{rem: hypersurfaces as in PMT are simple}
	If $N\subseteq M$ is a partitioning hypersurface that satisfies all three requirements of Theorem \ref{thm: PMT}, then it is automatically a simple hypersurface. Indeed, let $\phi\colon (-\varepsilon, \varepsilon)\rightarrow \R$ be a smooth function such that $\phi(x) = 0$ for all $x \leq -\varepsilon/2$ and $\phi(x) = 1$ for all $x \geq \varepsilon/2$. Define 
	\[
	\phi_+(p) =\begin{cases}
		0 & \text{ if $p \in M_-\setminus U$,}\\
		1 & \text{ if $p \in M_+\setminus U$,}\\
		\phi\circ p_{\R}\circ F^{-1}(p) & \text{ if $p \in U$.}
	\end{cases}
	\]
	Then $\phi_+$ is a smooth, bounded function that agrees with the indicator function $\chi_{M_+}$ on $M\setminus B_{\varepsilon}(N)$ because $U
	\subseteq B_{\varepsilon}(N)$. It is left to show that $\|d\phi_+\|$ is bounded. Obviously $d\phi_+ = 0$ outside $U$. To estimate $\|d\phi_+\|$ on $U$, denote $m = \sup_{x \in (-\varepsilon, \varepsilon)}|\phi'(x)|$, which is finite as $\phi'$ is compactly supported.
	
	Let $p \in U$ and $v\in T_pM$ be arbitrary. Let $t\mapsto c(t)$ be any curve in $U$ such that $c(0) = p$ and $c'(0) = v$. It follows that we can write $F^{-1}\circ c = (c_{\R}, c_N)$ as curve through $(-\varepsilon, \varepsilon)\times N$. A simple computation now gives 
	\begin{align*}
		|d\phi_+(v)| &= \left|\restr{\frac{d}{dt}}{t = 0}\phi_+\circ c(t)\right| =\left|\restr{\frac{d}{dt}}{t = 0}\phi \circ c_{\R}(t)\right| \leq m|c'_{\R}(0)|\\
		& \leq m\|d(F^{-1})_p(v)\|_{g_{(-\varepsilon, \varepsilon)\times N}} \leq m\sqrt{C}\|d(F^{-1})_p(v)\|_{F^*g_M}\\
		&= m\sqrt{C}\|v\|_{g_M},
	\end{align*}
	from which we conclude that $\|d\phi_+\|\leq m\sqrt{C}$. So $\phi_+$ has all the properties as in Definition \ref{def: simple hypersurface} and therefore $N$ is a simple hypersurface.
\end{remark}

\subsection{Corollaries to the generalized partitioned manifold index theorem}
\label{subsec: corollaries to pmt}
In this subsection, we give several corollaries of the partitioned manifold index theorem. First of all, we show that the original partitioned manifold index theorem for connected, compact partitioning hypersurfaces \cite[Thm. 4.4]{Roe1996indextheory} is a special case of our Theorem \ref{thm: PMT}.

\begin{corollary}
	\label{cor: PMT for compact hypersurface}
	Let $M$ be an oriented, complete Riemannian manifold of odd dimension at least 3, partitioned by a connected compact hypersurface $N$. Let $S$ be a Dirac bundle over $M$ with Dirac connection $\nabla$ and let $D$ be the associated Dirac operator on $S$. Let $D_N$ be the Dirac operator on $S_N$ associated to the restricted Clifford action and an arbitrary Dirac connection on $S_N$. Then
	$\Ind(D_N) = \Ind(D;N)$ after identifying both $K$-theory groups with $\Z$.
\end{corollary}
\begin{proof}
This is the special case of Corollary \ref{cor: PMT for cocompact hypersurface} below where $\Gamma$ is the trivial group.
\end{proof}

A second easy corollary is a cobordism invariance result on the coarse index of Dirac operators. To this end, the following notion will be introduced.

\begin{definition}
	\label{def: cobordism of partitions}
	Let $N$ and $N'$ be two even-dimensional complete oriented connected manifolds with Dirac operators $D_N$ and $D_{N'}$. Then $(N, D_N)$ and $(N', D_{N'})$ are \emph{cobordant} if there exists a manifold $M$ and embeddings $N\hookrightarrow M$ and $N'\hookrightarrow M$ as partitioning hypersurfaces such that the resulting partitions are mutually near and for both embeddings all assumption in Theorem \ref{thm: PMT} are satisfied. 
\end{definition}

\begin{corollary}
	\label{cor: cobordism invariance of the index}
	Let $(N, D_N)$ and $(N', D_{N'})$ be cobordant. Then 
	\[
	\Ind(D_N) = \Phi^{-1}\circ \Phi'(\Ind(D_{N'})),
	\]
	where $\Phi$ and $\Phi'$ are the maps as in \eqref{diag: identification maps} associated to $N\subseteq M$ and $N'\subseteq M$ as partitioned manifolds, respectively. In particular, if one of the indices vanishes, so does the other.
\end{corollary}
\begin{proof}
	This follows from Corollary \ref{cor: mutually near hypersurfaces induce same partitioned index} and Theorem \ref{thm: PMT}.
\end{proof}

We also obtain a new result on the (non)existence of metrics with uniformly positive scalar curvature.

\begin{corollary}
	\label{cor: obstruction UPSC metrics}
	Let $M$ be an odd-dimensional complete spin manifold and let $D_M$ be the associated spin-Dirac operator. If there exists a partitioning hypersurface $N\subseteq M$ that satisfies all conditions in Theorem \ref{thm: PMT} such that the induced Dirac operator has nonzero index, then the metric on $M$ cannot have uniformly positive scalar curvature.
\end{corollary}
\begin{proof}
Let $\Phi$ be as in \eqref{diag: identification maps} and let $\varphi_N$ be the Roe homomorphism. Then 
	\[
	\Ind(D_N) = \Phi^{-1}(\Ind(D_M;N)) = \Phi^{-1}(\varphi_N(\Ind(D_M))).
	\]
	Thus if $\Ind(D_N)\neq 0$, then $\Ind(D_M)\neq 0$, so the metric on $M$ does not have uniformly positive scalar curvature \cite[Prop. 12.3.7]{Higson2000Khomology}. 
\end{proof}

\section{Proof of a special case: product manifolds}
\label{section: proof special case}

In this section, we will prove a version of the partitioned manifold index theorem, Theorem \ref{thm: PMT},  for a product manifold $\R\times N$ with partitioning hypersurface $\{0\}\times N$.

\subsection{Outline of the proof}
\label{subsec: overview of proof}
Let $(N, g_N)$ be a connected, complete, oriented, even-dimensional Riemannian manifold equipped with a Hermitian vector bundle $S_N$ with metric $h_N$ and grading $\gamma_N$, Clifford action $c_N\colon TN \rightarrow \End(S_N)$ and Dirac connection $\nabla^N$. Denote the associated Dirac operator on $S_N$ by $D_N$. On $\R$, considered as a Riemannian manifold with the Euclidean metric, consider the Dirac bundle and Dirac connection as in Example \ref{example: Dirac operator on R}, but where the Clifford action acts as 
\begin{align}
	c\colon\ \R\times\R &\rightarrow \End(T^1_{\R}),\notag\\
	(x, v)&\mapsto [(x, \lambda)\mapsto (x, -iv\lambda)].
\end{align} 
The associated Dirac operator is $-i\frac{d}{dx}$. Denote by $p_{\R}\colon \R\times N\rightarrow \R$ the projection onto $\R$ and by $p_N\colon \R\times N\rightarrow N$ the projection onto $N$. 

Consider $M:=\R\times N$ as an oriented Riemannian manifold with the product metric and orientation, equipped with the vector bundle $S = p_{\R}^*T^1_{\R}\otimes p_N^*S_N \cong p_N^*S_N$. This vector bundle obtains the structure of a Dirac bundle if we consider the pulled-back Hermitian metric and if we extend the Clifford action $c_N$ of $TN$ to $T(\R\times N)$ by letting the unit vector $\frac{\partial}{\partial x}$ act by $-i\gamma_N$. Consider the pulled-back connection $p_N^*\nabla^N$, which is again a Dirac connection. Denote the associated Dirac operator on $S$ by $D$. Under the canonical identification $L^2(\R\times N; S)\cong L^2(\R, T^1_{\R})\otimes L^2(N;S_N)$, the operator $D$ is given by \cite[p. 441]{higson1991cobordisminvariance}
\[
D = 1\otimes D_N - i\frac{d}{dx}\otimes \gamma_N.
\] 

Consider $N\subseteq M$ as an embedded submanifold via inclusion at $0$. Then $M$ naturally has the structure of a partitioned manifold with partitioning hypersurface $N$, if we set $M_- = (-\infty, 0]\times N$ and $M_+ = [0, \infty)\times N$. Since $\restr{d_M}{N} = d_N$, it follows that the map $\Phi$ in \eqref{diag: identification maps} reduces to 
\[\begin{tikzcd}
	{K_0(C^*(N, d_{N}))} & {K_0(C^*(N, \restr{d_M}{N}))} & {K_0(C^*(N\subseteq M))}.
	\arrow[equals, from=1-1, to=1-2]
	\arrow["\sim", from=1-2, to=1-3]
\end{tikzcd}
\]
 We always identify $N$ with the subset $\{0\}\times N$ of $\R\times N$.

%
%

The goal of this section is to prove the following theorem.
\begin{theorem}
	\label{thm: PMT for product}
	In the setting above, the following equality holds:
	\[
	\Phi(\Ind(D_N)) = \Ind(D, N) \in K_0(C^*(N\subseteq M)).
	\]
\end{theorem}

We first introduce some new notation that is needed to give an overview of the proof. Denote by $\rho_N\colon C_0(N)\rightarrow \cB(L^2(N;S_N))$ the ample representation by multiplication operators and denote by $\rho_M\colon B^{\infty}(M)\rightarrow \cB(L^2(M;S_M))$ the representation by multiplication operators of the bounded Borel functions on $M$. Consider the very ample representation $\tilde{\rho}_N = 1\otimes\rho_N\colon C_0(N) \rightarrow L^2(\R, T^1_{\R})\otimes L^2(N;S_N)\cong L^2(M; S)$. 
%
%
In this way, all of the algebras $C^*(N)$, $D^*(N)$ etc can be realized inside either $\cB(L^2(N; S_N))$ or $\cB(L^2(M; S))$. To keep track of which representation is used, the relevant Hermitian bundle will always be included in the notation.

Choose any unit vector $e\in L^2(\R, T^1_{\R})$ and let $W\colon L^2(N;S_N)\rightarrow L^2(M; S)$ be the isometry given by 
\[
W\colon \sigma\mapsto e\otimes \sigma.
\]
This map uniformly covers the identity, so the maps $\Ad_W$ and $\ol{\Ad_W}$ in Diagram \eqref{diag: the central diagram} are well-defined. Moreover, it can be shown (see Corollary \ref{cor: localization algebra inside Roe algebra}) that $C^*(N\subseteq M; S)\subseteq  C^*(N; S)$. Denote the resulting inclusion map by $i$. Then the following diagram can be constructed:

\begin{equation}
	\label{diag: the central diagram}
	\begin{tikzcd}
		& {K_1(Q^*(N; S_N))} & {K_0(C^*(N;S_N))} \\
		{K_0(N)} &&& {K_0(C^*(N\subseteq M;S))} \\
		& {K_1(Q^*(N;S))} & {K_0(C^*(N;S))}
		\arrow["\partial", from=1-2, to=1-3]
		\arrow["{K_1(\ol{\Ad_W})}"{description}, from=1-2, to=3-2]
		\arrow["\Phi", from=1-3, to=2-4]
		\arrow["{K_0(\Ad_W)}"{description}, from=1-3, to=3-3]
		\arrow["\sim", sloped, from=2-1, to=1-2]
		\arrow["\sim"', sloped, from=2-1, to=3-2]
		\arrow["{K_0(i)}", from=2-4, to=3-3]
		\arrow["\partial", from=3-2, to=3-3]
	\end{tikzcd}
\end{equation}

\begin{proposition}
	\label{prop: diagram commutes}
	Diagram \eqref{diag: the central diagram} commutes.
\end{proposition}
\begin{proof}
	Commutativity of the left triangle follows by writing out the action of $\ol{\Ad_W}$ and splitting off a degenerate Fredholm module and commutativity of the middle square is just naturality of the boundary map.
	
	It is left to show that the right triangle commutes. Recall that the map $\Phi$ is defined via an inductive limit. Denote $Y_k := \ol{B_k(N)} = [-k, k]\times N\subseteq M$. By Example \ref{ex: closed n-ball is coarsely equivalent}, the inclusion 
	\[
	\left(N, d_N = \restr{d_M}{N}\right) \hookrightarrow \left(Y_k, \restr{d_M}{Y_k}\right)
	\]
	is a coarse equivalence. Let $q_k$ be the coarse inverse given by $(r, p)\mapsto p$.
	
	Denote the inclusion of $C^*(Y_k;S)$ (computed using $L^2(Y_k; S)\subseteq L^2(M; S)$) into the localized Roe algebra by
	\begin{align*}
		j_k\colon\ C^*(Y_k; S)&\hookrightarrow C^*(N\subseteq M; S)\\
		T&\mapsto \begin{pmatrix}
			T&0\\
			0&0
		\end{pmatrix}
	\end{align*}
	with respect to the decomposition $L^2(M;S) = L^2(Y_k; S)\oplus L^2(Y_k; S)^{\perp}$. We will first prove that for all $k \in \N$, the following diagram commutes:
	\begin{equation}
		\label{diag: diagram for inductive limit}
		\begin{tikzcd}
			{K_0(C^*(N; S_N))} & {K_0(C^*(Y_k;S))} \\
			{K_0(C^*(N; S))} & {K_0(C^*(N\subseteq M; S))}
			\arrow["{\id_*}", from=1-1, to=2-1]
			\arrow["{(q_k)_*}"', from=1-2, to=1-1]
			\arrow["{K_0(j_k)}", from=1-2, to=2-2]
			\arrow["{K_0(i)}", from=2-2, to=2-1]
		\end{tikzcd}
	\end{equation}
	Indeed, let $i_H\colon L^2(Y_k;S)\hookrightarrow L^2(M;S)$ be the inclusion of Hilbert spaces. Then  $i\circ j_k(T) = \Ad_{i_H}(T)$. It follows from the definitions that 
 $i_H$ is an isometry covering $q_k$. 
So $K_0(i\circ j_k) = K_0(\Ad_{i_H}) = (q_k)_*$ and therefore Diagram \eqref{diag: diagram for inductive limit} commutes.
	
	It follows by Proposition \ref{prop: localized roe alg isomorphic to roe alg of subset} that $\Phi$ completes Diagram \eqref{diag: diagram for inductive limit} as follows
	
	\[\begin{tikzcd}
		{K_0(C^*(N; S_N))} & {K_0(C^*(Y_k;S))} \\
		{K_0(C^*(N; S))} & {K_0(C^*(N\subseteq M; S))}
		\arrow["{\id_*}", from=1-1, to=2-1]
		\arrow["\Phi", dotted, from=1-1, to=2-2]
		\arrow["\sim"', from=1-2, to=1-1]
		\arrow["\sim", sloped, from=1-2, to=2-2]
		\arrow["{K_0(i)}", from=2-2, to=2-1]
	\end{tikzcd}\]
	
	The result now follows as the lower triangle is precisely the right triangle in Diagram \eqref{diag: the central diagram}.
\end{proof}

The proof of the partitioned manifold index theorem for the product case is based on the commutativity of Diagram \eqref{diag: the central diagram}. The starting point is the fact that the upper path of the diagram computes $\Phi(\Ind(D_N))$, when evaluated at $[D_N] \in K_0(N)$. The rest of the proof is then divided into two main steps:
\begin{enumerate}
	\item First, we will construct an alternative representation of the class $[D_N]$ that is of the right shape to compute its image under the lower Paschke duality map in Diagram \eqref{diag: the central diagram} using Proposition \ref{prop: Paschke duality}. This will be done in Section \ref{sect: alternative representation of D_N class}.
	\item Secondly, having found a suitable representation of the class $[D_N]$ and having computed its image under Paschke duality, we will relate the resulting class in $K$-theory to the $K$-theory class $[q_N(U_+)]_1$ used to define $\Ind(D;N)$ in Definition \ref{def: partitioned index}. This will be done in Section \ref{sect: relating ind(D_N) to ind(D;N)}. To do this efficiently, several properties of  and relations between the (localized) Roe algebras of $M$ and of $N$ (realized in $\cB(L^2(M;S))$) will be useful. These will be derived in Section \ref{sect: properties and relations of subalgebras}.
\end{enumerate}

\subsection{Useful properties of the \texorpdfstring{$C^*$}{TEX}-algebras}
\label{sect: properties and relations of subalgebras}
It will be a central theme in the proof of Theorem \ref{thm: PMT for product} that both of the $C^*$-algebras $C_0(M)$ and $C_0(N)$ can be represented on the Hilbert space $L^2(M;S)$. For $C_0(M)$ this is done using the usual representation by multiplication operators, and for $C_0(N)$ this is done using the representation $\tilde{\rho}_N$ defined in Section \ref{subsec: overview of proof}. Note that $p_N^*f$ is usually never an element of $C_0(M)$ (unless $f = 0$) but still $p_N^*f \in C_b(M)$, which may be identified with the multiplier algebra of $C_0(M)$.

This implies that the $*$-algebras of locally compact, pseudolocal and finite propagation operators on $L^2(M; S)$ can be formed with respect to either the representation $\tilde \rho_N$ of $C_0(N)$ or the representation $\rho_M$ of $C_0(M)$. These will be denoted by $\LC(X;S)$, $\PL(X; S)$ and $\FP(X;S)$ for $X = N$ or $M$, depending on which representation is used. This gives rise to the different $C^*$-algebras $C^*(X;S)$, $D^*(X;S)$ and $Q^*(X;S)= D^*(X;S)/C^*(X;S)$, some of which appear in Diagram \eqref{diag: the central diagram}. Our goal in this section is to derive properties and relations between these algebras that will be used in the proof of the partitioned manifold index theorem.

\begin{lemma}
	\label{lem: initial properties of the algebras}
	The following inclusions hold:
	\begin{enumerate}
		\item \label{item: prop algebras 1} $\FP(M;S) \subseteq \FP(N;S)$.
		\item \label{item: prop algebras 2} $\LC(M;S)\cap\SN_N(M;S)\subseteq \LC(N;S)$.
		\item \label{item: prop algebras 3} $\PL(M;S)\cap\SN_N(M;S)\subseteq \PL(N;S)$.
	\end{enumerate}
\end{lemma}
\begin{proof}
	\begin{enumerate}
		\item This follows from the definitions.
		\item Let $T \in \LC(M;S)\cap \SN_N(M;S)$. Then there exists an $R >0$ such that for any $f\in B^{\infty}(M)$ with $d_M(\supp(f), \{0\}\times N)> R$, the operators $\rho_M(f)T$ and $T\rho_M(f)$ equal 0. Choose a partition of unity $\{\psi, 1-\psi\}$ subordinate to the open cover 
		\[
		\{(-R-2, R+2), \R\setminus[-R-1, R+1]\}
		\] of $\R$. For any $f'\in C_0(N)$, it follows that
		\[
		p_N^*f' = (p_{\R}^*\psi)(p_N^*f') + (p_{\R}^*(1-\psi))(p_N^*f').
		\]
		Note that the first term is an element of $C_0(M)$ and the second is supported within $\supp(p_{\R}^*(1-\psi)) = (\R\setminus [-R-1, R+1])\times N$. Therefore,
		\[
		\tilde{\rho}_N(f')T =\rho_M((p_{\R}^*\psi)(p_N^*f') + (p_{\R}^*(1-\psi))(p_N^*f'))T = \rho_M((p_{\R}^*\psi)(p_N^*f'))T, 
		\]
		which is compact as $T \in \LC(M;S)$. Similarly it follows that $T\tilde{\rho}_N(f')$ is compact and therefore that $T \in \LC(N;S)$.	
		\item Let $T \in \PL(M;S)\cap\SN_N(M;S)$ and let $R$ and $\psi$ be as in the proof of part \ref{item: prop algebras 2}. For all $f'\in C_0(N)$, it then follows that 
		\begin{align*}
			[\tilde{\rho}_N(f'), T]&= [\rho_M((p_{\R}^*\psi)(p_N^*f'))+\rho_M((p_{\R}^*(1-\psi))(p_N^*f')), T]\\
			&=[\rho_M((p_{\R}^*\psi)(p_N^*f')), T]+[\rho_M((p_{\R}^*(1-\psi))(p_N^*f')), T] \\
			&=[\rho_M((p_{\R}^*\psi)(p_N^*f')), T]+0,
		\end{align*}
		which is compact as $T \in \PL(M;S)$ and $\rho_M((p_{\R}^*\psi)(p_N^*f') \in C_0(M)$. Therefore $T \in \PL(N;S)$.\qedhere
	\end{enumerate}
\end{proof}

\begin{corollary}
	\label{cor: localization algebra inside Roe algebra}
	$C^*(N\subseteq M; S)$ is a $C^*$-subalgebra of $C^*(N; S)$.
%
\end{corollary}
\begin{proof}
	By combining parts \ref{item: prop algebras 1} and \ref{item: prop algebras 2} of Lemma \ref{lem: initial properties of the algebras}, we see that 
	\[
	\FP(M;S)\cap\LC(M;S)\cap\SN_N(M;S) \subseteq \FP(N;S)\cap \LC(N;S).
	\]
	So we conclude by taking closures.
\end{proof}

\begin{lemma}
	\label{lem: properties of algebras using C_0(R) functions}
	Let $g \in C_0(\R)$ and $T \in D^*(M;S)$. Then the following hold:
	\begin{enumerate}
		\item \label{item: prop D^* 1} $\rho_M(p_{\R}^*g)T, T\rho_M(p_{\R}^*g) \in D^*(N;S)$.
		\item \label{item: prop D^* 2} $[\rho_M(p_{\R}^*g), T] \in C^*(N\subseteq M;S)$.
	\end{enumerate}
\end{lemma}
\begin{proof}
	\begin{enumerate}
		\item Let $T \in \FP(M;S)\cap\PL(M;S)$ and let $g \in C_c(\R)$. Then \[\rho_M(p_{\R}^*g) \in \FP(M;S)\cap\PL(M;S)\cap\SN_N(M;S).\] 
		By Proposition \ref{prop: property induced sets are *algebras}, it follows that
		\[
		\rho_M(p_{\R}^*g)T, T\rho_M(p_{\R}^*g) \in \FP(M;S)\cap\PL(M;S)\cap\SN_N(M;S),
		\] 
		and by combining parts \ref{item: prop algebras 1} and \ref{item: prop algebras 3} from Lemma \ref{lem: initial properties of the algebras}, we see that 
		\[
		\rho_M(p_{\R}^*g)T, T\rho_M(p_{\R}^*g) \in \FP(N;S)\cap\PL(N;S) \subseteq D^*(N;S).
		\] 
		The result for general $T$ and $g$ follows by an approximation argument. 
		\item Let $T \in \FP(M;S)\cap\PL(M;S)$ and let $g \in C_c(\R)$. By the proof of part \ref{item: prop D^* 1} it immediately follows that
		\[
		[\rho_M(p_{\R}^*g), T] \in \FP(M;S)\cap\SN_N(M;S).
		\]
		Therefore only local compactness must be proved. For this, write $X\sim Y$ if $X$ and $Y$ differ by a compact operator. Let $f \in C_0(M)$ be arbitrary. Recall that $p_\R^*g \in C_b(M)$, which is the multiplier algebra of $C_0(M)$. Therefore, it follows that
		\begin{align*}
			[\rho_M(p_{\R}^*g), T]\rho_M(f) &= (\rho_M(p_{\R}^*g)T-T\rho_M(p_{\R}^*g))\rho_M(f) \\
			&= \rho_M(p_{\R}^*g)T\rho_M(f)-T\rho_M((p_{\R}^*g)f) \\
			&\sim \rho_M(p_{\R}^*g)T\rho_M(f)-\rho_M((p_{\R}^*g)f)T\\
			&= \rho_M(p_{\R}^*g)[T, \rho_M(f)] \sim 0,
		\end{align*}
		and similarly $\rho_M(f)[\rho_M(p_{\R}^*g), T]$ is compact. So indeed
		\[
		[\rho_M(p_{\R}^*g), T] \in \FP(M;S)\cap\LC(M;S)\cap\SN_N(M;S) \subseteq C^*(N\subseteq M;S).
		\]
		The result for general $T$ and $g$ follows by an approximation argument.\qedhere
	\end{enumerate}
\end{proof}

Also, consider the following variation of part \ref{item: prop D^* 2} of Lemma \ref{lem: properties of algebras using C_0(R) functions}. 

\begin{lemma}
	\label{lem: commutator of D* with locally constant in localized algebra}
	Let $g \in C(\R)$ such that there exists an $R>0$ such that $g$ is locally constant on $\R\setminus (-R,R)$ and let $T \in D^*(M;S)$. Then $[\rho_M(p_{\R}^*g), T] \in C^*(N\subseteq M;S)$.
\end{lemma}
\begin{proof}
	Let $T \in \FP(M;S)\cap\PL(M;S)$. It follows trivially that $[\rho_M(p_{\R}^*g), T] \in \FP(M;S)$, being the commutator of two operators with finite propagation. Also the exact same argument as in part \ref{item: prop D^* 2} of Lemma \ref{lem: properties of algebras using C_0(R) functions} shows that $[\rho_M(p_{\R}^*g), T] \in \LC(M;S)$. It is left to show that $[\rho_M(p_{\R}^*g), T]$ is supported near $N$.
	
	Choose $P>0$ such that for any $f_1, f_2 \in B^{\infty}(M)$ with $d_M(\supp(f_1), \supp(f_2))>P$, the operator $\rho_M(f_1)T\rho_M(T_2)$ equals 0. Choose a partition of unity $\{\psi_0, \psi_+, \psi_-\}$ subordinate to the open cover 
	\[
	\{(-R-2, R+2), (R+1, \infty), (-\infty, -R-1)\}
	\]
	of $\R$. For the sake of readability, the operators $\rho_M(p_{\R}^*\psi_*)$ will be written as $\hat{\psi}_*$ for $* \in \{0, +, -\}$. It is clear that $\hat{\psi}_0+\hat{\psi}_-+\hat{\psi}_+ = 1$. Let $f \in C_0(M)$ be any function such that $d_M(\supp(f), N)>P+R+2$. It follows that $fp_{\R}^*\psi_0 = 0$ and that $d_M(\supp(fp_{\R}^*\psi_{\pm}),\supp(1-p_{\R}^*\psi_{\pm}))>P$. Therefore,
	\begin{align*}
		[\rho_M(p_{\R}^*g), T]\rho_M(f)&= (\hat{\psi}_0+\hat{\psi}_-+\hat{\psi}_+)[\rho_M(p_{\R}^*g), T](\hat{\psi}_0+\hat{\psi}_-+\hat{\psi}_+)\rho_M(f)\\
		&= \hat{\psi}_-[\rho_M(p_{\R}^*g), T]\hat{\psi}_-\rho_M(f) +\hat{\psi}_+[\rho_M(p_{\R}^*g), T]\hat{\psi}_+\rho_M(f).
	\end{align*}
	Since $g$ is locally constant on $\R\setminus (-R,R)$, there exist constants $c_{\pm}$ such that $\restr{g}{(-\infty, -R]} = c_-$ and $\restr{g}{[R, \infty)} = c_+$. It follows that $g\psi_{\pm} = c_{\pm}\psi_{\pm}$ and therefore
	\begin{align*}
		\hat{\psi}_{\pm}[\rho_M(p_{\R}^*g), T]\hat{\psi}_{\pm}\rho_M(f) &= \hat{\psi}_{\pm}(\rho_M(p_{\R}^*g)T-T\rho_M(p_{\R}^*g))\hat{\psi}_{\pm}\rho_M(f)\\
		&= \hat{\psi}_{\pm}\rho_M(p_{\R}^*g)T\hat{\psi}_{\pm}\rho_M(f)-\hat{\psi}_{\pm}T\rho_M(p_{\R}^*g)\hat{\psi}_{\pm}\rho_M(f) \\
		&=c_{\pm}(\hat{\psi}_{\pm}T\hat{\psi}_{\pm}\rho_M(f)-\hat{\psi}_{\pm}T\hat{\psi}_{\pm}\rho_M(f)) = 0,
	\end{align*}
	from which follows that $[\rho_M(p_{\R}^*g), T]\rho_M(f) = 0$. Similarly (or by considering adjoints) it follows that $\rho_M(f)[\rho_M(p_{\R}^*g), T] = 0$ and therefore $[\rho_M(p_{\R}^*g), T] \in \SN_N(M;S)$. We conclude that $[\rho_M(p_{\R}^*g), T] \in C^*(N\subseteq M;S)$. For general $T$ the result follows by an approximation argument. 
\end{proof}

\subsection{An alternative representation of the class \texorpdfstring{$[D_N]$}{TEX}}
\label{sect: alternative representation of D_N class}

\begin{lemma}
	\label{lem: dirac class is positive off diagonal class}
	Consider $\R$ with the Hermitian bundle $T^2_{\R}$ equipped with the canonical grading $\gamma_{\R}$ induced by viewing $T^2_{\R}$ as $T^1_{\R}\oplus T^1_{\R}$. 
	Then \[
	D_- = \begin{pmatrix}
		0&- i\frac{d}{dx}\\
		- i\frac{d}{dx}&0
	\end{pmatrix}
	\] is a Dirac operator on $T^2_{\R}$, which is $1$-multigraded by the operator 
	\[
	\varepsilon_{-} = \begin{pmatrix}
		0&- i\\
		- i&0
	\end{pmatrix}.
	\] Moreover, $[D_{-}] = [D_{\R}] \in K_1(\R)$, where $D_{\R}$ denotes the Dirac operator from Example \ref{example: Dirac operator induced Dirac class}.
\end{lemma}
\begin{proof}
	$D_{-}$ is the Dirac operator associated to the Clifford action $c_{-}\colon T\R\rightarrow \End(T^2_{\R})$ defined by $\frac{d}{dx}\mapsto \begin{pmatrix}
		0&- i\\
		- i&0
	\end{pmatrix}$ 
	and the Dirac connection on $T^2_{\R}$ which is defined as the direct sum of two copies of the connection on $T^1_{\R}$ from Example \ref{example: Dirac operator on R}. It is trivial to show that $\{D_{-}, \gamma_{\R}\} = 0$ and that $[D_{-}, \varepsilon_{-}] = 0$, so that $D_{-}$ is a $1$-multigraded operator. It follows easily that $[D_{-}] = [D_{\R}]$ as these Fredholm modules are unitarily equivalent via the unitary map $U \in \cB(L^2(\R, T^2_{\R}))$ given by the matrix
	$\begin{pmatrix}
		1&0\\
		0&i
	\end{pmatrix}$.	
%
%
\end{proof}

\begin{proposition}
	\label{prop: alternative class of D_N}
	Let $\alpha\colon \R\rightarrow [-1,1]$ be the bounded continuous function
	\begin{equation*}
		\alpha(x) = \begin{cases}
			-1 &\text{ if $x \leq -1$},\\
			x&\text{ if $x \in [-1, 1]$},\\
			1 &\text{ if $x \geq 1$},
		\end{cases}
	\end{equation*} and denote the operator on $L^2(M;S)$ given by  $\rho_M(p_{\R}^*\alpha)$ by $\hat{\alpha}$. Write 
	\begin{equation}
		\label{eq: T_pm}
		T_{\pm} := \pm i\hat{\alpha} + (1-\hat{\alpha}^2)^{\frac{1}{2}}\chi(D),
	\end{equation}
	for some normalizing function $\chi$. Then
	\begin{equation}
	[D_N]_0 =
		-\left[L^2(M;S)^{\oplus 2},\tilde{\rho}_N^{\oplus 2}, \begin{pmatrix}
			0&T_-\\
			T_+&0
		\end{pmatrix} \right] \quad  \in K_0(N).
	\end{equation}	
\end{proposition}

The proof of Proposition \ref{prop: alternative class of D_N} will be divided into several lemmas. On the 1-multigraded Hilbert space $L^2(N, S_N)\otimes L^2(\R, T^2_{\R})$ with grading $\gamma = \gamma_N\otimes \gamma_{\R}$ and multigrading $\varepsilon_1 = 1\otimes \varepsilon_-$, consider the operator $D_N\times D_- = D_N\otimes 1 + \gamma_N\otimes D_-$. For $r=1,2$, let  $\rho_{\R}^r\colon C_0(\R)\rightarrow \cB(L^2(\R, T^r_{\R}))$ be the standard representation by multiplication operators on $T^r_{\R}$. 
%
%
Denote by $j\colon C_0(-1, 1)\rightarrow C_0(\R)$ the inclusion of $C^*$-algebras and by $\tau\colon C_0(\R)\otimes C_0(N)\rightarrow C_0(N)\otimes C_0(\R)$ the map that flips the tensor factors. For readability, we also define $\cH:= L^2(\R, T^1_{\R})\otimes L^2(N, S_N)$.

\begin{lemma}
	\label{lem: alternative class 1}
	Let $\chi$ be a normalizing function. Then
	\[
	[D_N\times D_-] = (\tau^{-1})^*\left[\cH^{\oplus 2}, (\rho^1_{\R}\otimes \rho_N)^{\oplus 2}, \begin{pmatrix}
		0&\chi(D)\\
		\chi(D)&0
	\end{pmatrix} \right]
	\]
	as elements of $K^{-1}(C_0(N)\otimes C_0(\R))$.
\end{lemma}
\begin{proof}
		Consider the unitary map 
	\begin{align*}
		U\colon L^2(N, S_N)\otimes L^2(\R, T^2_{\R}) &\rightarrow \cH\oplus \cH \\
		\sigma \otimes (f_+, f_-)&\mapsto (f_+\otimes \sigma_+ + f_-\otimes \sigma_-, f_+\otimes \sigma_-+f_-\otimes \sigma_+),
	\end{align*}
	where $\sigma_{\pm}$ are the homogeneous components of $\sigma$ such that $\gamma_N(\sigma_{\pm}) = \pm\sigma_{\pm}$. Endow the Hilbert space  $\cH\oplus \cH$ with a 1-multigrading where the grading operator $\delta$ and the multigrading operator $\eta_1$ are defined by
	\begin{equation*}
		\delta = \begin{pmatrix}
			1&0\\
			0&-1
		\end{pmatrix}\quad\text{ and }\quad
		\eta_1 = \begin{pmatrix}
			0&-i\\
			-i&0
		\end{pmatrix}.
	\end{equation*}
	
Easy computations show that
%
	$U$ is a 1-multigraded unitary operator. A very similar computation shows that
	\[
	((\rho^1_{\R}\otimes \rho_N)\circ\tau^{-1})^{\oplus 2} U - U(\rho_N\otimes\rho_{\R}^2) = 0,
	\] 
	so that the representation $((\rho^1_{\R}\otimes \rho_N)\circ\tau^{-1})^{\oplus 2}$ is unitarily equivalent to $\rho_N\otimes\rho_{\R}^2$.
	
	Finally, the Dirac operator $D_N\times D_-$ is unitarily equivalent to the Dirac operator 
	\[
	\begin{pmatrix}
		0&D\\
		D&0
	\end{pmatrix}
	\]
	on $\cH\oplus \cH$ via the unitary $U$. Indeed, for smooth and compactly supported $\sigma$ and $f_{\pm}$, writing out definitions shows that
%
 \[
	\begin{pmatrix}
		0&D\\
		D&0
	\end{pmatrix}U(\sigma\otimes (f_+, f_-)) = U(D_N\times D_-(\sigma\otimes (f_+, f_-))).
	\]
	Therefore
	\begin{align*}
		[D_N\times D_-] &= \left[\cH^{\oplus 2}, ((\rho^1_{\R}\otimes \rho_N)\circ\tau^{-1})^{\oplus 2},\chi \begin{pmatrix}
			0&D\\
			D&0
		\end{pmatrix} \right] \\
		&= (\tau^{-1})^*\left[\cH^{\oplus 2}, (\rho^1_{\R}\otimes \rho_N)^{\oplus 2}, \begin{pmatrix}
			0&\chi(D)\\
			\chi(D)&0
		\end{pmatrix} \right]\in K^{-1}(C_0(N)\otimes C_0(\R)).
	\end{align*}
   The last equality is due to the fact that $\chi$ is odd. 
\end{proof}

\begin{lemma}
	\label{lem: alternative class 2} 
	We have 
	\begin{equation*}
		s((j\otimes 1)^*\tau^*[D_N\times D_-])=
		\left[L^2(M;S)^{\oplus 2},\tilde{\rho}_N^{\oplus 2}, \begin{pmatrix}
			0&T_-\\
			T_+&0
		\end{pmatrix} \right] \in K_0(N).
	\end{equation*}
\end{lemma}
\begin{proof}
	By Lemma \ref{lem: alternative class 1}, it follows that
	\begin{equation*}
		(j\otimes 1)^*\tau^*[D_N\times D_-] = \left[\cH^{\oplus 2}, (\rho^1_{\R}\circ j\otimes \rho_N)^{\oplus 2},\begin{pmatrix}
			0&\chi(D)\\
			\chi(D)&0
		\end{pmatrix} \right]
	\end{equation*}
	as element of $K^{-1}(C_0(-1,1)\otimes C_0(N))$. To compute the suspension of this element in $K$-homology, it should be represented by a nondegenerate Fredholm module. For this, denote by $L\subseteq \cH$ the closed subspace $L^2((-1, 1); T^1_{(-1,1)})\otimes L^2(N;S_N)$ and denote by $P$ the orthogonal projection of $\cH$ onto $L$. The representation 
\[
\rho^1_{\R}\circ j\otimes \rho_N\colon C_0(-1,1) \otimes C_0(N) \rightarrow \cB(L)
\]
is nondegenerate.	
	It follows by \cite[Lem. 8.3.8]{Higson2000Khomology} that 
	\begin{align}
		\label{eq: nondeg Fredholm module}
		&\left[\cH^{\oplus 2}, (\rho^1_{\R}\circ j\otimes \rho_N)^{\oplus 2}, \begin{pmatrix}
			0&\chi(D)\\
			\chi(D)&0
		\end{pmatrix} \right] \notag \\
		&\qquad=\left[L^{\oplus 2}, (\rho^1_{\R}\circ j\otimes \rho_N)^{\oplus 2}, \begin{pmatrix}
			0&P\chi(D)P\\
			P\chi(D)P&0
		\end{pmatrix} \right],
	\end{align}
	where the latter is a nondegenerate Fredholm module. By nondegeneracy, the representation $\rho^1_{\R}\circ j\otimes \rho_N$ extends uniquely to any $C^*$-algebra containing $C_0(-1, 1)\otimes C_0(N)$ as an ideal (see \cite[~p.255]{Higson2000Khomology}), such as $C([-1, 1])\otimes \widetilde{C_0(N)}$.
	
	Write $z:=(\rho^1_{\R}\circ j\otimes \rho_N)(\id_{(-1, 1)}\otimes 1)$ (using the extended representation). It follows that the suspension of the Fredholm module in \eqref{eq: nondeg Fredholm module} equals
	\beq{eq: D_N representation 1}
		\left[L^{\oplus 2}, \uprestr{\tilde{\rho}_N}{L}{\oplus 2}, \begin{pmatrix}
			0&-iz + (1-z^2)^{\frac{1}{2}}P\chi(D)P\\
			iz+(1-z^2)^{\frac{1}{2}}P\chi(D)P&0
		\end{pmatrix}\right].
		\eeq
	
	Finally, note that $(\rho^1_{\R}\otimes \rho_N)(\alpha\otimes 1)$ is unitarily equivalent to the operator $\hat{\alpha}$ on $L^2(M;S)$ under the isomorphism $L^2(M;S)\cong L^2(\R; T^1_{\R})\otimes L^2(N;S_N)$. This operator will be denoted by $\hat{\alpha}$ too. Using the graded unitary map
	\[
	L^2(M;S)^{\oplus 2}\cong \cH^{\oplus 2}\xrightarrow{\sim} L^{\oplus 2}\oplus (L^{\bot})^{\oplus 2},
	\]
	we see that
	\begin{align}
		\label{eq: intermediate Fredholm module 1}
		\left[L^2(M;S)^{\oplus 2}, \tilde{\rho}_N^{\oplus 2}, \begin{pmatrix}
			0 & T_-\\
			T_+ & 0
		\end{pmatrix}\right] =\left[L^{\oplus 2}\oplus (L^{\perp})^{\oplus 2}, \uprestr{\tilde{\rho}_N}{L}{\oplus 2}\oplus \uprestr{\tilde{\rho}_N}{L^{\perp}}{\oplus 2}, \cM\right],
	\end{align}
	where 
	\[
	\cM = 	\begin{pmatrix}
		0&PT_-P&0&PT_-(1-P)\\
		PT_+P&0&PT_+(1-P)&0\\
		0&(1-P)T_-P&0&(1-P)T_-(1-P)\\
		(1-P)T_+P&0&(1-P)T_+(1-P)&0
	\end{pmatrix}.
	\]
	To simplify the above expression, it will be shown that the anti-diagonal terms of $\cM$ form a locally compact operator. First note that since $\hat{\alpha}$ is local, we have that
	\[
	P\hat{\alpha}(1-P) = (1-P)\hat{\alpha}P = 0.
	\]
	So
\begin{equation}\label{eq: off diagonal is locally compact 1}
		(1-P)T_{\pm}P = (1-P)(1-\hat{\alpha}^2)^{\frac{1}{2}}\chi(D)P
= 0,
\end{equation}
	and 
	\begin{align}
		\label{eq: off diagonal is locally compact 2}
		PT_{\pm}(1-P) &= P(1-\hat{\alpha}^2)^{\frac{1}{2}}\chi(D)(1-P) \notag\\
		&=P[(1-\hat{\alpha}^2)^{\frac{1}{2}}, \chi(D)](1-P) + P\chi(D)(1-\hat{\alpha}^2)^{\frac{1}{2}}(1-P)\notag\\
		&= P[(1-\hat{\alpha}^2)^{\frac{1}{2}}, \chi(D)](1-P) +0.
	\end{align}
	
	Since $[(1-\hat{\alpha}^2)^{\frac{1}{2}}, \chi(D)] \in C^*(N\subseteq M;S)$ by part \ref{item: prop D^* 1} of Lemma \ref{lem: properties of algebras using C_0(R) functions} and  $1-P, P \in D^*(M;S)$, it follows that $PT_{\pm}(1-P) \in  C^*(N\subseteq M;S) \subseteq C^*(N;S)$. It follows from \eqref{eq: off diagonal is locally compact 1} and \eqref{eq: off diagonal is locally compact 2} that the anti-diagonal in $\cM$ indeed is locally compact for the representation $\uprestr{\tilde{\rho}_N}{L}{\oplus 2}\oplus \uprestr{\tilde{\rho}_N}{L^{\perp}}{\oplus 2}$ and thus $\cM$ can be replaced by 
	\[
	\cM' = 	\begin{pmatrix}
		0&PT_-P&0&0\\
		PT_+P&0&0&0\\
		0&0&0&(1-P)T_-(1-P)\\
		0&0&(1-P)T_+(1-P)&0
	\end{pmatrix}
	\]
	without changing the $K$-homology class. It follows that the right hand side of  \eqref{eq: intermediate Fredholm module 1} equals
	\begin{multline}\label{eq: intermediate Fredholm module 3}
		\left[L^{\oplus 2}, \uprestr{\tilde{\rho}_N}{L}{\oplus 2}, \begin{pmatrix}
			0&PT_-P\\
			PT_+P&0
		\end{pmatrix}\right] 
		\\
		+ \left[(L^{\perp})^{\oplus 2}, \uprestr{\tilde{\rho}_N}{L^{\perp}}{\oplus 2}, \begin{pmatrix}
			0&(1-P)T_-(1-P)\\
			(1-P)T_+(1-P)&0
		\end{pmatrix}\right] 
		\end{multline}
		We have
			\[
	PT_{\pm}P = \pm iz+(1-z^2)^{\frac{1}{2}}P\chi(D)P.
	\]
So the first term in \eqref{eq: intermediate Fredholm module 3} equals		$s((j\otimes 1)^*\tau^*[D_N\times D_-])$. Furthermore, write 
%
%
%
%
%
%
%
%
%
%
%
	$L^{\perp} = \cH_{\leq -1}\oplus \cH_{\geq 1}$, with
	\begin{align*}
		\cH_{\leq -1} &= L^2((-\infty, -1]; \restr{T^1_{\R}}{(-\infty, -1]})\otimes L^2(N;S_N), \text{ and}\\
		\cH_{\geq 1} &= L^2([1,\infty); \restr{T^1_{\R}}{[1,\infty)})\otimes L^2(N;S_N).
	\end{align*}
Then
	\[
	(1-P)T_{\pm}(1-P)= \pm i (1-P)\hat{\alpha}(1-P) = \pm i\begin{pmatrix}
		-1&0\\
		0&1
	\end{pmatrix}.
	\]
	This shows that second term in \eqref{eq: intermediate Fredholm module 3}
	 is given by the degenerate module
	\[
	\left[(\cH_{\leq -1}\oplus \cH_{\geq 1})^{\oplus 2}, \uprestr{\tilde{\rho}_N}{L^{\perp}}{\oplus 2}, \begin{pmatrix}
		0&0&i&0\\
		0&0&0&-i\\
		-i&0&0&0\\
		0&i&0&0
	\end{pmatrix}\right],
	\]
	and hence zero. We conclude that \eqref{eq: intermediate Fredholm module 3} equals 	$s((j\otimes 1)^*\tau^*[D_N\times D_-])$. 
\end{proof}

\begin{proof}[Proof of Proposition \ref{prop: alternative class of D_N}]
	 Let $d \in K_1(C_0(-1,1))$ be the Dirac class as defined in Definition \ref{def: Dirac class}. By combining \cite[Prop. 9.2.13, 9.2.14, 10.8.8]{Higson2000Khomology}, \cite[Ex. 10.9.7]{Higson2000Khomology}, \cite[Thm. 9.5.2, 10.7.3]{Higson2000Khomology} and taking into account the sign error in \cite{Higson2000Khomology} that was resolved in Lemma \ref{lem: suspension maps Dirac class to -1}, we conclude that 
	\begin{align*}
		-[D_N] &= s(d\times [D_N]) = s(j^*[D_{\R}]\times [D_N]) = s((j\otimes 1)^*\tau^*([D_N]\times [D_{\R}])) \\
		&=s((j\otimes 1)^*\tau^*([D_N]\times [D_-])) = s((j\otimes 1)^*\tau^*[D_N\times D_-]) \\
		& =\left[L^2(M;S)^{\oplus 2},\tilde{\rho}_N^{\oplus 2}, \begin{pmatrix}
			0&T_-\\
			T_+&0
		\end{pmatrix} \right], 		
	\end{align*}
	where the last equality follows from Lemma \ref{lem: alternative class 2}.
\end{proof}


Recall the definition of the operators $T_{\pm}$ in \eqref{eq: T_pm}. Moreover, write $D^*(N, M;S):= D^*(M;S)\cap D^*(N;S)$. Then $D^*(N, M;S)$ contains $C^*(N\subseteq M;S)$ as 
\[
C^*(N\subseteq M;S) \subseteq C^*(N;S)\cap C^*(M;S) \subseteq D^*(N, M;S),
\]
where the first inclusion follows as $C^*(N\subseteq M;S)\subseteq C^*(N;S)$ by Corollary \ref{cor: localization algebra inside Roe algebra} and $C^*(N\subseteq M;S)\subseteq C^*(M;S)$ trivially. Therefore, since $C^*(N\subseteq M;S)\subseteq D^*(M;S)$ is an ideal, it follows that $C^*(N\subseteq M;S)\subseteq D^*(N, M;S)$ is an ideal. Let $q$ be the corresponding quotient map.
\begin{lemma}
	\label{lem: T in intersection and has unitary quotient}
	The operator $T_+$ is an element of $D^*(N, M;S)$ and $q(T_+)$ is unitary as element of $D^*(N, M;S)/C^*(N\subseteq M;S)$.
\end{lemma}
\begin{proof}
	The operators $\hat{\alpha}$ and $(1-\hat{\alpha})^{\frac{1}{2}}$ are elements of $D^*(N, M;S)$ as these are multiplication operators. So since $\chi(D) \in D^*(M;S)$ it is immediate that $T_+ \in D^*(M;S)$. By part \ref{item: prop D^* 1} of Lemma \ref{lem: properties of algebras using C_0(R) functions} it follows that 
	$(1-\hat{\alpha})^{\frac{1}{2}}\chi(D)\in D^*(N;S)$ as $(1-\alpha^2)^{\frac{1}{2}} \in C_0(\R)$. It follows that $T_+ \in D^*(N;S)$ and therefore $T_+ \in D^*(N, M;S)$. To show that $q(T_+) \in D^*(N, M;S)/C^*(N\subseteq M;S)$ is unitary, write $X\sim Y$ when $X$ and $Y$ differ by an element in $C^*(N\subseteq M;S)$. We have that
	\begin{align*}
		T_+^*T_+ &= (-i\hat{\alpha} + \chi(D)(1-\hat{\alpha}^2)^{\frac{1}{2}})(i\hat{\alpha} + (1-\hat{\alpha}^2)^{\frac{1}{2}}\chi(D)) \\
		&= \hat{\alpha}^2 + \chi(D)(1-\hat{\alpha}^2)\chi(D) + i(\chi(D)(1-\hat{\alpha}^2)^{\frac{1}{2}}\hat{\alpha} - \hat{\alpha}(1-\hat{\alpha}^2)^{\frac{1}{2}}\chi(D)) \\
		&\sim \hat{\alpha}^2 + \chi(D)(1-\hat{\alpha}^2)\chi(D) \sim \hat{\alpha}^2 + \chi(D)^2(1-\hat{\alpha}^2) \\
		&= (1-\chi(D)^2)(\hat{\alpha}^2-1) + 1\sim 1,
	\end{align*}
	where the equivalences on the third line follow from part \ref{item: prop D^* 2} of Lemma \ref{lem: properties of algebras using C_0(R) functions} and the final equivalence follows from Lemma \ref{lem: Higsons 1.2 general} as $1-\chi(D)^2 \in C^*(M;S)$ and $\alpha^2-1$ is a $C_0(\R)$ function. Similarly, we have that $T_+T_+^*\sim 1$. The result follows.
%
\end{proof}

\begin{corollary}
	\label{cor: quotients of T+ are unitary}
	The class $q(T_+)$ is unitary in the following quotient algebras: $Q^*(N;S)$, $D^*(N, M;S)/C^*(N\subseteq M;S)$ and $D^*(M;S)/C^*(N\subseteq M;S)$ (where by abuse of notation all quotient maps are denoted by $q$).
\end{corollary}
\begin{proof}
	This follows easily as the inclusions 
	\[
	D^*(N, M;S)\hookrightarrow D^*(M;S)\quad \text{ and }\quad D^*(N, M;S)\hookrightarrow D^*(N;S)
	\]
	are unital $*$-homomorphisms that descend to unital $*$-homomorphisms
	\[
	D^*(N, M;S)/C^*(N\subseteq M;S)\rightarrow D^*(M;S)/C^*(N\subseteq M;S)
	\]
	and
	\[D^*(N, M;S)/C^*(N\subseteq M;S)\rightarrow D^*(N;S)/C^*(N;S)
	\]
	(for the latter this follows from Corollary \ref{cor: localization algebra inside Roe algebra}). It directly follows that these homomorphisms on the quotients map the element $q(T_+) \in D^*(N, M;S)/C^*(N\subseteq M;S)$ to the elements $q(T_+)$ in the other two quotient algebras. Since the former is unitary by Lemma \ref{lem: T in intersection and has unitary quotient}, so are the latter two being images of a unitary element under a unital $*$-homomorphism.
\end{proof}

\begin{corollary}
	\label{cor: image of D_N class under paschke duality}
	The Paschke duality map 
	\[
	K_0(N) \xrightarrow{\sim} K_1(Q^*(N;S))
	\]
	maps the class $[D_N]$ to $-[q(T_+)]_1 \in K_1(Q^*(N;S))$.
\end{corollary}
\begin{proof}
	Note that
    \[
		T_--T_+^* 
		= [(1-\hat{\alpha}^2)^{\frac{1}{2}}, \chi(D)] \in C^*(N\subseteq M;S)\subseteq C^*(N;S),
        \]
	by part \ref{item: prop D^* 2} of Lemma \ref{lem: properties of algebras using C_0(R) functions} and Corollary \ref{cor: localization algebra inside Roe algebra}. It follows that $T_+^*$ is a locally compact perturbation of $T_-$. Therefore, by Proposition \ref{prop: alternative class of D_N}, $[D_N]$ can be represented as
	\[
	-\left[L^2(M;S)^{\oplus 2},\tilde{\rho}_N^{\oplus 2}, \begin{pmatrix}
		0&T_+^*\\
		T_+&0
	\end{pmatrix} \right].
	\]
	We know that $T_+ \in D^*(N;S)$ by Lemma \ref{lem: T in intersection and has unitary quotient} and that $q(T_+) \in Q^*(N;S)$ is unitary by Corollary \ref{cor: quotients of T+ are unitary}, so it follows by Proposition \ref{prop: Paschke duality} that the Paschke duality map maps $[D_N]$ to $-[q(T_+)]_1$.
\end{proof}

\subsection{Relating \texorpdfstring{$\Ind(D_N)$}{TEX} to \texorpdfstring{$\Ind(D;N)$}{TEX}}
\label{sect: relating ind(D_N) to ind(D;N)}

At this point, the $K$-homology class $[D_N]$ is represented as the $K$-theory class $-[q(T_+)]_1$. The rest of the proof of Theorem \ref{thm: PMT for product} will consist of finding the right homotopies relating $-[q(T_+)]_1$ to $[q(U_+)]_1$, which is used to compute $\Ind(D;N)$. These classes live in different $K$-theory groups but they can be related. To this end, consider the following diagram of $C^*$-algebras with short exact rows

\begin{equation}
	\label{diag: relating C^* extensions}
	\begin{tikzcd}
		0 & {C^*(N\subseteq M;S)} & {D^*(M;S)} & {\frac{D^*(M;S)}{C^*(N\subseteq M;S)}} & 0 \\
		0 & {C^*(N\subseteq M;S)} & {D^*(N, M;S)} & {\frac{D^*(N, M;S)}{C^*(N\subseteq M;S)}} & 0 \\
		0 & {C^*(N;S)} & {D^*(N;S)} & {Q^*(N;S)} & 0
		\arrow[from=1-1, to=1-2]
		\arrow[from=1-2, to=1-3]
		\arrow[from=1-3, to=1-4]
		\arrow[from=1-4, to=1-5]
		\arrow[from=2-1, to=2-2]
		\arrow[equals, from=2-2, to=1-2]
		\arrow[from=2-2, to=2-3]
		\arrow["i", hook, from=2-2, to=3-2]
		\arrow["{j_M}"', hook', from=2-3, to=1-3]
		\arrow[from=2-3, to=2-4]
		\arrow["{j_N}", hook, from=2-3, to=3-3]
		\arrow["{\exists!\ol{j_M}}"', dotted, from=2-4, to=1-4]
		\arrow[from=2-4, to=2-5]
		\arrow["{\exists!\ol{j_N}}", dotted, from=2-4, to=3-4]
		\arrow[from=3-1, to=3-2]
		\arrow[from=3-2, to=3-3]
		\arrow[from=3-3, to=3-4]
		\arrow[from=3-4, to=3-5]
	\end{tikzcd}
\end{equation}
where the maps induced on the quotients are precisely the maps that were already used in the proof of Corollary \ref{cor: quotients of T+ are unitary}. 
%
By Corollary \ref{cor: quotients of T+ are unitary} it follows that the class $[q(T_+)]_1 \in K_1(Q^*(N;S))$ lifts under $K_1(\ol{j_N})$ to the class $[q(T_+)]_1$ as element of $K_1(D^*(N, M;S)/C^*(N\subseteq M;S))$ and this latter class  maps to the class $[q(T_+)]_1 \in K_1(D^*(M;S)/C^*(N\subseteq M;S))$, which is precisely the group where $[q(U_+)]_1$ lives.

\begin{proposition}
	\label{prop: relating K theory classes}
	Let $U_+$ be defined as in Proposition \ref{prop: Higson's 1.4}. Then	$[q(T_+)]_1 = -[q(U_+)]_1$ in $K_1(D^*(M;S)/C^*(N\subseteq M;S))$.
\end{proposition}
The proof of Proposition \ref{prop: relating K theory classes} will be spread over Lemmas \ref{lem: relating K theory classes 1}, \ref{lem: relating K theory classes 2} and \ref{lem: relating K theory classes 3}, which together clearly imply that the equality in Proposition \ref{prop: relating K theory classes} hold. The homotopies used in Lemmas \ref{lem: relating K theory classes 1} and \ref{lem: relating K theory classes 2} are inspired by \cite[Lem. 8.6.12]{Higson2000Khomology} and the homotopy used in Lemma \ref{lem: relating K theory classes 3} is inspired by \cite[Prop. 8.3.16]{Higson2000Khomology}.

Let $\theta(x) = \frac{2}{\pi}\arctan(x)$ and let $\alpha$ be defined as in Proposition \ref{prop: alternative class of D_N}.

\begin{lemma}
	\label{lem: relating K theory classes 1}
	We have
	\[
	[q(U_+)]_1 = \left[q\left(i\sin\left(\frac{\pi}{2}\theta(D)\right) + \hat{\alpha}\cos\left(\frac{\pi}{2}\theta(D)\right)\right)\right]_1
	\]
	as elements of $K_1(D^*(M;S)/C^*(N\subseteq M;S))$.
\end{lemma}
\begin{proof}
		Write $\alpha_+:= \frac{1 + p_{\R}^*\alpha}{2}$ and $\alpha_- = 1-\alpha_+$. Let $\phi_+$ be as in Definition \ref{def: simple hypersurface}, such that $U_+ = \rho_M(\phi_-)+\rho_M(\phi_+)U$ with $U$ the Cayley transform of $D$. Then
	\begin{align*}
		\rho_M(\alpha_-) + \rho_M(\alpha_+)U - U_+ &= \rho_M(\phi_+-\alpha_+) + \rho_M(\alpha_+-\phi_+)U \\
		&= \rho_M(\alpha_+-\phi_+)(U-1) \in C^*(N\subseteq M;S),
	\end{align*} 
	which follows from Lemma \ref{lem: Higsons 1.2 general} as $U-1 \in C^*(M;S)$ and $\alpha_+-\phi_+$ is supported near $N$. Therefore it follows that
	\[
	q(U_+) = q(\rho_M(\alpha_-) + \rho_M(\alpha_+)U) = q\left(\rho_M(\alpha_-) - \rho_M(\alpha_+)e^{i\pi\theta(D)}\right),
	\]
	where the last equality follows as $U = -e^{i\pi\theta(D)}$, see Equations \eqref{eq: Cayley transform as exponential} and \eqref{eq: Ind(D) = [U]}.
	
	For $t \in [0,1]$, consider the family of functions
	\[
	f_t(x) = e^{-\frac{\pi}{2} i \theta(x)t}.
	\]
	It is obvious that $f_t^*f_t = f_tf_t^* = 1$ for all $t$, and an easy computation using the power series expansion of the exponential shows that for all $s, t \in [0,1]$, we have
	\[
	\|f_s-f_t\|_{\infty} \leq e^{\frac{\pi}{2} |s-t|}-1.
	\]
	This shows that $t\mapsto f_t$ is a continuous path in $B^{\infty}(\R)$. Moreover, the function $f_t$ has finite limits as $x \rightarrow \pm \infty$, which implies that $f_t(D) \in D^*(M;S)$ for all $t$. It follows that the map
	\begin{align*}
		f_{\bullet}(D)\colon\ [0, 1]&\rightarrow D^*(M;S),\\
		t &\mapsto f_t(D),
	\end{align*}
	is a well-defined continuous path of unitaries in $D^*(M;S)$ from $1$ to $ e^{-\frac{\pi}{2} i \theta(D)}$. So
	\begin{align*}
		H\colon\ [0, 1]&\rightarrow D^*(M;S)/C^*(N\subseteq M;S), \\
		t&\mapsto q\left(\rho_M(\alpha_-) - \rho_M(\alpha_+)e^{\pi i\theta(D)}\right)q(f_t(D))),
	\end{align*}
	is a homotopy of unitaries connecting $q\left(\rho_M(\alpha_-) - \rho_M(\alpha_+)e^{\pi i\theta(D)}\right)$ to 
	\begin{align}
		\label{eq: first homotoped operator}
		&q\left(\rho_M(\alpha_-) - \rho_M(\alpha_+)e^{\pi i\theta(D)}\right)q(e^{-\frac{\pi}{2} i \theta(D)})\notag\\
		&\qquad = q\left(\rho_M(\alpha_-)e^{-\frac{\pi}{2} i \theta(D)} - \rho_M(\alpha_+)e^{\frac{\pi}{2}i\theta(D)}\right) \notag\\
		& \qquad = q\left(\left(\frac{1-\hat{\alpha}}{2}\right)e^{-\frac{\pi}{2} i \theta(D)} - \left(\frac{1+\hat{\alpha}}{2}\right)e^{\frac{\pi}{2} i \theta(D)}\right)\notag\\
		& \qquad =-q\left(i\sin\left(\frac{\pi}{2}\theta(D)\right) + \hat{\alpha}\cos\left(\frac{\pi}{2}\theta(D)\right)\right).
	\end{align}
	 Since $[-1]_1 = 0$ as $1$ and $-1$ are homotopic via the homotopy of unitaries $t\mapsto e^{i\pi t}$, the result follows.
\end{proof}

\begin{lemma}
	\label{lem: relating K theory classes 2}
	We have
	\[
	\left[q\left(i\sin\left(\frac{\pi}{2}\theta(D)\right) + \hat{\alpha}\cos\left(\frac{\pi}{2}\theta(D)\right)\right)\right]_1 = \left[q\left(i\theta(D) + \hat{\alpha} (1-\theta(D)^2)^{\frac{1}{2}}\right)\right]_1
	\]
	as elements of $K_1(D^*(M;S)/C^*(N\subseteq M;S))$.
\end{lemma}
\begin{proof}
	Consider the family of functions 
	\[
	g_t(x) = t\theta(x) + (1-t)\sin\left(\frac{\pi}{2}\theta(x)\right)
	\]
	for all $t\in [0, 1]$. Then the map $t\mapsto g_t$ is Lipschitz continuous with respect to the supremum norm. Moreover, it follows that for any $t\in[0,1]$
	\begin{equation}
		\label{eq: family of funcs has finite limits}
		\lim_{x \rightarrow \pm\infty}g_t(x) = t(\pm 1) + (1-t)\sin\left(\frac{\pi}{2}(\pm1)\right)=  t(\pm 1) + (1-t)(\pm 1) = \pm 1. 
	\end{equation}
    It directly follows from Equation \eqref{eq: family of funcs has finite limits} that $(1-g_t(x)^2)^{\frac{1}{2}} \in C_0(\R)$ for all $t\in [0, 1]$. 
	
	So, the map
	\begin{align*}
		X\colon\ [0, 1]&\rightarrow D^*(M;S),\\
		t &\mapsto ig_t(D) + \hat{\alpha}(1-g_t(D)^2)^{\frac{1}{2}},
	\end{align*}
	is a well-defined continuous path in $D^*(M;S)$ from $i\sin\left(\frac{\pi}{2}\theta(D)\right) + \hat{\alpha}\cos\left(\frac{\pi}{2}\theta(D)\right)$ to $i\theta(D) + \hat{\alpha} (1-\theta(D)^2)^{\frac{1}{2}}$. Consider the induced path
	\begin{align*}
		H\colon\ [0, 1]&\rightarrow D^*(M;S)/C^*(N\subseteq M;S), \\
		t&\mapsto q(X(t)),
	\end{align*}
	then $H$ is a homotopy of unitaries. Indeed, writing $Y\sim Z$ if $Y$ and $Z$ differ by an element in $C^*(N\subseteq M;S)$, it follows that 
	\begin{align*}
		X^*(t)X(t)&= (-ig_t(D) + (1-g_t(D)^2)^{\frac{1}{2}}\hat{\alpha})(ig_t(D) + \hat{\alpha}(1-g_t(D)^2)^{\frac{1}{2}}) \\
		&= g_t(D)^2 + (1-g_t(D)^2)^{\frac{1}{2}}\hat{\alpha}^2(1-g_t(D)^2)^{\frac{1}{2}}\\
		&\qquad +i((1-g_t(D)^2)^{\frac{1}{2}}\hat{\alpha}g_t(D)- g_t(D)\hat{\alpha}(1-g_t(D)^2)^{\frac{1}{2}})\\
		&\sim g_t(D)^2 + (1-g_t(D)^2)^{\frac{1}{2}}\hat{\alpha}^2(1-g_t(D)^2)^{\frac{1}{2}}\\
		&\qquad +i(\hat{\alpha}(1-g_t(D)^2)^{\frac{1}{2}}g_t(D)- g_t(D)(1-g_t(D)^2)^{\frac{1}{2}}\hat{\alpha})\\
		&\sim g_t(D)^2 + (1-g_t(D)^2)^{\frac{1}{2}}\hat{\alpha}^2(1-g_t(D)^2)^{\frac{1}{2}}\\
		&= g_t(D)^2 + (1-g_t(D)^2)^{\frac{1}{2}}(\hat{\alpha}^2-1)(1-g_t(D)^2)^{\frac{1}{2}} + (1-g_t(D)^2) \\
		&\sim  g_t(D)^2 +(1-g_t(D)^2) = 1,
	\end{align*}
	where the first two equivalences follow from Lemma \ref{lem: commutator of D* with locally constant in localized algebra} and the last equivalence follows from Lemma \ref{lem: Higsons 1.2 general} since $(1-g_t(D)^2)^{\frac{1}{2}} \in C^*(M;S)$ by Equation \eqref{eq: family of funcs has finite limits} and Theorem \ref{thm: func calc thm for C^*M}. Similarly it follows that
    \[
X(t)X^*(t) \sim 1.
    \]
    So $H$ indeed is a homotopy through the unitaries from $q\left(i\sin\left(\frac{\pi}{2}\theta(D)\right) + \hat{\alpha}\cos\left(\frac{\pi}{2}\theta(D)\right)\right)$ to $q\left(i\theta(D) + \hat{\alpha} (1-\theta(D)^2)^{\frac{1}{2}}\right)$ and therefore the result follows.
\end{proof}

\begin{lemma}
	\label{lem: relating K theory classes 3}
	We have
	\[
	\left[q\left(i\theta(D) + \hat{\alpha} (1-\theta(D)^2)^{\frac{1}{2}}\right)\right]_1 = -[q(T_+)]_1
	\]
	as elements of $K_1(D^*(M;S)/C^*(N\subseteq M;S))$.
\end{lemma}

\begin{proof}
	Consider the homotopy
	\begin{align*}
		H\colon\ [0,1]&\rightarrow D^*(M;S)/C^*(N\subseteq M;S),\\
		t&\mapsto \cos\left(\frac{\pi}{2}t\right)q\left(i\theta(D) + \hat{\alpha} (1-\theta(D)^2)^{\frac{1}{2}}\right) + \sin\left(\frac{\pi}{2}t\right)q(iT_+^*).
	\end{align*}
	We claim that $H(t)$ is invertible for each $t\in [0,1]$. Indeed, abbreviating $R = i\theta(D) + \hat{\alpha} (1-\theta(D)^2)^{\frac{1}{2}}$, we compute
	\begin{align}
		\label{eq: homotopy through invertibles 1}
		&H(t)^*H(t)\notag\\
		&\qquad=\left(\cos\left(\frac{\pi}{2}t\right)q(R^*) + \sin\left(\frac{\pi}{2}t\right)q(-iT_+)\right)\left(\cos\left(\frac{\pi}{2}t\right)q(R) + \sin\left(\frac{\pi}{2}t\right)q(iT_+^*)\right)\notag\\
		&\qquad=\cos\left(\frac{\pi}{2}t\right)^2q(R^*R) + \sin\left(\frac{\pi}{2}t\right)^2q(T_+T_+^*)  +i\cos\left(\frac{\pi}{2}t\right)\sin\left(\frac{\pi}{2}t\right)q(R^*T_+^*-T_+R) \notag\\
		&\qquad= 1+i\cos\left(\frac{\pi}{2}t\right)\sin\left(\frac{\pi}{2}t\right)q(R^*T_+^*-T_+R), 
	\end{align}
	where the last equality follows as $q(T_+)$ and $q(R)$ are unitary by Corollary \ref{cor: quotients of T+ are unitary} and Lemma \ref{lem: relating K theory classes 2}, respectively. We claim that the second term in \eqref{eq: homotopy through invertibles 1} is positive. 
	
	Writing out the operator in the quotient, we see that 
	\begin{multline}
		 R^*T_+^*-T_+R=\\
		\label{eq: intermediate operator 1}
		-i\left(\theta(D)^2(1-\hat{\alpha}^2)^{\frac{1}{2}} + (1-\theta(D)^2)^{\frac{1}{2}}\hat{\alpha}^2+\hat{\alpha}^2 (1-\theta(D)^2)^{\frac{1}{2}} + (1-\hat{\alpha}^2)^{\frac{1}{2}}\theta(D)^2\right)\\
		+[\hat{\alpha}, \theta(D)]+ (1-\theta(D)^2)^{\frac{1}{2}}\hat{\alpha}\theta(D)(1-\hat{\alpha}^2)^{\frac{1}{2}}  -(1-\hat{\alpha}^2)^{\frac{1}{2}}\theta(D)\hat{\alpha} (1-\theta(D)^2)^{\frac{1}{2}}. 
	\end{multline}
	Again, writing $X\sim Y$ if $X$ and $Y$ differ by an element in $C^*(N\subseteq M;S)$, it follows that the terms in 
    the second line on the right 
    are equivalent to 0. For the first term, this follows from Lemma \ref{lem: commutator of D* with locally constant in localized algebra} and for the last two terms, this follows from Lemma \ref{lem: Higsons 1.2 general}.
    The first line on the right hand side of 
    \eqref{eq: intermediate operator 1}  is easily seen to be equal to 
	\begin{align}
		\label{eq: intermediate operator 3}
		&-2i\left((1-\hat{\alpha}^2)^{\frac{1}{2}}+(1-\theta(D)^2)^{\frac{1}{2}}\right)\\
		\label{eq: intermediate operator 4}
		&-i\left((\theta(D)^2-1)(1-\hat{\alpha}^2)^{\frac{1}{2}}+ (1-\theta(D)^2)^{\frac{1}{2}}(\hat{\alpha}^2-1)\right)\\
		\label{eq: intermediate operator 5}
		&-i\left((\hat{\alpha}^2-1) (1-\theta(D)^2)^{\frac{1}{2}}+(1-\hat{\alpha}^2)^{\frac{1}{2}}(\theta(D)^2-1)\right),
	\end{align}
	where again the terms in \eqref{eq: intermediate operator 4} and \eqref{eq: intermediate operator 5} are equivalent to $0$ by Lemma \ref{lem: Higsons 1.2 general}. We conclude that 
	\begin{equation*}
		R^*T_+^*-T_+R \sim -2i\left((1-\hat{\alpha}^2)^{\frac{1}{2}}+(1-\theta(D)^2)^{\frac{1}{2}}\right).
	\end{equation*}
	As a result, \eqref{eq: homotopy through invertibles 1} becomes
    \[
		H(t)^*H(t)
        = 
        1+2\cos\left(\frac{\pi}{2}t\right)\sin\left(\frac{\pi}{2}t\right)q((1-\hat{\alpha}^2)^{\frac{1}{2}}+(1-\theta(D)^2)^{\frac{1}{2}}) \geq 1.        \]
    It follows that $H(t)$ is invertible for any  $t \in [0,1]$. 
	
	Since the map $\omega$ that sends an invertible element $a$ to its unitary part $a|a|^{-1}$ is continuous and fixes the unitaries (see for example \cite[Prop. 2.1.8.ii]{Rordam2000Ktheory}) the map
	\begin{align*}
		H'\colon\ [0, 1]&\rightarrow D^*(M;S)/C^*(N\subseteq M;S),\\
		t&\mapsto \omega(H(t)),
	\end{align*}
	is a homotopy of unitaries from $\omega(H(0)) = \omega(q(R)) = q(R)$ to $\omega(H(1)) = \omega(q(iT_+^*)) = q(iT_+^*)$. The result follows as $[q(iT_+^*)]_1 = -[q(T_+)]_1$.
\end{proof}

\begin{proof}[Proof of Theorem \ref{thm: PMT for product}]
	By Corollary \ref{cor: image of D_N class under paschke duality}, $[D_N]$ is mapped to $-[q(T_+)]_1 \in K_1(Q^*(N;S))$ under Paschke duality. So since Diagram \eqref{diag: the central diagram} commutes it follows that 
	\[
	\Phi(\Ind(D_N)) = K_0(i)^{-1}\circ\partial (-[q(T_+)]_1).
 	\]
 	Moreover,
  %
 	by naturality of the boundary map and Diagram \eqref{diag: relating C^* extensions} it follows that
	\begin{align*}
		\Phi(\Ind(D_N)) &= K_0(i)^{-1}\circ\partial (-[q(T_+)]_1) =-K_0(i)^{-1}\circ\partial \circ K_1(\ol{j_N}) ([q(T_+)]_1)\\
		&=-\partial\circ K_1(\ol{j_M})([q(T_+)]_1) = -\partial(-[q(U_+)]_1) = \Ind(D;N),
	\end{align*}
	where the second last equality follows from Proposition \ref{prop: relating K theory classes}.
\end{proof}

\section{Proof of the general case: reduction to a product manifold}
\label{section: proof general case}

The goal of this section is to show that under the conditions stated in Theorem \ref{thm: PMT}, we can derive the partitioned manifold index theorem for a general partitioned manifold from the special case of a product manifold treated in Section \ref{section: proof special case}, Theorem \ref{thm: PMT for product}. The strategy will be to construct a transition manifold, which is a partitioned manifold $M'$ that looks like our original manifold $M$ on one side and like a product on the other side. This is inspired by the proof demonstrated by Higson \cite[Thm. 1.5]{higson1991cobordisminvariance}. We are careful about the details of various gluing constructions, to avoid issues due to noncompactness of $N$. (See e.g.\ Corollary \ref{cor: g' is complete} on completeness of a Riemannian metric obtained from a gluing construction.)

\subsection{Half-isomorphisms and the transition map}
\label{subsec: half-isos and the transition map}
In this section, we introduce the notion of a half-isomorphism between partitioned manifolds. As the name suggests, the idea is that two partitioned manifolds are half-isomorphic if they are isomorphic on one of their halves. The concept of a half-isomorphism is motivated by \cite[Lem. 3.1]{higson1991cobordisminvariance}.

\begin{definition}
	\label{def: half iso}
	For $i = 1, 2$, let $(M_i, g_i, S_i, h_i)$ be two sets of complete Riemannian manifolds $(M_i, g_i)$ with Hermitian vector bundles $(S_i, h_i)$. Let $D_i$ be differential operators acting on the smooth sections of $S_i$. Suppose that both $M_i$ are partitioned by a simple hypersurface $N_i\subseteq M_i$ and have been given a partition labeling $M_{i\pm}$.
	\begin{enumerate}
		\item \label{item: half iso} A \emph{half-isomorphism} is a pair of opens $O_i \supseteq M_{i+}$ such that the function $\phi_{+}$ as in Definition \ref{def: simple hypersurface} can be chosen such that $\supp(\phi_{+})\subseteq O_i$, and a pair of maps $(\phi, u)$ with $\phi\colon O_1\rightarrow O_2$ an orientation-preserving Riemannian isometry such that $\phi(N_1) = N_2$ and $\phi(M_{1+}) = M_{2+}$, and $u\colon \restr{S_1}{O_1}\rightarrow \restr{S_2}{O_2}$ a fiberwise unitary vector bundle morphism covering $\phi$. We will usually suppress the opens $O_i$ in the notation of a half-isomorphism.
		\item \label{item: uniform half iso} A \emph{uniform half-isomorphism} is a half-isomorphism $(\phi, u)$ such that the map $\phi$ extends to a homeomorphism $\phi\colon \ol{O}_1\rightarrow \ol{O}_2$ that is a uniform equivalence, when $\ol{O}_i$ is considered as a metric space with the distance function $\restr{d_{M_i}}{\ol{O}_i}$.
		\item  \label{item: compatible half iso} A half-isomorphism $(\phi, u)$ is \emph{compatible with the differential operators} if 
		\begin{equation}
			\label{eq: compatibility with diff ops}
			u\circ D_1\sigma \circ \phi^{-1} = D_2(u\circ\sigma\circ\phi^{-1})
		\end{equation}
		for all $\sigma \in C^{\infty}(O_1;S_1)$.
	\end{enumerate}
\end{definition}  

Our goal in this subsection is to show that the following hold:
\begin{enumerate}
	\item If $M_1$ and $M_2$ are partitioned manifolds with a uniform half-isomorphism between them, then there exists an induced isomorphism \[
	\cT \colon K_p(C^*(N_2\subseteq M_2))\xrightarrow{\sim} K_p(C^*(N_1\subseteq M_1)).
	\]
	\item If $\Phi_1$ and $\Phi_2$ are defined as in \eqref{diag: identification maps} for the partitioned manifolds $M_1$ and $M_2$ respectively, then the following diagram commutes:
	\begin{equation}
		\label{diag: transition diagram}
		\begin{tikzcd}
			{K_0(C^*(N_1))} & {K_0(C^*(N_1\subseteq M_1))} \\
			{K_0(C^*(N_2))} & {K_0(C^*(N_2\subseteq M_2))}
			\arrow["{\Phi_1}", from=1-1, to=1-2]
			\arrow["({\restr{\phi}{N_1})_*}"', from=1-1, to=2-1]
			\arrow["{\cT}"', from=2-2, to=1-2]
			\arrow["{\Phi_2}", from=2-1, to=2-2]
		\end{tikzcd}
	\end{equation}
\end{enumerate}

For notational purposes, denote $H_{i+}:= L^2(M_{i+};\restr{S_i}{M_{i+}})$ and denote the inclusion by $W_i\colon H_{i+}\hookrightarrow L^2(M_i;S_i)$, which is an isometry. Similarly, for an unlabeled partitioned manifold $N\subseteq M$ we also write $H_{+}:= L^2(M_{+};\restr{S}{M_{+}})$ and denote the inclusion by $W\colon H_+\hookrightarrow L^2(M_{+};\restr{S}{M_{+}})$. 
A half-isomorphism induces a unitary map
$V\colon L^2(O_1;\restr{S_1}{O_1})\rightarrow L^2(O_2;\restr{S_2}{O_2})$ defined by $\sigma\mapsto u\circ \sigma\circ\phi^{-1}$, which restricts to a unitary map $\Gamma\colon H_{1+}\rightarrow H_{2+}$.  
\begin{lemma}
	\label{lem: half iso map on base preserves distance to hypersurface}
	For $i = 1, 2$, let $N_i\subseteq M_i$ be two partitioned manifolds and let $\phi$ be a map as in part \ref{item: half iso} of Definition $\ref{def: half iso}$. Then for any $p \in M_{1+}$,
	\[
	d_{M_1}(p, N_1) = d_{M_2}(\phi(p), N_2).
	\]
\end{lemma}
\begin{proof}
	Let $p \in M_{1+}$ and $n \in N_1$ be arbitrary and let $\gamma$ be any piecewise smooth curve connecting $p$ to $n$. Then $\gamma^{-1}(N)\subseteq [0, 1]$ is closed, hence compact. Let $t_0 = \min(\gamma^{-1}(N))$. Then $\gamma(t_0)\in N$ and $\restr{\gamma}{[0, t_0]}$ is a curve through $M_{1+}$. It follows that 
	\begin{align*}
		L_{g_1}(\gamma) &\geq L_{g_1}(\restr{\gamma}{[0, t_0]}) = L_{g_2}(\restr{\phi\circ \gamma}{[0, t_0]})\geq d_{M_2}(\phi(p), \phi\circ \gamma(t_0)) \geq d_{M_2}(\phi(p), N_2). 
	\end{align*}
	Taking infima over all curves $\gamma$ and then over all $n \in N_1$, we conclude that 
	\[
	d_{M_1}(p, N_1)\geq d_{M_2}(\phi(p), N_2).
	\]
	The equality 
	\[
	d_{M_1}(p, N_1)= d_{M_2}(\phi(p), N_2)
	\]
	then follows by symmetry under swapping $M_1$ with $M_2$ (and replacing $\phi$ by $ \phi^{-1}$). 
\end{proof}

\begin{lemma}
	\label{lem: Gamma covers uniform equivalence}
	Let $(\phi, u)$ be a uniform half-isomorphism between two partitioned manifolds $N_i\subseteq M_i$ with Hermitian bundles $S_i$, and let $\Gamma\colon H_{1+}\rightarrow H_{2+}$ be the induced unitary map. Then $\Gamma$ uniformly covers the uniform equivalence $\restr{\phi}{M_{1+}}$. Moreover, 
	\[
	\Ad_{\Gamma}(C^*(N_1\subseteq M_{1+}))= C^*(N_2\subseteq M_{2+}).
	\]
\end{lemma}
\begin{proof}
	First note that $\restr{\phi}{M_{1+}}$ is a uniform equivalence as it is a restriction of the uniform equivalence $\phi\colon \ol{O}_1\rightarrow \ol{O}_2$. Denote by $\rho_{i+}\colon C_0(M_{i+})\rightarrow H_{i+}$ the corresponding representations by multiplication operators.
	It follows from the definitions that $\Gamma$ uniformly covers $\phi$ and therefore $\Ad_{\Gamma}$ preserves finite propagation, local compactness and pseudolocality. 

	It follows from Lemma \ref{lem: grading operator} that $\Ad_{\Gamma}$ maps operators supported near $N_1$ to operators supported near $N_2$. Since $\Gamma$ also uniformly covers $\phi$, it follows that $\Ad_{\Gamma}$ maps the localized Roe algebra $C^*(N_1\subseteq M_{1+})$ into the localized Roe algebra $C^*(N_2\subseteq M_{2+})$. After reversing the roles of $M_1$ and $M_2$, $\Ad_{\Gamma^*}$ does the same. Therefore it follows that $\Ad_{\Gamma}(C^*(N_1\subseteq M_{1+}))= C^*(N_2\subseteq M_{2+})$.
\end{proof}

Using Proposition \ref{prop: localized roe alg isomorphic to roe alg of subset} twice yields
isomorphisms 
\begin{equation}
	\label{eq: half loc alg to entire loc alg}
	K_p(C^*(N_i\subseteq M_i, d_{M_i}))\cong K_p(C^*(N_i, \restr{d_{M_i}}{N}))
	\cong  K_p(C^*(N_i\subseteq M_{i+}, \restr{d_{M_i}}{M_{i+}})).
\end{equation}
It is not hard to show that the isometric inclusions $W_i$ uniformly cover the uniform maps $M_{i+}\hookrightarrow M_i$ and that they also preserve the localized Roe algebras. It can then be shown that the map in \eqref{eq: half loc alg to entire loc alg} going backwards is just $K_p(\Ad_{W_i})$.

Suppose that the half-isomorphism is uniform.
%
%
%
Then it follows from Lemma \ref{lem: Gamma covers uniform equivalence} that the right vertical map in the following diagram is a well-defined isomorphism:
\begin{equation}
	\label{diag: def transition map}
	\begin{tikzcd}
		{K_p(C^*(N_2\subseteq M_2, d_{M_2};S_2))} & {K_p(C^*(N_2\subseteq M_{2+}, \restr{d_{M_2}}{M_{2+}};S_2))} \\
		{K_p(C^*(N_1\subseteq M_1, d_{M_1};S_1))} & {K_p(C^*(N_1\subseteq M_{1+}, \restr{d_{M_1}}{M_{1+}};S_1))}
		\arrow["\sim", from=1-1, to=1-2]
		\arrow["\cT", dotted, from=1-1, to=2-1]
		\arrow["{K_p(\Ad_{\Gamma^*})}", from=1-2, to=2-2]
		\arrow["\sim"', from=2-2, to=2-1]
	\end{tikzcd}
\end{equation} 
In the above diagram, we included the distance functions to emphasize that we use the restricted distance functions on the positive parts of the partitions.

\begin{definition}
	\label{def: transition map}
	Let $(\phi, u)$ be a uniform half-isomorphism between two partitioned manifolds $N_i\subseteq M_i$. Then the map $\cT$ as constructed in Diagram \eqref{diag: def transition map} is the \emph{transition map} induced by the uniform half-isomorphism $(\phi, u)$.
\end{definition}

\begin{remark}
	\label{rem: transition map is isomorphism}
	The transition map $\cT$ induced by a uniform half-isomorphism is an isomorphism by construction.
\end{remark}

Our final goal for this section is to show that Diagram \eqref{diag: transition diagram} commutes.
\begin{lemma}
	\label{lem: conjugation with Gamma fits into comm diagram}
	Let $(\phi, u)$ be a uniform half-isomorphism between two partitioned manifolds $N_i\subseteq M_i$ with Hermitian bundles $S_i$. Then 
	\[
	\restr{\phi}{N_1}\colon (N_1, \restr{d_{M_1}}{N_1})\rightarrow (N_2, \restr{d_{M_2}}{N_2})
	\]
	is a uniform equivalence and the following diagram commutes:
	\begin{equation}
		\label{diag: preparation transition lemma}
		\begin{tikzcd}
			{K_p(C^*(N_1, \restr{d_{M_1}}{N_1}))} & {K_p(C^*(N_2, \restr{d_{M_2}}{N_2}))} \\
			{K_p(C^*(N_1\subseteq M_{1+}))} & {K_p(C^*(N_2\subseteq M_{2+}))}
			\arrow["{(\restr{\phi}{N_1})_*}", from=1-1, to=1-2]
			\arrow["\sim"' {anchor = south, rotate = 90}, from=1-1, to=2-1]
			\arrow["\sim"{anchor = south, rotate = 90}, from=1-2, to=2-2]
			\arrow["{K_p(\Ad_{\Gamma})}"', from=2-1, to=2-2]
		\end{tikzcd}
	\end{equation}
\end{lemma}
\begin{proof}
	First note that $\restr{\phi}{N_1}\colon N_1\rightarrow N_2$ is a uniform equivalence as it is a restriction of the uniform equivalence $\phi\colon \ol{O}_1\rightarrow \ol{O}_2$. For any positive integer $k$, write $Y^i_k = \ol{B_k(N_i)}\subseteq M_{i+}$. It follows from Lemma \ref{lem: half iso map on base preserves distance to hypersurface} that $\phi(Y_k^1)= Y_k^2$ for each $k$ and therefore $\phi_k:= \restr{\phi}{Y_k^1}$ is a uniform equivalence from $Y_k^1$ to $Y_k^2$. Similarly to Lemma \ref{lem: Gamma covers uniform equivalence}, it follows that $\phi_k$ is covered by the restriction 
	\[
	\Gamma_k\colon L^2(Y^1_k;\restr{S_1}{Y^1_k})\rightarrow L^2(Y^2_k; \restr{S_2}{Y^2_k})
	\]
	of $\Gamma$. Therefore, it follows that for each $k$, the following diagram commutes: 
	\[\begin{tikzcd}
		{K_p(C^*(N_1, \restr{d_{M_1}}{N_1}))} & {K_p(C^*(N_2, \restr{d_{M_2}}{N_2}))} \\
		{K_p(C^*(Y^1_k;\restr{d_{M_1}}{Y^1_k}))} & {K_p(C^*(Y^2_k;\restr{d_{M_2}}{Y^2_k}))} \\
		{K_p(C^*(N_1\subseteq M_{1+}))} & {K_p(C^*(N_2\subseteq M_{2+}))}
		\arrow["{(\restr{\phi}{N_1})_*}", from=1-1, to=1-2]
		\arrow["\sim"' {anchor = south, rotate = 90}, from=1-1, to=2-1]
		\arrow["\sim"{anchor = south, rotate = 90}, from=1-2, to=2-2]
		\arrow["{(\phi_k)_*}", from=2-1, to=2-2]
		\arrow["\sim"', sloped, from=2-1, to=3-1]
		\arrow["\sim"', sloped, from=2-2, to=3-2]
		\arrow["{K_p(\Ad_{\Gamma})}", from=3-1, to=3-2]
	\end{tikzcd}\]
	The upper square commutes since passing from coarse spaces to the $K$-theory of the Roe algebra is functorial. The lower square commutes just by functoriality of $K$-theory as $\Gamma_k$ covers $\phi_k$ so that $(\phi_k)_* = K_p(\Ad_{\Gamma_k})$. The result now follows from Proposition \ref{prop: localized roe alg isomorphic to roe alg of subset}.
\end{proof}

\begin{proposition}
	\label{prop: transition diagram commutes}
	Let $(\phi, u)$ be a uniform half-isomorphism between two partitioned manifolds $N_i\subseteq M_i$ with Hermitian bundles $S_i$ and let $\cT$ be the transition map induced by the half-isomorphism. Then Diagram \eqref{diag: transition diagram} commutes.
\end{proposition}
\begin{proof}
	Inserting the definition of $\cT$ into the diagram, we see that it is left to show that the following diagram commutes:
	\[\begin{tikzcd}
		{K_p(C^*(N_1, d_{N_1}))} & {K_p(C^*(N_2, d_{N_2}))} \\
		{K_p(C^*(N_1, \restr{d_{M_1}}{N_1}))} & {K_p(C^*(N_2, \restr{d_{M_2}}{N_2}))} \\
		{K_p(C^*(N_1\subseteq M_1))} & {K_p(C^*(N_2\subseteq M_2))} \\
		{K_p(C^*(N_1\subseteq M_{1+}))} & {K_p(C^*(N_2\subseteq M_{2+}))}
		\arrow["{(\restr{\phi}{N_1})_*}", from=1-1, to=1-2]
		\arrow["{\id_{N_1, *}}"', from=1-1, to=2-1]
		\arrow["{\id_{N_2, *}}"', from=1-2, to=2-2]
		\arrow["{(\restr{\phi}{N_1})_*}", dotted, from=2-1, to=2-2]
		\arrow["\sim"', sloped, from=2-1, to=3-1]
		\arrow["\sim"{description}, sloped, bend right = 90, dashed, from=2-1, to=4-1]
		\arrow["\sim"', sloped, from=2-2, to=3-2]
		\arrow["\sim"{description}, sloped, bend left = 90, dashed, from=2-2, to=4-2]
		\arrow["\sim"', sloped, from=3-1, to=4-1]
		\arrow["\sim"', sloped, from=3-2, to=4-2]
		\arrow["{K_p(\Ad_{\Gamma})}", from=4-1, to=4-2]
	\end{tikzcd}\]
	This is a combination of different results. Commutativity of the upper square follows directly since passing from coarse spaces to $K$-theory of Roe-algebras is functorial. Commutativity of the outer cells follows by Equation \eqref{eq: half loc alg to entire loc alg}. Finally, commutativity of the square with the dashed and dotted arrows follows from Lemma \ref{lem: conjugation with Gamma fits into comm diagram}.
\end{proof}

\subsection{Half-isomorphisms and the partitioned index}
\label{subsec: half-isos and the partitioned index}
In this section, we assume that we have two complete partitioned manifolds $N_i \subseteq M_i$ for $i = 1, 2$, equipped with Hermitian vector bundles $S_i\rightarrow M_i$ and differential operators $D_i$ on $S_i$ that are elliptic and symmetric and have finite propagation speed. The goal is to show that if there is a uniform half-isomorphisms between these partitioned manifolds that is compatible with the differential operators, then the induced transition map $\cT$ maps the partitioned index $\Ind(D_2;N_2)$ to the partitioned index $\Ind(D_1;N_1)$.

So for the rest of this section, let $(\phi, u)$ be a uniform half-isomorphism between these partitioned manifolds that is compatible with the differential operators. Recall that this induces unitary maps 
\[
V\colon L^2(O_1;\restr{S_1}{O_1}) \rightarrow L^2(O_2;\restr{S_2}{O_2}) \qquad\text{ and }\qquad \Gamma\colon H_{1+}\rightarrow H_{2+}.
\]
Denote by $\tilde{\Gamma}$ the partial isometry from $L^2(M_1;S_1)$ to $L^2(M_2;S_2)$ that arises from extending $\Gamma$ by 0 on $H_{1+}^{\bot}$. Then
\[
\tilde{\Gamma}^*\tilde{\Gamma} =  P_{1+} \qquad\text{ and }\qquad \tilde{\Gamma}\tilde{\Gamma}^* =  P_{2+}.
\]
Similarly, denote the partial isometry obtained by extending  $V$ by 0 on $L^2(O_1; \restr{S_1}{O_1})^{\bot}$ by $\tilde{V}$. Denote the projections of $L^2(M_i;S_i)$ onto $L^2(O_i;\restr{S_i}{O_i})$ by $P_{O_i} = \rho_i(\chi_{O_i})$.
Then
\[
\tilde{V}^*\tilde{V} = P_{O_1} \qquad\text{ and }\qquad \tilde{V}\tilde{V}^* =  P_{O_2}.
\]
Note that the two operators are related via $\tilde{\Gamma} = P_{2+}\tilde{V} = \tilde{V}P_{1+}$.

\begin{lemma}
	\label{lem: adjungation with Gamma circ preserves Roe alg props}
	The following inclusions hold:
	\[
	\Ad_{\tilde{V}^*}(C^*(M_2))\subseteq C^*(M_1)\quad \text{ and }\quad \Ad_{\tilde{V}^*}(C^*(N_2\subseteq M_2))\subseteq C^*(N_1\subseteq M_1).
	\]
\end{lemma}
\begin{proof}
	Recall that by definition, $\phi\colon \ol{O}_1\rightarrow \ol{O}_2$ is a uniform equivalence with respect to the subspace metrics. 
	It is straightforward to deduce that  $\Ad_{\tilde{V}^*}$ preserves the properties of finite propagation and being supported near the partitioning hypersurface. 
	
	It is left to show that $\Ad_{\tilde{V}^*}$ preserves local compactness. First, note that if $f_1 \in C_0(M_1)$, then $f_1\circ \phi^{-1} \in C_0(\ol{O}_2)$ as $\phi$ is a homeomorphism and thus $f_1\circ \phi^{-1}$ extends to an function $\tilde{f}_2 \in C_0(M_2)$ (apply the Tietze extension theorem to the one-point compactification of $M_2$). Note that 
	\[
	\Ad_{\tilde{V}^*}(\rho_2(\tilde{f}_2)T) = \rho_1(f_1)	\Ad_{\tilde{V}^*}(T),
	\]
	and
	\[
	\Ad_{\tilde{V}^*}(T\rho_2(\tilde{f}_2)) = \Ad_{\tilde{V}^*}(T)\rho_1(f_1),
	\]
	which implies that $\Ad_{\tilde{V}^*}$ indeed preserves local compactness.
\end{proof}

In the following lemma, we will translate the relation between the differential operators $D_1$ and $D_2$ given by the half-isomorphism to a relation between the induced unbounded operators. To distinguish the original operators from their operator-theoretic closures, we will distinguish between $D_i$ and $\ol{D_i}$ in the next result.

\begin{lemma}
	\label{lem: Gamma tilde intertwines differential operators}
	The following equality holds: 
	\begin{equation*}
		q(P_{1+}(\ol{D}_1+i)^{-1}P_{1+}) = q(\tilde{\Gamma}^*(\ol{D}_2+i)^{-1}\tilde{\Gamma}) \in C^*(M_1)/C^*(N_1\subseteq M_1).
	\end{equation*}
\end{lemma}
\begin{proof}
	Let $O_i\supseteq M_{i+}$ be the opens from part \ref{item: half iso} of Definition \ref{def: half iso} and let $\phi_{2+}$ be a function as in Definition \ref{def: simple hypersurface} such that $\supp(\phi_{2+})\subseteq O_2$. Put $\phi_{1+} = \phi_{2+}\circ \phi$, which extends to a function as in Definition \ref{def: simple hypersurface} when we declare it to be 0 outside of $O_1$. Since the half-isomorphism is compatible with the differential operators, it follows that 
	\[
	D_1\tilde{V}^*\rho(\phi_{2+})\sigma = \tilde{V}^*D_2\rho(\phi_{2+})\sigma = \tilde{V}^*\rho(\phi_{2+})D_2\sigma + \tilde{V}^*\sigma_{D_2}(d\phi_{2+})\sigma
	\]
	for all $\sigma \in C_c^{\infty}(M_2;S_2)$. This gives an equality as unbounded operators on the domain $C_c^{\infty}(M_2;S_2) = \dom(D_2)$. We will first extend this equality to $\dom(\ol{D}_2)$. 
	Note that the last term extends to a bounded operator as $\phi_{2+}$ has bounded gradient and $D_2$ is assumed to have finite propagation speed. If $\xi \in \dom(\ol{D}_2)$, then there exists a sequence $\{\sigma_n\}_{n\in \N}$ in $C_c^{\infty}(M_2;S_2)$ such that $\sigma_n\rightarrow \xi$ and $D_2\sigma_n\rightarrow \ol{D}_2\xi$. It follows that 
	\begin{align}
		\label{eq: computation D_2}
		\lim_{n\rightarrow \infty} D_1\tilde{V}^*\rho(\phi_{2+})\sigma_n &= \lim_{n\rightarrow \infty}\tilde{V}^*\rho(\phi_{2+})D_2\sigma_n + \tilde{V}^*\sigma_{D_2}(d\phi_{2+})\sigma_n \notag\\
		&= \tilde{V}^*\rho(\phi_{2+})\ol{D}_2\xi + \tilde{V}^*\sigma_{D_2}(d\phi_{2+})\xi.
	\end{align}
	Since this limit exists and $\tilde{V}^*\rho(\phi_{2+})\sigma_n\rightarrow \tilde{V}^*\rho(\phi_{2+})\xi$, we conclude that $\tilde{V}^*\rho(\phi_{2+})\xi \in \dom(\ol{D}_1)$ and that we have an equality of unbounded operators
	\begin{equation}
		\label{eq: unbounded operator 1}
		\ol{D}_1\tilde{V}^*\rho(\phi_{2+}) = \tilde{V}^*\rho(\phi_{2+})\ol{D}_2 + \tilde{V}^*\sigma_{D_2}(d\phi_{2+})
	\end{equation}
	on the domain $\dom(\ol{D}_2)$. Adding $i\tilde{V}^*\rho(\phi_{2+})$ to both sides and multiplying on the left with $(\ol{D}_1+i)^{-1}$ and on the right with $(\ol{D}_2+i)^{-1}$, which is well-defined as the operator in \eqref{eq: unbounded operator 1} has domain $\dom(\ol{D}_2)$, we obtain the equality 
	\begin{equation}
		\label{eq: operator equality}
		\tilde{V}^*\rho(\phi_{2+})(\ol{D}_2+i)^{-1} = (\ol{D}_1+i)^{-1}\tilde{V}^*\rho(\phi_{2+}) + (\ol{D}_1+i)^{-1}\tilde{V}^*\sigma_{D_2}(d\phi_{2+})(\ol{D}_2+i)^{-1}.
	\end{equation}
	Using that $\tilde{V}\rho(\phi_{1+}) = \rho(\phi_{2+})\tilde{V}$ since $\supp(\phi_{i+})\subseteq O_i$ and multiplying on the right by $\tilde{V}$, we then obtain
	\begin{multline}
		\label{eq: operator equality 2}
		\rho(\phi_{1+})\tilde{V}^*(\ol{D}_2+i)^{-1}\tilde{V} =\\
		(\ol{D}_1+i)^{-1}\rho(\phi_{1+})\tilde{V}^*\tilde{V} + (\ol{D}_1+i)^{-1}\tilde{V}^*\sigma_{D_2}(d\phi_{2+})(\ol{D}_2+i)^{-1}\tilde{V},
	\end{multline}
	where the last term is an element of $C^*(N_1\subseteq M_1)$ by Lemma \ref{lem: adjungation with Gamma circ preserves Roe alg props} since $\sigma_{D_2}(d\phi_{2+})$ is an element of $D^*(M_2)$ which is supported near $N_2$. 
	
	Finally, write $X\sim Y$ if $X$ and $Y$ differ by an element in $C^*(N_1\subseteq M_1)$. It follows that 
	\begin{align*}
		\tilde{\Gamma}^*(\ol{D}_2+i)^{-1}\tilde{\Gamma} &= P_{1+}\tilde{V}^*(\ol{D}_2+i)^{-1}\tilde{V}P_{1+}\\
		& \sim P_{1+}\tilde{V}^*(\ol{D}_2+i)^{-1}\tilde{V} \\
		&\sim \rho_1(\phi_{1+})\tilde{V}^*(\ol{D}_2+i)^{-1}\tilde{V} \\
		&\sim (\ol{D}_1+i)^{-1}\rho(\phi_{1+}) \\
		&\sim  (\ol{D}_1+i)^{-1}P_{1+} \\
		&\sim P_{1+}(\ol{D}_1+i)^{-1}P_{1+},
	\end{align*}
	where the first and last equivalence follow from Lemma \ref{lem: commutator with projection lands in localized algebra}, the second and second-last equivalence follow from Lemma \ref{lem: Higsons 1.2 general} and the third equivalence follows from Equation \eqref{eq: operator equality 2}. The result follows.
\end{proof}

We can now show that the transition map induced by a uniform half-isomorphism that is compatible with the differential operators indeed maps the partitioned index of $D_2$ to the partitioned index of $D_1$.

\begin{proposition}
	\label{prop: iso on localized Roe algs preserves partitioned index}
	For $i= 1, 2$, let $N_i\subseteq M_i$ be two partitioned manifolds with Hermitian bundles $S_i$ and symmetric, elliptic differential operators $D_i$ with finite propagation speed. Let $(\phi, u)$ be a uniform half-isomorphism between them that is compatible with the differential operators and let $\cT$ be the induced transition map. Then 
	\[
	\cT(\Ind(D_2;N_2)) = \Ind(D_1;N_1).
	\]
\end{proposition}
\begin{proof}
	Denote by $\gamma_i\colon C^*(M_i)\rightarrow C^*(M_i)/C^*(N_i\subseteq M_i)$ the $*$-homomorphisms from Lemma \ref{lem: construction gamma},  given by $\gamma_i(T) = q(P_{i+}TP_{i+})$, with $P_{i+}$ the projection onto $H_{i+}$. Recall from Proposition \ref{prop: Roe homomorphism} that $\Ind(D_i;N_i) = \partial\circ K_1(\gamma_i)([U_i]_1)$ with $U_i$ the Cayley transform of $D_i$. It follows that
	\begin{align*}
		\Ind(D_1; N_1) &= \partial\circ K_1(\gamma_1)([U_1]_1) = \partial([1+q(P_{1+}(U_1-1)P_{1+})]_1)
	\end{align*}
	with 
	\begin{align*}
		q(P_{1+}(U_1-1)P_{1+})&= -2iq(P_{1+}(D_1+i)^{-1}P_{1+})\\
		&= -2i	q(\tilde{\Gamma}^*(D_2+i)^{-1}\tilde{\Gamma}) = q(\tilde{\Gamma}^*(U_2-1)\tilde{\Gamma}),
	\end{align*}
	where the second equality follows from Lemma \ref{lem: Gamma tilde intertwines differential operators}. By using Lemma \ref{lem: Gamma covers uniform equivalence} and the fact that $W_i$ uniformly covers the inclusion $M_{i+}\hookrightarrow M_i$, we see that the following diagram with short exact rows is commutative:
	\begin{equation}
		\label{diag: naturality diagram}
		\begin{tikzcd}
			0 & {C^*(N_2\subseteq M_{2})} & {D^*(M_{2})} & {Q^*(N_2\subseteq M_{2})} & 0 \\
			0 & {C^*(N_2\subseteq M_{2+})} & {D^*(M_{2+})} & {Q^*(N_2\subseteq M_{2+})} & 0 \\
			0 & {C^*(N_1\subseteq M_{1+})} & {D^*(M_{1+})} & {Q^*(N_1\subseteq M_{1+})} & 0 \\
			0 & {C^*(N_1\subseteq M_{1})} & {D^*(M_{1})} & {Q^*(N_1\subseteq M_{1})} & 0
			\arrow[from=1-1, to=1-2]
			\arrow[from=1-2, to=1-3]
			\arrow[from=1-3, to=1-4]
			\arrow[from=1-4, to=1-5]
			\arrow[from=2-1, to=2-2]
			\arrow["{\Ad_{W_1}}"{description}, from=2-2, to=1-2]
			\arrow[from=2-2, to=2-3]
			\arrow["{\Ad_{W_1}}"{description}, from=2-3, to=1-3]
			\arrow[from=2-3, to=2-4]
			\arrow["{\ol{\Ad_{W_1}}}"{description}, from=2-4, to=1-4]
			\arrow[from=2-4, to=2-5]
			\arrow[from=3-1, to=3-2]
			\arrow["{\Ad_{\Gamma}}"{description}, from=3-2, to=2-2]
			\arrow[from=3-2, to=3-3]
			\arrow["{\Ad_{W_2}}"{description}, from=3-2, to=4-2]
			\arrow["{\Ad_{\Gamma}}"{description}, from=3-3, to=2-3]
			\arrow[from=3-3, to=3-4]
			\arrow["{\Ad_{W_2}}"{description}, from=3-3, to=4-3]
			\arrow["{\ol{\Ad_{\Gamma}}}"{description}, from=3-4, to=2-4]
			\arrow[from=3-4, to=3-5]
			\arrow["{\ol{\Ad_{W_2}}}"{description}, from=3-4, to=4-4]
			\arrow[from=4-1, to=4-2]
			\arrow[from=4-2, to=4-3]
			\arrow[from=4-3, to=4-4]
			\arrow[from=4-4, to=4-5]
		\end{tikzcd}
	\end{equation}
	It follows that
	\begin{align*}
		\Ind(D_1; N_1) &= \partial([1+q(\tilde{\Gamma}^*(U_2-1)\tilde{\Gamma})]_1) \\
		&=\partial\circ K_1(\ol{\Ad_{W_1}})\circ K_1(\ol{\Ad_{\Gamma^*}})([q(W_2^*U_2W_2)]_1) \\
		&= K_0(\Ad_{W_1})\circ K_0(\Ad_{\Gamma^*})\circ\partial([q(W_2^*U_2W_2)]_1)\\
		&= \cT\circ K_0(\Ad_{W_2}) \circ\partial([q(W_2^*U_2W_2)]_1) \\
		&=  \cT\circ \partial\circ  K_1(\ol{\Ad_{W_2}}) ([q(W_2^*U_2W_2)]_1) \\
		&=\cT\circ\partial([1 + q(W_2W_2^*(U_2-1)W_2W_2^*)]_1) = \cT(\Ind(D_2;N_2)),
	\end{align*}
	where the third and fifth equality follow from Diagram \eqref{diag: naturality diagram} by naturality of the index map.
\end{proof}

\subsection{Gluing manifolds and metrics}
\label{subsec: gluing on M'}
In this subsection, it will be demonstrated that starting with a partitioned manifold $N\subseteq M$, we can build a ``transition manifold" $M'$. The transition manifold $M'$ will be a partitioned manifold that (loosely speaking) looks like $M_+$ on one side of the partitioning hypersurface and like $(-\infty,0] \times N$ on the other side. The procedure is inspired by the final part of \cite[Proof of Thm. 1.5]{higson1991cobordisminvariance}. We will also show that we can glue the vector bundles $S\rightarrow M$ and $p_N^*S_N\rightarrow \R\times N$ to a vector bundle $S'\rightarrow M'$ and that we can equip $M'$ with a complete Riemannian metric and $S'$ with a Hermitian metric by gluing the Riemannian/Hermitian metrics of the original manifolds/vector bundles.

For the rest of this section, suppose that we are in the setting of Theorem \ref{thm: PMT}. Recall that $N\subseteq M$ is assumed to have a uniform tubular neighbourhood $U$ of radius $\varepsilon >0$ and that we have a diffeomorphism 
\[
F\colon (-\varepsilon, \varepsilon)\times N\xrightarrow{\sim} U
\]
implementing the tubular neighbourhood as in \eqref{eq: tub nbhd map}. Moreover, by homotopy invariance of vector bundle pull-backs,
it follows that $\restr{p_N^*S_N}{(-\varepsilon, \varepsilon)\times N} \cong F^*(\restr{S}{U})$, where $S_N = \restr{S}{N}$ and $p_N\colon \R\times N\rightarrow N$ is the projection onto $N$. Denote such an isomorphism by 
\[
\Psi\colon\restr{p_N^*S_N}{(-\varepsilon, \varepsilon)\times N} \xrightarrow{\sim} \restr{S}{U},
\]
which is a vector bundle isomorphism covering $F$.

On $(\R_{<\varepsilon}\times N)\coprod (U\cup M_+)$, define the equivalence relation $\sim_F$ by setting $(r, p)\sim_F F(r, p)$ for all $r \in (-\varepsilon, \varepsilon)$ and $p \in N$. Similarly, define an equivalence relation $\sim_{\Psi}$ on $\restr{p_N^*S_N}{\R_{<\varepsilon}\times N}\coprod \restr{S}{U\cup M_+}$ by setting $((r, p), s) \sim_{\Psi} \Psi((r, p), s)$ for all $r \in (-\varepsilon, \varepsilon)$, $p \in N$ and $s \in S_N$. The following can be verified directly.
\begin{lemma}
	\label{lem: quotient manifold}
	The space $M':= ((\R_{<\varepsilon}\times N)\coprod (U\cup M_+)) /\sim_F$ has a unique structure of a smooth manifold such that the quotient map 
	\[\pi_{M'}\colon (\R_{<\varepsilon}\times N)\coprod (U\cup M_+) \rightarrow M'\]
	is a local diffeomorphism. Similarly, the space $S':= (\restr{p_N^*S_N}{\R_{<\varepsilon}\times N}\coprod \restr{S}{U\cup M_+})/\sim_{\Psi}$ has a unique structure of a smooth manifold such that the quotient map 
	\[
	\pi_{S'}\colon \restr{p_N^*S_N}{\R_{<\varepsilon}\times N}\coprod \restr{S}{U\cup M_+} \rightarrow S'
	\] 
	is a local diffeomorphism. The canonical map $S'\rightarrow M'$ induced by the vector bundle maps of $p_N^*S_N$ and $S$ make $S'$ into a complex vector bundle over $M'$.
\end{lemma}

We will identify the open manifolds $\R_{<\varepsilon}\times N$ and $U\cup M_+$ with their images in $M'$. Similarly, we will identify tangent vectors on these manifolds with tangent vectors on the manifold $M'$. In the rest of this section, we will equip $M'$ with a complete Riemannian metric and $S'$ with a Hermitian metric. In Section \ref{subsec: gluing on S'}, we will equip $S'$ with a Clifford action and a Dirac operator that agree with the corresponding structures on $M_+$ and $\R_{\leq0}\times N$ away from the gluing region.

We will construct a smooth Riemannian metric on $M'$ and a smooth Hermitian metric on $S'$ by initially gluing two metrics continuously along $N$, and then applying a smoothing procedure. This construction is convenient for constructing a Clifford action and a connection compatible with these metrics.

Denote the Riemannian metric on $M$ by $g_M$ and the product metric on $\R\times N$ by $g_{\R\times N}$. Note that $dF_{(0,p)}(\frac{\partial}{\partial r}) = n_p$ for all $p \in N$ and therefore the restriction of $g_M$ to $\restr{TM}{N}$ agrees with the restriction of $g_{\R\times N}$ to $\restr{T(\R\times N)}{\{0\}\times N}$. Therefore, on $M'$ we define the \emph{continuous} Riemannian metric $\tilde{g}$ as follows:
\begin{equation}
	\label{def: g tilde}
	\tilde{g} =\begin{cases}
		g_{\R\times N} & \text{ on $\restr{TM'}{\R_{\leq0}\times N}$}, \\
		g_M & \text{ on $\restr{TM'}{M_+}$}.
	\end{cases}
\end{equation} 
Since $T(\R\times N)\cong p_N^*(T^1_N\oplus TN)$, where $T^1_N = \R\times N$ is the trivial rank 1 vector bundle over $N$, it follows directly that for any smooth map $f\colon \R\rightarrow \R$, 
we have a vector bundle morphism 
\begin{equation}
	\label{eq: vector bundle morphism TM'}
	\alpha_f\colon T(\R\times N)\rightarrow T(\R\times N), \qquad \lambda \restr{\frac{\partial}{\partial r}}{(r, p)} + v_N \mapsto \lambda\restr{\frac{\partial}{\partial r}}{(f(r), p)} + v_N,
\end{equation}
covering $f\times \id_N$, that is an isomorphism on fibers.

The idea of our construction is that we can smear-out the nonsmoothness in $\tilde{g}$ that occurs at $\{0\}\times N$. We will do this by modifying the metric with a vector bundle morphism \eqref{eq: vector bundle morphism TM'}, where $f$ is chosen to have vanishing derivatives at $r = 0$. 

Fix a smooth, bijective and increasing function
$f\colon\R\rightarrow \R$ such that:
	\begin{enumerate}
		\item \label{appitem: prop f 2} $f^{(k)}(0) = 0$ for all $k \in \N_0$;
		\item \label{appitem: prop f 3} $\restr{f}{\R\setminus(-\varepsilon/3, \varepsilon/3)} = \id_{\R\setminus(-\varepsilon/3, \varepsilon/3)}$.
	\end{enumerate}
      
\begin{lemma}
	\label{lem: smoothing of metric}
 Then the metric $g'$ defined by 
	\[
	g'_{(r,p)}(v_1, v_2) = \tilde{g}_{(f(r), p)}(\alpha_f(v_1), \alpha_f(v_2))
	\] 
	on $\restr{TM'}{\R_{<\varepsilon}\times N}$ is smooth and extends to a smooth Riemannian metric $g'$ on $M'$ when we declare it to be equal to $g_M$ outside $\restr{TM'}{\R_{<\varepsilon}\times N}$.
\end{lemma}
\begin{proof}

 Let $V\subseteq N$ be an open subset that admits a local frame $\{e_1, \ldots, e_k\}$ for $TN$. Together with $e_0 := \frac{\partial}{\partial r}$, this yields a local frame for $TM'$ on $\R_{<\varepsilon}\times V$. The coefficients of $g'$ with respect to this frame equal
 \[
g'_{ij}(r,p)=\tilde g_{ij}(f(r),p).
 \]
Now for every 
	  continuous function $h\colon\R\rightarrow \R$  that is smooth on $\R\setminus \{0\}$ and such that $h$ has left and right derivatives of all orders at $0$, the function
 $h\circ f$ is smooth. This shows that the functions $g'_{ij}$, and hence the metric $g'$, are smooth.

 The second property of $f$ implies that 
    $\restr{g'}{(\varepsilon/3, \varepsilon)\times N} = \restr{g_M}{(\varepsilon/3, \varepsilon)\times N}$. Therefore $g'$ extends smoothly to a Riemannian metric on $M'$ when we declare it to be equal to $g_M$ outside $\R_{<\varepsilon}\times N$.
\end{proof}

Let $h$ be the Hermitian metric on the original bundle $S$ over $M$. Denote by $h_N$ the metric $\restr{h}{N}$ over $S_N = \restr{S}{N}$. The bundle $p_N^*S_N$ over $\R\times N$ is endowed with the pull-back metric $p_N^*h_N$. The bundle $S'$ was constructed by gluing $p_N^*S_N$ to $S$ along their (isomorphic) parts over $U\cong (-\varepsilon, \varepsilon)\times N$. Since on $N \cong \{0\}\times N$ both metrics are equal to $h_N$, it follows that $S'$ can be equipped with the \emph{continuous} Hermitian metric
\begin{equation}
	\label{def: h tilde}
	\tilde{h} =\begin{cases}
		p_N^*h_N & \text{ on $\restr{S'}{\R_{\leq0}\times N}$}, \\
		h & \text{ on $\restr{S'}{M_+}$}.
	\end{cases}
\end{equation} 
Let $f\colon \R\rightarrow \R$ be a smooth map. Then the vector bundle morphism
\begin{equation}
	\label{eq: vector bundle morphism S}
	\beta_f\colon p_N^*S_N\rightarrow  p_N^*S_N, \qquad (r, p, s)\mapsto (f(r), p, s),
\end{equation}
covers $f\times \id_N$ and is an isomorphism on fibers. The following result is now completely analogous to Lemma \ref{lem: smoothing of metric}.

\begin{lemma}
	\label{lem: smoothing of hermitian metric}
 The metric $h'$ on $\restr{p_N^*S_N}{\R_{<\varepsilon}\times N}$ defined by
	\[
	h'_{(r, p)}(s_1, s_2) = \tilde{h}_{(f(r),p)}(\beta_f(s_1), \beta_f(u_2)) 
	\] 
	is smooth and extends to a smooth Hermitian metric on $S'$ when we declare it to be equal to $h$ outside $\restr{p_N^*S_N}{\R_{<\varepsilon}\times N}$.
\end{lemma}

From now on, $M'$ will always be considered as a Riemannian manifold with the metric $g'$ and $S'$ as a Hermitian vector bundle with the metric $h'$. The rest of this section will be devoted to showing that $g'$ is a complete Riemannian metric.

\begin{lemma}
	\label{lem: metric estimates}
	Let $C\geq 1$ be the constant from Theorem \ref{thm: PMT}. 
	For the metric $g'$ on $M'$, the following equalities and inequalities hold, where the inequalities should be interpreted as inequalities of quadratic forms on vectors:
	\begin{enumerate}
		\item \label{item: metric est 1} $g' = g_{\R\times N}$ on $\R_{\leq 0}\times N$.
		\item \label{item: metric est 2} $g' = g_M$ on $M'\setminus( \R_{\leq \varepsilon/3} \times N)$.
		\item \label{item: metric est 3} $\frac{1}{C}g_M \leq g' = g_{\R\times N} \leq Cg_M$ on $(-\varepsilon, 0]\times N$.
		\item \label{item: metric est 4} $\frac{1}{C}g_{\R\times N}\leq g' \leq Cg_{\R\times N}$ on $[0, \varepsilon)\times N$.
		\item \label{item: metric est 5} $\frac{1}{C^2}g_M\leq g' \leq C^2g_M$ on $[0, \varepsilon)\times N$.
	\end{enumerate}
\end{lemma}
\begin{proof}
	These are all easy verifications.
	\begin{enumerate}[(i)]
		\item For $(r, p) \in \R_{\leq 0}\times N$ and $v_1, v_2 \in T_{(r,p)}(\R \times N)$, it follows that
		\begin{align*}
			g'_{(r,p)}(v_1, v_2) &=	\tilde{g}_{(f(r),p)}(\alpha_f(v_1), \alpha_f(v_2)) \\ 
			&= (g_{\R\times N})_{(f(r),p)}(\alpha_f(v_1), \alpha_f(v_2))
			= (g_{\R\times N})_{(r,p)}(v_1, v_2),
		\end{align*}
		where the last equality follows from the fact that $\alpha_f$ is fiberwise unitary for the metric $g_{\R\times N}$.
		\item This was already shown in the proof of Lemma \ref{lem: smoothing of metric}.
		\item On $(-\varepsilon, 0]\times N$ it follows that $g' = g_{\R\times N}$ by part \ref{item: metric est 1}. The rest of the estimates follows directly from the third assumption in Theorem \ref{thm: PMT}.
		\item For $(r, p) \in [0, \varepsilon) \times N$ and $v \in T_{(r,p)}(\R \times N)$, the third assumption in Theorem \ref{thm: PMT} and the the fact that $\alpha_f$ is fiberwise unitary for the metric $g_{\R\times N}$ imply that 
		\begin{align*}
			\frac{1}{C}(g_{\R\times N})_{(r,p)}(v, v) &= \frac{1}{C}(g_{\R\times N})_{(f(r),p)}(\alpha_f(v), \alpha_f(v))\\
			&\leq (g_M)_{(f(r),p)}(\alpha_f(v), \alpha_f(v)) = g'_{(r,p)}(v, v)\\ 
			&\leq C(g_{\R\times N})_{(f(r),p)}(\alpha_f(v), \alpha_f(v)) = C(g_{\R\times N})_{(r,p)}(v, v).
		\end{align*}
		\item This follows directly from part \ref{item: metric est 4} and the third assumption in Theorem \ref{thm: PMT}.\qedhere
	\end{enumerate}
\end{proof}
\begin{corollary}
	\label{cor: quotient map is isometry on end parts}
	The quotient map 
	\[
	\pi_{M'}\colon (\R_{<\varepsilon}\times N)\coprod (U\cup M_+) \rightarrow M'
	\]
	restricts to an Riemannian isometry on $\R_{<0}\times N$ and $M_+ \setminus([0, \varepsilon/3]\times N)$.
\end{corollary}

The following two results are very important for the remainder of the proof of Theorem \ref{thm: PMT}. Together, they imply that our metric $g'$ is in fact a complete Riemannian metric and they allow us to construct uniform half-isomorphisms as discussed in Section \ref{subsec: half-isos and the transition map}. Moreover, these are the results where all of the assumptions from Theorem \ref{thm: PMT} are used in a crucial way.

\begin{proposition}
	\label{prop: uniform equivalence 1}
	The identity map 
	\[
	(M_+, \restr{d_M}{M_+})\rightarrow (M_+, \restr{d_{M'}}{M_+})
	\]
	is a uniform equivalence.
\end{proposition}
\begin{proof}
	Since the map is a bijection, it suffices to show uniform expansiveness in both directions; for a general bijective map, uniform expansiveness of its inverse implies metric boundedness. For establishing uniform expansiveness, it will suffice to show that the distance functions can be estimated in terms of one another. 
	
	Suppose that $p_1, p_2 \in M_+$ and choose a curve $\gamma$ connecting $p_1$ to $p_2$ through $M$ such that $L_{g_M}(\gamma) \leq d_M(p_1, p_2) + 1$. Two possibilities can occur: either $\gamma$ lies entirely inside $U\cup M_+ \subseteq M$; or it does not. In case it does, it immediately follows that
	\begin{equation}
		\label{eq: case 1 forward}
		d_{M'}(p_1, p_2) \leq L_{g'}(\gamma) \leq CL_{g_M}(\gamma) \leq C(d_M(p_1, p_2) + 1),
	\end{equation}
	where the second inequality follows from parts \ref{item: metric est 2}, \ref{item: metric est 3} and \ref{item: metric est 5} of Lemma \ref{lem: metric estimates}.
    
    Now suppose that  $\gamma$ does not lie inside of $U\cup M_+$. Then the estimate \eqref{eq: case 1 forward} does not hold. It will be shown that in this case, the curve $\gamma$ can be modified to lie entirely within $U\cup M_+$, while still controlling the length of the modified curve. By assumption, the set $\cI:=\gamma^{-1}(M_-\setminus U)) \subseteq [0,1]$ is nonempty and it is trivially contained within the set $\cJ:=\gamma^{-1}(M_-)$. 
	
	
	Let $[a, b]\subseteq [0,1]$ be a connected component of $\cJ$ such that there exists a $c \in [a, b]\cap\cI$. Since $\gamma(a), \gamma(b) \in N$ and $\gamma(c) \notin U$, it follows that $d_M(\gamma(a), \gamma(c))>\varepsilon$ and $d_M(\gamma(b), \gamma(c))>\varepsilon$ so that $L_{g_M}(\restr{\gamma}{[a,b]})>2\varepsilon$. It follows that $\cJ$ has at most $\frac{L_{g_M}(\gamma)}{2\varepsilon} \leq \frac{d_M(p_1, p_2) + 1}{2\varepsilon}$ of such connected components. 
	
	Because $\id_N\colon (N, d_N)\rightarrow (N, \restr{d_M}{N})$ is assumed to be a uniform equivalence, there exists an $S>0$ such that for all $q_1, q_2 \in N$ we have that
	\[
	d_M(q_1, q_2)\leq d_M(p_1, p_2)+1 \implies d_N(q_1, q_2) \leq S.
	\]
Then  $d_M(\gamma(a), \gamma(b)) \leq d_M(p_1, p_2)+1$ and therefore that $d_N(\gamma(a), \gamma(b)) \leq S$. So $\gamma(a)$ and $\gamma(b)$ can be connected through $N$ by a curve $\gamma_{ab}$ of length $L_{g_N}(\gamma_{ab}) = L_{g_M}(\gamma_{ab}) < S+1$. The concatenation of curves 
	\[
	\bar{\gamma} = \restr{\gamma}{[0, a]}*\gamma_{ab}*\restr{\gamma}{[b,1]}
	\]
	is piecewise smooth and has length $L_{g_M}(\bar{\gamma}) \leq d_M(p_1, p_2) + 1 + (S+1)$. Doing this for all at most $\frac{d_M(p_1, p_2) + 1}{2\varepsilon}$ connected components of $\cJ$ on which $\gamma$ leaves $U\cup M_+$, the resulting curve $\tilde{\gamma}$ is piecewise smooth, stays in $U\cup M_+$ and, using the metric $g_M$, has length less than
	\[
	d_M(p_1, p_2) + 1 + \frac{d_M(p_1, p_2) + 1}{2\varepsilon}(S+1).
	\]	
	By combining parts \ref{item: metric est 2}, \ref{item: metric est 3} and \ref{item: metric est 5} of Lemma \ref{lem: metric estimates}, we conclude that 
	\[
	d_{M'}(p_1, p_2) \leq L_{g'}(\tilde{\gamma}) \leq CL_{g_M}(\tilde{\gamma}) \leq C\left( d_M(p_1, p_2) + 1 + \frac{d_M(p_1, p_2) + 1}{2\varepsilon}(S+1)\right).
	\]
	Note that this estimate extends the one in Equation \eqref{eq: case 1 forward} where $\gamma$ does not leave $U\cup M_+$, such that this holds for any $p_1, p_2 \in M_+$.
	
	For the converse inequality, a similar bound of $d_M$ in terms of $d_{M'}$ must be found. This is now considerably simpler as the manifold $M'$ is a product on one side. Let $p_1, p_2 \in M_+$ and choose $\gamma$ connecting $p_1$ to $p_2$ through $M'$ such that $L_{g'}(\gamma) \leq d_{M'}(p_1, p_2) +1$. If $\gamma$ stays within $U\cup M_+$ it immediately follows that 
	\[
	d_M(p_1, p_2) \leq L_{g_M}(\gamma) \leq CL_{g'}(\gamma) \leq C(d_{M'}(p_1, p_2)+1).
	\] 
	If $(a, b)\subseteq [0,1]$ is a maximal interval on which $\gamma$ leaves $M_+$, then $\restr{\gamma}{[a,b]}$ lies in $\R_{\leq 0}\times N$ and therefore is of the form $\restr{\gamma}{[a,b]} = (\gamma_{\R}, \gamma_N)$. Replacing $\restr{\gamma}{[a,b]}$ by $(0, \gamma_N)$ immediately gives a piecewise smooth curve $\tilde{\gamma}$ contained in $M_+$. Note that by part \ref{item: metric est 1} of Lemma \ref{lem: metric estimates}, it follows that 
	\begin{align*}
		L_{g'}(\restr{\gamma}{[a, b]}) &= L_{g_{\R\times N}}(\restr{\gamma}{[a, b]})= \int_a^b\sqrt{|\dot{\gamma_{\R}}(t)|^2 + \|\dot{\gamma_N}(t)\|^2_{g_N}}dt \\
		&\geq \int_a^b \|\dot{\gamma_N}(t)\|_{g_N}dt= L_{g'}(0, \gamma_N).
	\end{align*}
	So that this modification of $\gamma$ is actually length-decreasing. So without loss of generality, it may be assumed that the curve $\gamma$ is contained within $M_+$ and thus that $d_M(p_1,p_2) \leq C(d_{M'}(p_1, p_2) + 1)$. The result now follows.
\end{proof}
Analogously to Proposition \ref{prop: uniform equivalence 1}, we have the following.
\begin{proposition}
	\label{prop: uniform equivalence 2}
	The identity map
	\[
	(\R_{\leq0}\times N, \restr{d_{M'}}{\R_{\leq0}\times N})\rightarrow (\R_{\leq0}\times N, \restr{d_{\R\times N}}{\R_{\leq0}\times N})
	\]
	is a uniform equivalence.
\end{proposition}

\begin{corollary}
	\label{cor: g' is complete}
	The metric $g'$ constructed in Lemma \ref{lem: smoothing of metric} is complete.
\end{corollary}
\begin{proof}
	To show that $g'$ is complete, it will be shown that the metric space $(M', d_{M'})$ has the Heine--Borel property, so that completeness follows by the Hopf--Rinow theorem. Let $K \subseteq M'$ be closed and bounded. Set $K_- = K\cap (\R_{\leq0}\times N)$ and $K_+ = K\cap M_+$. Then $K_{\pm}$ are closed and bounded. Since $M' = (\R_{\leq 0}\times N )\cup M_+$, it follows that $K = K_+\cup K_-$. It follows from Propositions \ref{prop: uniform equivalence 1} and \ref{prop: uniform equivalence 2} that $K_{\pm}$ are also closed and bounded when considered as subsets of $M$ and $\R\times N$. Since these are both complete, it follows that $K_{\pm}$ are compact. Therefore, $K$ is compact and thus $g'$ is complete.
\end{proof}

\subsection{Gluing Clifford actions and Dirac operators}
\label{subsec: gluing on S'}
We now turn our attention to the Clifford actions and the Dirac operators. Note that on $N\cong \{0\}\times N$ we have $\restr{TM'}{N} \cong T^1_{N}\oplus TN \cong \restr{TM}{N} $ with the isomorphism being $\restr{dF}{N}$. Note that by assumption, the unit normal field $n = dF\left(\frac{\partial}{\partial r}\right)$ on $N$ acts via $c_M(n_p) = -i\gamma_{N, p}$ on $S_{N, p}$. So if $p \in N$ and $\lambda n_p + v_{N,p} \in T_pM$ with $v_{N,p}\in T_pN$, then 
\[
c_M(\lambda n_p + v_{N,p}) = -i\lambda\gamma_{N, p} + c_{N,p}(v_{N,p}) = c_{\R\times N}(\lambda\frac{\partial}{\partial r} + v_{N,p}),
\] 
where $c_{\R\times N}$ is the Clifford action of $T(\R\times N)$ on $p_N^*S_N$ as constructed in Section \ref{subsec: overview of proof}. So the Clifford actions $c_{\R\times N}$ and $c_M$ agree on $N \cong \{0\}\times N$, so that the following defines a \emph{continuous} vector bundle morphism
\begin{equation}
	\label{def: c tilde}
	\tilde{c} =\begin{cases}
		c_{\R\times N} & \text{ on $\restr{TM'}{\R_{\leq0}\times N}$}, \\
		c_M & \text{ on $\restr{TM'}{M_+}$}.
	\end{cases}
\end{equation} 
Note that when $M'$ is endowed with the continuous metric $\tilde{g}$ and $S'$ with the continuous Hermitian metric $\tilde{h}$, then for any $(x, v) \in TM'$
\[
\tilde{c}(x, v)^2 = -\tilde{g}_x(v, v)\id_{S'_x},
\]
and
\[
\tilde{c}(x, v)^* = -\tilde{c}(x,v),
\]
where the adjoint is taken with respect to $\tilde{h}$. Indeed this is easily seen to hold true as these are fiberwise equalities that hold for both $c_{\R\times N}$ and $c_M$. 

For the following lemma, let $\Ad_{(\beta_f)^{-1}_{(r,p)}}$ denote conjugation with the invertible linear map $(\beta_f)^{-1}_{(r,p)}\colon S'_{(f(r), p)}\rightarrow S'_{(r, p)}$.

\begin{lemma}
	\label{lem: smoothing of c}
	Let $f$ be as in Lemma \ref{lem: smoothing of metric} and let $\alpha_f$ and $\beta_f$ be defined as in \eqref{eq: vector bundle morphism TM'} and \eqref{eq: vector bundle morphism S}. The map $c'\colon \restr{TM'}{\R_{<\varepsilon}\times N}\rightarrow \restr{\End(S')}{\R_{<\varepsilon}\times N}$ that is defined on fibers by
	\begin{equation}
		\label{diag: smoothened clifford action}
		\begin{tikzcd}
			{T_{(r, p)}M'} & {\End(S'_{(r,p)})} \\
			{T_{(f(r),p)}M'} & {\End(S'_{(f(r), p)})}
			\arrow["{c'_{(r,p)}}", dotted, from=1-1, to=1-2]
			\arrow["{(\alpha_f)_{(r,p)}}"', from=1-1, to=2-1]
			\arrow["{\tilde{c}_{(f(r),p)}}", from=2-1, to=2-2]
			\arrow["{\Ad_{(\beta_f)^{-1}_{(r,p)}}}"', from=2-2, to=1-2]
		\end{tikzcd}
	\end{equation}
	is a smooth vector bundle morphism and extends smoothly to $TM'$ when we declare it to be equal to $c_M$ outside of $\restr{TM'}{\R_{<\varepsilon}\times N}$. The resulting vector bundle morphism defines a Clifford action on $S'$ with respect to the Riemannian metric $g'$ which is skew-Hermitian with respect to the metric $h'$.	
\end{lemma}
\begin{proof}
	Let $\alpha_f$ be the vector bundle morphism defined in Equation \eqref{eq: vector bundle morphism TM'}. The composition 
    \begin{equation}\label{eq: tilde c compos}
\begin{tikzcd}
		{\restr{TM'}{\R_{<\varepsilon}\times N}} & {\restr{TM'}{\R_{<\varepsilon}\times N}} & {\restr{\End(S')}{\R_{<\varepsilon}\times N}}
		\arrow["\alpha_f", from=1-1, to=1-2]
		\arrow["{\tilde{c}}", from=1-2, to=1-3]
	\end{tikzcd}
    \end{equation}
	is a continuous vector bundle morphism covering $f\times \id_N$. Analogously to the proof of Lemma \ref{lem: smoothing of metric}, one can show that the matrix-valued functions defined by this composition with respect to local frames of $T(\R_{<\varepsilon}\times N)$ and $S_N$, are smooth functions. So 
    \eqref{eq: tilde c compos} is
%
%
%
%
%
%
%
%
%
%
    a smooth vector bundle morphism covering $f\times \id_N$ and therefore induces a smooth vector bundle morphism
	\[\begin{tikzcd}
		{c'\colon \restr{TM'}{\R_{<\varepsilon}\times N}} & {(f\times\id_N)^*(\restr{\End(S')}{\R_{<\varepsilon}\times N})} & {\restr{\End(S')}{\R_{<\varepsilon}\times N},}
		\arrow[from=1-1, to=1-2]
		\arrow["\sim", from=1-2, to=1-3]
	\end{tikzcd}\]
	where the last isomorphism follows as $\restr{S'}{\R_{<\varepsilon}\times N} \cong (f\times \id_N)^*(\restr{S'}{\R_{<\varepsilon}\times N})$.
	Let $(r, p) \in \R_{<\varepsilon}\times N$. We see that the map $c'_{(r,p)}$ on the fiber over $(r,p)$ is given by 
	\[\begin{tikzcd}[column sep=large]
		{T_{(r,p)}M'} & {T_{(f(r),p)}M'} & {\End(S'_{(f(r), p)})} & {\End(S'_{(r, p)})}
		\arrow["{(\alpha_f)_{(r,p)}}", from=1-1, to=1-2]
		\arrow["{\tilde{c}_{(f(r),p)}}", from=1-2, to=1-3]
		\arrow["{\Ad_{(\beta_f)^{-1}_{(r,p)}}}", from=1-3, to=1-4]
	\end{tikzcd}\]
	and therefore $c'$ is indeed given by Diagram \eqref{diag: smoothened clifford action}. So if $r> \varepsilon/3$, then $c'_{(r,p)} = \tilde{c}_{(r,p)} =( c_M)_{(r,p)}$, which shows that $c'$ extends smoothly to $TM'$ when we declare it to be equal to $c_M$ outside $\restr{TM'}{\R_{<\varepsilon}\times N}$.

    By writing out the definitions, we find that
    $c'$ is a Clifford action with respect to the Riemannian metric $g'$ and Hermitian metric $h'$.
\end{proof}

Now that $S'$ is endowed with the structure of a Dirac bundle, it is possible to define a differential operator that interpolates between the Dirac operators on $M$ and $\R\times N$. To this end, let $\psi\colon M'\rightarrow \R$ be any smooth function such that $\psi(p) \in [0,1]$ for all $p\in M'$ and such that $\restr{\psi}{\R_{\leq-\varepsilon/3}\times N} = 0$ and $\restr{\psi}{M'\setminus(\R_{<\varepsilon/3}\times N)} = 1$. Denote by $\nabla^M$ the Dirac connection on $S$ used to construct $D$ and let $\nabla^{\R\times N}:= p_N^*\nabla^N$ be the Dirac connection on $p_N^*S_N$ induced by a choice of Dirac connection $\nabla^N$ on $S_N$. Use $\psi$ to define the connection 
\begin{equation}
	\label{eq: connection on S'}
	\nabla' = (1-\psi)\nabla^{\R\times N} + \psi\nabla^{M}
\end{equation}
on the bundle $S'$. 

Using this connection, define the first-order differential operator $\tilde{D}$ on $S'$ as the composition
\[
\tilde{D}\colon 
C^{\infty}(M';S') \xrightarrow{\nabla'} C^{\infty}(M';T^*M'\otimes S') \cong C^{\infty}(M';TM'\otimes S')
\xrightarrow{c'}C^{\infty}(M';S').
\]
Here we identify $TM'\cong T^*M'$ via the Riemannian metric $g'$.
Define 
\begin{equation}
	\label{eq: diff operator D'}
	D' = \frac{1}{2}(\tilde{D} + \tilde{D}^{\dagger}),
\end{equation}
where $\tilde{D}^{\dagger}$ denotes the formal adjoint of $\tilde{D}$ with respect to the inner product on $L^2(M';S')$.

\begin{lemma}
	\label{cor: D' is Dirac operator}
	The differential operator $D'$ is a Dirac operator. In particular, $D'$ is elliptic, has finite propagation speed and the corresponding unbounded operator on $L^2(M';S')$ is essentially self-adjoint.
\end{lemma}
\begin{proof}
It follows from the definitions that $D'$ is a symmetric first-order differential operator and $\sigma_{D'} = c'$.
\end{proof}

For the following result, write $O_+ = M_+\setminus([0, \varepsilon/3]\times N)$ and $O_- = \R_{<-\varepsilon/3}\times N$. 

\begin{lemma}
	\label{lem: Dirac operators are intertwined}
	Let $D'$ be the constructed Dirac operator on $S'$. Then 
	\[\restr{D'}{O_+} = \restr{D_M}{O_+}\]
	and 
	\[
	\restr{D'}{O_-} = \restr{D_{\R\times N}}{O_-}.
	\]
\end{lemma}
\begin{proof}
	By construction of $c'$ and $\nabla'$, the following equalities hold:
	\[
	\restr{\tilde{D}}{O_+} = \restr{D_M}{O_+} 
	\qquad\text{ and }\qquad
	\restr{\tilde{D}}{O_-} = \restr{D_{\R\times N}}{O_-}.
	\]
	
	It follows that for all smooth, compactly supported sections $\sigma$ and $\tau$ of $S'$, supported within $O_+$, the following holds by construction of $g'$ and $h'$: 
	\begin{align*}
		\langle\tilde{D}^{\dagger}\sigma, \tau\rangle_{L^2(M';S')}&= \langle\sigma, \tilde{D}\tau\rangle_{L^2(M';S')} = \langle\sigma, D_M\tau\rangle_{L^2(M;S)}\\
		&= \langle D_M\sigma, \tau\rangle_{L^2(M;S)} = \langle \tilde{D}\sigma, \tau\rangle_{L^2(M';S')}. 
	\end{align*}
	So $\restr{\tilde{D}^{\dagger}\sigma}{O_+} = \restr{\tilde{D}\sigma}{O_+}$, but since $\supp(\sigma)\subseteq O_+$ and differential operators are local, this implies that $\tilde{D}^{\dagger}\sigma=\tilde{D}\sigma$. So 
	\[
	\restr{\tilde{D}^{\dagger}}{O_+} = \restr{\tilde{D}}{O_+}  = \restr{D_M}{O_+},
	\] 
	and therefore 
	\[
	\restr{D'}{O_+} = \frac{1}{2}(\restr{\tilde{D}}{O_+}+\restr{\tilde{D}^{\dagger}}{O_+}) = \restr{D_M}{O_+}.
	\]
	A similar computation shows that $\restr{D'}{O_-} = \restr{D_{\R\times N}}{O_-}$.
\end{proof}

\begin{proposition}
	\label{prop: manifolds are half-isomorphic}
	With $(M', g')$, $(S', h')$ and $D'$ as constructed in this section, the following hold:
	\begin{enumerate}
		\item When $M$ and $M'$ are considered as partitioned manifolds with partitioning hypersurface $N_{2\varepsilon/3}:=\{2\varepsilon/3\}\times N$ and positive part $M_+':= M_+\setminus ([0, 2\varepsilon/3)\times N) \subseteq O_+$, then the canonical maps 
		\[
		(O_+, \restr{g_M}{O_+}) \rightarrow (O_+, \restr{g'}{O_+}),
		\]
		and
		\[
		(\restr{S}{O_+}, \restr{h}{O_+}) \rightarrow (\restr{S'}{O_+}, \restr{h'}{O_+})
		\]
		form a uniform half-isomorphism that is compatible with the Dirac operators.
		\item When $M'$ and $\R\times N$ are considered as partitioned manifold with partitioning hypersurface $N_{-2\varepsilon/3}:=\{-2\varepsilon/3\}\times N$ and negative part $C_-:=\R_{\leq -2\varepsilon/3}\times N \subseteq O_-$, then the canonical maps
		\[
		(O_-, \restr{g'}{O_-})\rightarrow (O_-, \restr{g_{\R\times N}}{O_-}),
		\]
		and
		\[
		(\restr{S'}{O_-}, \restr{h'}{O_-})\rightarrow (\restr{p_N^*S_N}{O_-}, \restr{p_N^*h_N}{O_-})
		\]
		form a uniform half-isomorphism that is compatible with the Dirac operators.
	\end{enumerate}
\end{proposition}
\begin{proof}
	In both cases, it follows directly from the construction of $g'$ that the map on the manifold-level is a Riemannian isometry. Also it is easy to see that it extends to a homeomorphism on the closures of the open sets on which it is defined and that it preserves the partitioning hypersurface and the positive (respectively negative) side of the partition that is entirely contained in its domain.	Moreover, it follows directly from the construction of $h'$ that the maps on the vector bundle-level are fiberwise unitary. By exactly the same argument as in Remark \ref{rem: hypersurfaces as in PMT are simple}, it follows that we can choose functions as in Definition \ref{def: simple hypersurface} that are supported within $O_{\pm}$ and therefore both pairs of maps form a half-isomorphism.	
	
	It follows by Propositions \ref{prop: uniform equivalence 1} and \ref{prop: uniform equivalence 2} that the maps on the manifold-level are uniform equivalences and Lemma \ref{lem: Dirac operators are intertwined} implies that the half-isomorphisms are compatible with the Dirac operators on the manifolds. The result follows.
\end{proof}

\subsection{Completing the proof of the partitioned manifold index theorem}
\label{subsec: completing proof}
The interpolating manifold $M'$ is now equipped with a Dirac operator that agrees with the Dirac operator on $M$ on $O_+$ and with the Dirac operator on $\R\times N$ on $O_-$. It was shown in Proposition \ref{prop: manifolds are half-isomorphic} that we obtain uniform half-isomorphisms that are compatible with the Dirac operators on the separate halves of the transition manifold $M'$. However, to obtain the isometric maps needed in a half-isomorphism, the partitioning hypersurfaces had to be moved away from $\{0\}\times N$ to avoid the region where the two structures on the original manifold $M$ and the product manifold $\R\times N$ are glued and smoothened out. Therefore, to prove Theorem \ref{thm: PMT}, only some final results that show what happens when we move around the partitioning hypersurface are needed. 

\begin{lemma}
	\label{lem: commutative square near hypersurfaces}
	Let $M$ be a complete Riemannian manifold which is partitioned by two mutually near hypersurfaces $N$ and $N'$. Let $q\colon (N, d_N)\rightarrow (N', d_{N'})$ be a coarse map such that
	\begin{equation}
		\label{eq: coarse map close to id}
		\sup_{p \in N}d_M(q(p), p) <\infty.
	\end{equation} 
	Let $\Phi$ and $\Phi'$ be the maps as in \eqref{diag: identification maps} associated to $N$ and $N'$, respectively.  Then the following diagram commutes:
	\begin{equation}
		\label{diag: commuting diagram moving hypersurface}
		\begin{tikzcd}
			{K_p(C^*(N, d_N))} & {K_p(C^*(N\subseteq M))} \\
			{K_p(C^*(N', d_N'))} & {K_p(C^*(N'\subseteq M))}
			\arrow["\Phi", from=1-1, to=1-2]
			\arrow["{q_*}", from=1-1, to=2-1]
			\arrow[equal, from=1-2, to=2-2]
			\arrow["\Phi'", from=2-1, to=2-2]
		\end{tikzcd}
	\end{equation}
\end{lemma}
\begin{remark}
	Note that the vertical map on the right really is the identity as the partitioning hypersurfaces $N$ and $N'$ are mutually near and therefore their localized Roe algebras are the same; see also Corollary \ref{cor: mutually near hypersurfaces induce same partitioned index}.
\end{remark}
\begin{proof}
	Fix an integer $R>0$ such that $N\subseteq B_R(N')$ and $N'\subseteq B_R(N)$. Consider $Y_k = \ol{B_k(N)}$ and $Y'_k = \ol{B_k(N')}$. It follows that $Y_k \subseteq Y'_{R+k}$ and $Y'_k \subseteq Y_{R+k}$. Write $j^{(\prime)}_k\colon N^{(\prime)}\hookrightarrow Y^{(\prime)}_k$ and $i_k\colon Y_k\hookrightarrow Y'_{k+R}$ for the inclusions. The assumption  \eqref{eq: coarse map close to id} precisely implies that for any $k\in \N$, the two maps
    \[
	i_k \circ j_k\circ \id_N, j'_{R+k}\circ \id_{N'}\circ q\colon
	(N, d_N) \rightarrow (Y'_{R+k}, d_M|_{Y'_{R+k}})
    \]
    are close.
	Since passing from coarse spaces to the $K$-theory of the Roe algebra is functorial and invariant under closeness of coarse maps, it follows that 
    the upper part of the following diagram commutes:
	\[\begin{tikzcd}
		{K_p(C^*(N, d_N))} & {K_p(C^*(N', d_{N'}))} \\
		{K_p(C^*(N, \restr{d_M}{N}))} & {K_p(C^*(N', \restr{d_M}{N'}))} \\
		{K_p(C^*(Y_k, \restr{d_M}{Y_k}))} & {K_p(C^*(Y'_{R+k}, \restr{d_M}{Y'_{R+k}}))} \\
		{K_p(C^*(N\subseteq M))} & {K_p(C^*(N'\subseteq M))}
		\arrow["{q_*}", from=1-1, to=1-2]
		\arrow["{\id_{N,*}}", from=1-1, to=2-1]
		\arrow["{\id_{N',*}}", from=1-2, to=2-2]
		\arrow["\sim", sloped, from=2-1, to=3-1]
		\arrow["\sim", sloped, from=2-2, to=3-2]
		\arrow["\sim", sloped, from=3-1, to=3-2]
		\arrow["\sim", sloped, from=3-1, to=4-1]
		\arrow["\sim", sloped, from=3-2, to=4-2]
		\arrow[equal, from=4-1, to=4-2]
	\end{tikzcd}\]
	Commutativity of the lower square follows as both compositions are given by conjugation with the inclusion $L^2(Y_k; S)\hookrightarrow L^2(M;S)$. The compositions of the vertical maps are exactly $\Phi$ and $\Phi'$ by Lemma \ref{prop: localized roe alg isomorphic to roe alg of subset} and therefore the result follows.
\end{proof}

Consider our original manifold $M$ with partition $(N,M_+, M_-)$. Let $N_{\pm 2\varepsilon/3}$, $M_+'$ and $C_-$ be defined as in Proposition \ref{prop: manifolds are half-isomorphic}. Also, let $C\geq 1$ be as in Theorem \ref{thm: PMT}. It is easy to see that the partitions of $M$ defined by $N$ and $N_{2\varepsilon/3}$ are mutually near and that the map $\phi_1\colon (N, d_N)\rightarrow (N_{2\varepsilon/3}, d_{N_{2\varepsilon/3}})$ defined by $\phi_1(p) = (2\varepsilon/3, p)$ is a uniform equivalence which satisfies 
$
\sup_{p \in N}d_M(\phi_1(p), p) \leq \frac{2\sqrt{C}\varepsilon}{3}.
$
Similarly, it easy to see that the partitions of $M'$ defined by $N_{\pm 2\varepsilon/3}$ are mutually near and that the map $\phi_2\colon (N_{2\varepsilon/3}, d_{N_{2\varepsilon/3}}) \rightarrow (N_{-2\varepsilon/3}, d_{N_{-2\varepsilon/3}})$ defined by $\phi_2(2\varepsilon/3, p) = (-2\varepsilon/3, p)$ is a uniform equivalence which satisfies 
$
\sup_{p \in N_{2\varepsilon/3}}d_{M'}(\phi_2(p), p)\leq\frac{4\sqrt{C}\varepsilon}{3}.
$
Finally, it is obvious that the partitions of $\R\times N$ defined by $N_0$ and $N_{-2\varepsilon/3}$ are mutually near and that the map $\phi_3\colon (N_{-2\varepsilon/3}, d_{N_{-2\varepsilon/3}}) \rightarrow (N_0, d_{N_0})$ defined by $\phi_3(-2\varepsilon/3, p) = (0, p)$ is a uniform equivalence which satisfies 
$
\sup_{p \in N_{-2\varepsilon/3}}d_{\R\times N}(\phi_3(p), p)\leq\frac{2\varepsilon}{3}.
$
So Lemma \ref{lem: commutative square near hypersurfaces} can be applied to these three changes in hypersurface.

\begin{proof}[Proof of Theorem \ref{thm: PMT}]
	To prove the theorem for general $N \subseteq M$, the general case will be reduced to the product case, for which the theorem is known to be true (Theorem \ref{thm: PMT for product}). To reduce the general case to the product, the partitioned manifold must be modified five times; three times by moving around the partitioning hypersurface and two times by using a half-isomorphism. 
	
	Recall that the manifold $M'$ that was constructed in Section \ref{subsec: gluing on M'} has the structure of a partitioned manifold by considering either of the partitioning hypersurfaces $N_{\pm2\varepsilon/3}$. By Proposition \ref{prop: manifolds are half-isomorphic}, it follows that for $N_{2\varepsilon/3}$ we have a uniform half-isomorphism between $(N_{2\varepsilon/3}, M_+', M\setminus (M_+')^{\circ})$ and $(N_{2\varepsilon/3}, M_+', M'\setminus (M_+')^{\circ})$.  Similarly, it follows that for $N_{-2\varepsilon/3}$ we have a uniform half-isomorphism between $(N_{-2\varepsilon/3}, M'\setminus (C_-)^{\circ}, C_-)$ and $(N_{-2\varepsilon/3}, \R_{\geq -2\varepsilon/3}\times N, C_-)$. Moreover, both of these uniform half-isomorphisms are compatible with the Dirac operators. The Riemannian isometries of these half-isomorphisms are just $\id_{O_+}$ and $\id_{O_-}$, respectively. Consider the following diagram:
	\begin{equation}
		\label{diag: the final diagram}
		\begin{tikzcd}
			{K_0(C^*(N, d_N))} & {K_0(C^*(N\subseteq M))} \\
			{K_0(C^*(N_{2\varepsilon/3}, d_{N_{2\varepsilon/3}}))} & {K_0(C^*(N_{2\varepsilon/3}\subseteq M))} \\
			{K_0(C^*(N_{2\varepsilon/3}, d_{N_{2\varepsilon/3}}))} & {K_0(C^*(N_{2\varepsilon/3}\subseteq M'))} \\
			{K_0(C^*(N_{-2\varepsilon/3}, d_{N_{-2\varepsilon/3}}))} & {K_0(C^*(N_{-2\varepsilon/3}\subseteq M'))} \\
			{K_0(C^*(N_{-2\varepsilon/3}, d_N))} & {K_0(C^*(N_{-2\varepsilon/3}\subseteq \R\times N))} \\
			{K_0(C^*(N_0, d_N))} & {K_0(C^*(N_0\subseteq\R\times N))}
			\arrow["{\Phi^M}", from=1-1, to=1-2]
			\arrow["{(\phi_1)_*}"', from=1-1, to=2-1]
			\arrow[bend right = 90, dotted, from=1-1, to=6-1]
			\arrow[equals, from=1-2, to=2-2]
			\arrow["{\Phi_{2\varepsilon/3}^M}", from=2-1, to=2-2]
			\arrow["{(\id_{N_{2\varepsilon/3}})_*}"', from=2-1, to=3-1]
			\arrow["{\cT_+}"', from=3-2, to=2-2]
			\arrow["{\Phi_{2\varepsilon/3}^{M'}}", from=3-1, to=3-2]
			\arrow["{(\phi_2)_*}"', from=3-1, to=4-1]
			\arrow[equals, from=3-2, to=4-2]
			\arrow["{\Phi^{M'}_{-2\varepsilon/3}}", from=4-1, to=4-2]
			\arrow["{(\id_{N_{-2\varepsilon/3}})_*}"', from=4-1, to=5-1]
			\arrow["{\cT_-}"', from=5-2, to=4-2]
			\arrow["{\Phi^{\R\times N}_{-2\varepsilon/3}}", from=5-1, to=5-2]
			\arrow["{(\phi_3)_*}"', from=5-1, to=6-1]
			\arrow[equals, from=5-2, to=6-2]
			\arrow["{\Phi^{\R\times N}_0}", from=6-1, to=6-2]
		\end{tikzcd}
	\end{equation}
	In Diagram \eqref{diag: the final diagram}, $\cT_{\pm}$
	denote the transition maps on the $K$-theories of the localized Roe algebras induced by the corresponding uniform half-isomorphisms. Finally, the horizontal maps are all defined as in \eqref{diag: identification maps} and decorated with the corresponding total manifold as superscript and the $r$-coordinate of the hypersurface as subscript to distinguish them. Commutativity of the second and fourth square, counting from above, follows from Proposition \ref{prop: transition diagram commutes}. Commutativity of the first, third and fifth square follows from Lemma \ref{lem: commutative square near hypersurfaces}. Note that by functoriality, the dotted arrow is the map induced by the coarse map
	\[
	\phi_N\colon N\rightarrow N_0\subseteq \R\times N \qquad p\mapsto (0, p), 
	\]
	which is exactly the identification used in the proof of Theorem \ref{thm: PMT for product} to view $N$ as submanifold of $\R\times N$. 
	
	It follows from Theorem \ref{thm: PMT for product} that
	\[
	\Phi^{\R\times N}_0 (\phi_N)_*(\Ind(D_N)) = \Ind(D_{\R\times N}; N_0).
	\]
	Using commutativity of Diagram \eqref{diag: the final diagram}, we find that 
	\begin{align*}
		\Phi^M(\Ind(D_N)) &= \cT_+\circ \cT_-(\Ind(D_{\R\times N}; N_0))\\
		&= \cT_+\circ \cT_-(\Ind(D_{\R\times N}; N_{-2\varepsilon/3}))\\
		&=\cT_+(\Ind(D'; N_{-2\varepsilon/3}))\\
		& = \cT_+(\Ind(D'; N_{2\varepsilon/3})) \\
		&=\Ind(D; N_{2\varepsilon/3}) \\
		&= \Ind(D;N),
	\end{align*}
	where the third and fifth equality follow from Proposition \ref{prop: iso on localized Roe algs preserves partitioned index} and the second, fourth and sixth equality follow from Corollary \ref{cor: mutually near hypersurfaces induce same partitioned index}.	
\end{proof}

\section{Generalizations}

We discuss how Theorem \ref{thm: PMT} can be generalized to an equivariant version (Theorem \ref{thm: PMT Gamma}) and to the case where the hypersurface $N$ may be disconnected (Theorem \ref{thm: PMT (mult comp)}).

\subsection{An equivariant version}\label{subsec: equivariant results}

Consider the setting of Theorem \ref{thm: PMT}. Let $\Gamma$ be a  discrete group acting isometrically, properly and freely on $M$, preserving $M_+$, $M_-$ and hence $N$. Suppose that the vector bundle $S$ is $\Gamma$-equivariant, and that the action by $\Gamma$ preserves the metric on $S$. Suppose that the Clifford action, the Dirac connection $\nabla$ used to define $D$ and the Dirac connection on $S_N$ used to define $D_N$, are $\Gamma$-invariant. 

The algebras in Definition \ref{def: (localized) Roe algebra} have $\Gamma$-equivariant versions, denoted by superscripts $\Gamma$, if $H$ carries a unitary representation of $\Gamma$ such that
\[
\gamma \circ \rho(f)\circ \gamma^{-1} = \rho(\gamma \cdot f) \in \cB(H)
\]
for all $\gamma \in \Gamma$ and $f \in C^{\infty}(M)$. These algebras are defined by taking $\Gamma$-equivariant operators before completing the relevant subalgebras of $\cB(H)$.
Then we have the \emph{equivariant coarse index}
\[
\Ind_{\Gamma}(D_N) \in K_0(C^*(N)^{\Gamma}).
\]
\begin{remark}
Equivariant Roe algebras and the equivariant coarse index can be generalized to possibly non-discrete groups, acting possibly non-freely, but this involves a few subtleties \cite{GHM21b, GHM21a}.
\end{remark}

Furthermore, one has the equivariant analogue of Proposition \ref{prop: Higson's 1.4}, if one takes $\phi_+$ to be $\Gamma$-invariant. Hence Definition \ref{def: partitioned index} generalizes to yield
\[
\Ind_{\Gamma}(D; N):= \partial [q_N(U_+)] \in K_0(C^*(N \subseteq M)^{\Gamma}).
\]
In the setting of Theorem \ref{thm: PMT Gamma}, existence of such a $\Gamma$-invariant $\phi_+$ is guaranteed by the same argument as in Remark \ref{rem: hypersurfaces as in PMT are simple} as the map $F$ from \eqref{eq: tub nbhd map} is $\Gamma$-equivariant.
\begin{theorem}[Equivariant partitioned manifold index theorem]
	\label{thm: PMT Gamma}
Suppose that the conditions of Theorem \ref{thm: PMT} hold.	
	Then
	\[
	\Phi_{\Gamma}(\Ind_{\Gamma}(D_N)) = \Ind_{\Gamma}(D;N),
	\]
	with $\Phi_{\Gamma}$ defined as  the $\Gamma$-equivariant analogue of \eqref{diag: identification maps}.
\end{theorem}
The proof of Theorem \ref{thm: PMT Gamma} is completely analogous to the proof of Theorem \ref{thm: PMT}, since the $\Gamma$-equivariance and -invariance assumptions made mean that all operators that occur in the proof are $\Gamma$-equivariant, and therefore lie in the $\Gamma$-equivariant versions of the relevant algebras. Furthermore, all diffeomorphisms, isomorphisms and coarse equivalences can now be taken to be $\Gamma$-equivariant.

The case where $N/\Gamma$ is compact is particularly interesting.
First of all, in that case we will see in  Corollary \ref{cor: PMT for cocompact hypersurface} that the conditions in Theorem \ref{thm: PMT} do not need to be verified. Secondly, in that case, the equivariant coarse index of $D_N$ equals its $K$-theoretic $\Gamma$-index \cite{BCH94}, as shown in \cite{Roe02}:
\[
\Ind_{\Gamma}(D_N) = \Ind_{C^*_r(\Gamma)}(D_N) \in K_0(C^*(N)^{\Gamma}) = K_0(C^*_r(\Gamma)).
\]
\begin{corollary}
	\label{cor: PMT for cocompact hypersurface}
	If $N/\Gamma$ is compact, then
		\[
	\Phi_{\Gamma}(\Ind_{\Gamma}(D_N)) = \Ind_{\Gamma}(D;N).
	\]
\end{corollary}
\begin{proof}
We show  that if $N/\Gamma$ is compact, then the three conditions in Theorem \ref{thm: PMT} are automatically satisfied.

Let $d_{N/\Gamma}$ and $d_{M/\Gamma}$ be the Riemannian distances on  $N/\Gamma$ and $M/\Gamma$ with respect to the Riemannian metrics induced by the ones on $N$ and $M$, respectively. The trivial coarse equivalence  $\id_{N/\Gamma}\colon (N/\Gamma, d_{N/\Gamma})\rightarrow (N/\Gamma, \restr{d_{M/\Gamma}}{N})$ of bounded spaces lifts to a $\Gamma$-equivariant coarse equivalence  $\id_{N}\colon (N, d_{N})\rightarrow (N, \restr{d_{M}}{N})$.
%
%
 And we can find a uniform tubular neighbourhood $U \cong (-\varepsilon, \varepsilon)\times N$ of $N$ as the pullback of a uniform tubular neighbourhood of the compact hypersurface $N/\Gamma$ in $M/\Gamma$ along the quotient map.

The quadratic forms $g_{(-\varepsilon, \varepsilon) \times N}$ and $F^*(g_M|_U)$ on $T((-\varepsilon, \varepsilon) \times N)$ restrict to nowhere vanishing, continuous, $\Gamma$-invariant functions on the unit sphere bundle $S$ (with respect to either metric) in $T((-\varepsilon, \varepsilon) \times N)|_{[-\varepsilon/2, \varepsilon/2] \times N}$. The set $S$ is invariant under $\Gamma$, with compact quotient. So the quotient 
\[
\frac{g_{(-\varepsilon, \varepsilon) \times N}}{F^*(g_M|_U)}
\]
has positive upper and lower bounds on $S$, and hence on $T((-\varepsilon, \varepsilon) \times N)|_{[-\varepsilon/2, \varepsilon/2] \times N}$. So \eqref{eq: estimate on Riemannian metric} holds if one replaces $\varepsilon$ by $\varepsilon/2$.

The claim now follows by Theorem \ref{thm: PMT Gamma}.
\end{proof}

Corollaries \ref{cor: cobordism invariance of the index} and \ref{cor: obstruction UPSC metrics} have $\Gamma$-equivariant analogues as well. These also have interesting special cases if $N/\Gamma$ is compact. 
Then the equivariant version of Corollary \ref{cor: cobordism invariance of the index} implies cobordism invariance of the $\Gamma$-index (which seems to be known). The equivariant version of Corollary \ref{cor: obstruction UPSC metrics} now yields obstructions in $K_0(C^*_r(\Gamma))$ to Riemannian metrics on $M/\Gamma$ of uniformly positive scalar curvature. In this case, various numerical obstructions can be extracted from these $K$-theoretic obstructions and computed via index theorems. See for example \cite{CM90, PPT15a, PPT15b}. 

If $N$ is compact and $\Gamma$ is trivial, then Corollary \ref{cor: PMT for cocompact hypersurface} reduces to Corollary \ref{cor: PMT for compact hypersurface}. But Corollary \ref{cor: PMT for cocompact hypersurface} has a more interesting implication if $N$ is compact. Let $\tilde M$ and $\tilde N$ be the universal covers of $M$ and $N$, respectively. Then $\pi_1(N)$ acts on $\tilde M$ via the homomorphism $j\colon \pi_1(N) \to \pi_1(M)$ induced by the inclusion $N \hookrightarrow M$. Suppose that $j$ is injective, so that the action by $\pi_1(N)$  on $\tilde M$ is proper and free. Then $\tilde N$ naturally embeds into $\tilde M$, and its image is preserved by $\pi_1(N)$. 
Let $\tilde D$ be the lift of $D$ to $\tilde M$, and let $D_{\tilde N}$ be the lift of $D_N$ to $\tilde N$. Then Corollary \ref{cor: PMT for cocompact hypersurface} implies the following higher versions of Roe's partitioned manifold index theorem and the resulting obstruction to positive scalar curvature. 
\begin{corollary}
In  this setting,
\[
	\Phi_{\pi_1(N)}(\Ind_{\pi_1(N)}(D_{\tilde N})) = \Ind_{\pi_1(N)}(\tilde D;\tilde N) \quad \in K_0(C^*_r(\pi_1(N))).
\]
\end{corollary}
\begin{corollary}\label{cor higher PMIT}
If $M$ is spin and admits a Riemannian metric of uniformly positive scalar curvature, then $\Ind_{\pi_1(N)}(D_{\tilde N}) = 0$ for the spin-Dirac operator $D_{\tilde N}$ on $\tilde N$, for every compact, connected hypersurface $N \subset M$ for which the map $\pi_1(N) \to \pi_1(M)$ induced by the inclusion $N \hookrightarrow M$ is injective.
\end{corollary}

\subsection{Overview of the disconnected setting}
\label{subsect: overview of setting}
In this subsection, we will discuss how Theorem \ref{thm: PMT} can be extended to the setting where the partitioning hypersurface is not necessarily connected, but has finitely many connected components. 
Applications of Theorem \ref{thm: PMT (mult comp)} analogous to those in Subsection \ref{subsec: corollaries to pmt}, and an equivariant generalization analogous to Theorem \ref{thm: PMT Gamma} are possible; we do not work out the details here.

Consider the setting where $(M, g_M)$ is an oriented, connected,  complete Riemannian manifold of odd dimension at least 3, partitioned by a hypersurface $N \subseteq M$. If $N$ is not assumed to be connected, then in general $M_{\pm}$ need not be connected as well. As a result, the transition manifold $M'$ constructed in Section \ref{subsec: gluing on M'} and the product manifold $\R\times N$ also need not be connected. This poses difficulties to our construction because connectedness is necessary to give a Riemannian manifold the structure of a metric space, which is used in the construction of the (localized) Roe algebras. We will show how to overcome these difficulties by looking at the direct sum of the $K$-theories of the Roe algebras of the separate connected components instead. 

Let us introduce the following notation. If $\{A_i\}_{i \in I}$ and $\{B_i\}_{i \in I}$ are finite collections of Abelian groups and we have homomorphisms $\alpha_i\colon A_i\rightarrow B_i$ for each $i \in I$, then we denote by 
\[
\bigoplus_{i \in I}\alpha_i\colon \bigoplus_{i \in I}A_i\rightarrow \bigoplus_{i \in I}B_i, \qquad \{a_i\}_{i \in I}\mapsto \{\alpha_i(a_i)\}_{i \in I}, 
\]
the homomorphism obtained by applying the homomorphisms $\alpha_i$ entrywise. Similarly, if $\{A_i\}_{i \in I}$ is a finite collection of Abelian groups and we have homomorphisms $\beta_i\colon A_i\rightarrow B$ to a fixed Abelian group $B$ for each $i \in I$, then we denote by 
\[
\sum_{i \in I}\beta_i\colon \bigoplus_{i \in I}A_i\rightarrow B, \qquad \{a_i\}_{i \in I}\mapsto \sum_{i \in I}\beta_i(a_i), 
\]
the homomorphism obtained by summing the images of the homomorphisms $\beta_i$.

As a general rule, we will label connected components with \emph{upper} indices. Since $M_+$ need not be connected, we write $M_+ =: \coprod_{j = 1}^cM_+^j$, where each $M_+^j$ is a connected component of $M_+$. Also define $N^j:= \partial M_+^j$ for the part of $N$ that forms the boundary of $M_+^j$. Since $N^j$ need not be connected as well, we write $N^j:= \coprod_{k_j = 1}^{n_j}N^{j,k_j}$ for the different connected components of $N^j$. So $N = \coprod_{j = 1}^c\coprod_{k_j = 1}^{n_j}N^{j,k_j}$ as a disjoint union of connected components. For each connected component $N^{j,k_j}$ of $N$, denote by $\iota^{j,k_j}\colon N^{j,k_j}\hookrightarrow N$ the inclusion map. This is a uniform map when $N^{j,k_j}$ carries the Riemannian distance and $N$ the subspace distance from $M$. Analogously to \eqref{diag: identification maps}, we define
\begin{equation}
	\Phi^{j,k_j}\colon
	\label{diag: identification maps (mult comp)}
	\begin{tikzcd}
		{K_0(C^*(N^{j,k_j}))} & {K_0(C^*(N, \restr{d_M}{N}))} & {K_0(C^*(N\subseteq M)).}
		\arrow["{(\iota^{j,k_j})_*}", from=1-1, to=1-2]
		\arrow["\sim", from=1-2, to=1-3]
	\end{tikzcd}
\end{equation}
Then we get the following generalization of Theorem \ref{thm: PMT}.
\begin{theorem}[Partitioned manifold index theorem for disconnected hypersurfaces]
	\label{thm: PMT (mult comp)}
	Let $M$ be an oriented, complete Riemannian manifold of odd dimension at least 3, partitioned by a hypersurface $N$. Let $S$ be a Dirac bundle over $M$ with Dirac connection $\nabla$ and let $D$ be the associated Dirac operator on $S$. Let $D^{j,k_j}$ be the Dirac operator on $S^{j,k_j}:= \restr{S}{N^{j,k_j}}$ associated to the restricted Clifford action and an arbitrary Dirac connection on $S^{j,k_j}$. Suppose moreover that the following hold:
	\begin{enumerate}
		\item \label{item: assumption 1 (mult comp)} The map $\id_{N^{j,k_j}}\colon (N^{j,k_j}, d_{N^{j,k_j}})\rightarrow (N^{j,k_j}, \restr{d_M}{N^{j,k_j}})$ is a uniform equivalence for each connected component $N^{j,k_j}$ of $N$.
		\item \label{item: assumption 2 (mult comp)} The map $\id_{M_+^j}\colon (M_+^j, d_{M_+^j})\rightarrow (M_+^j, \restr{d_M}{M_+^j})$ is a uniform equivalence for each connected component $M_+^j$ of $M_+$.
		\item \label{item: assumption 3 (mult comp)} The submanifold $N\subseteq M$ has a uniform tubular neighbourhood $U$ of radius $\varepsilon >0$. 
		\item \label{item: assumption 4 (mult comp)} Let $F\colon (-\varepsilon, \varepsilon)\times N \rightarrow U$ be defined as in \eqref{eq: tub nbhd map} and let $g_{(-\varepsilon, \varepsilon)\times N}$ be the product metric on the manifold $(-\varepsilon, \varepsilon)\times N$. Then there exists a constant $C \geq 1$ such that on $U \cong (-\varepsilon, \varepsilon)\times N$ the estimates
		\begin{equation}
			\label{eq: estimate on Riemannian metric (mult comp)}
			\frac{1}{C}g_{(-\varepsilon, \varepsilon)\times N}\leq F^*(\restr{g_M}{U}) \leq Cg_{(-\varepsilon, \varepsilon)\times N}
		\end{equation} 
		hold as quadratic forms on vectors.
	\end{enumerate}
	
	Then
	\[
	\sum_{j= 1}^c\sum_{k_j = 1}^{n_j}\Phi^{j,k_j}(\Ind(D^{j,k_j})) = \Ind(D;N)
	\]
	with each $\Phi^{j,k_j}$ defined as in \eqref{diag: identification maps (mult comp)}.
\end{theorem}

\begin{remark}
	\label{rem: extra assumption}
	\begin{enumerate}
		\item Note that assumption \ref{item: assumption 2 (mult comp)} in Theorem \ref{thm: PMT (mult comp)} has no counterpart in Theorem \ref{thm: PMT}. In Section \ref{subsect: adaptations to the proof} we will show that we have to add this assumption to prove a counterpart to Proposition \ref{prop: uniform equivalence 1} in this new setting. We will also discuss that if $N$ is connected, then assumption \ref{item: assumption 2 (mult comp)} follows from assumptions \ref{item: assumption 1 (mult comp)}, \ref{item: assumption 3 (mult comp)} and \ref{item: assumption 4 (mult comp)} and therefore Theorem \ref{thm: PMT (mult comp)} is a proper generalization of Theorem \ref{thm: PMT}.			
		\item If the hypersurface $N$ is compact, also assumption \ref{item: assumption 2 (mult comp)} in Theorem \ref{thm: PMT (mult comp)} is trivially satisfied. This follows from the fact that by compactness, the boundary $N^j$ of $M^j_+$ in $M$ is bounded with respect to both $d_{M^j_+}$ and $d_M$.
		
	\end{enumerate}
\end{remark}

\subsection{Adaptations to the proof in the disconnected case}
\label{subsect: adaptations to the proof}
In this section, we will briefly sketch where and how the proof of Theorem \ref{thm: PMT} should be modified to obtain a proof of Theorem \ref{thm: PMT (mult comp)}. The main modifications to the proof of the theorem occur in Sections \ref{subsec: half-isos and the transition map} and \ref{subsec: half-isos and the partitioned index} and in Proposition \ref{prop: uniform equivalence 1}. 

To introduce the right notion of a half-isomorphism, note that when we construct the transition manifold $M'$ as in Section \ref{subsec: gluing on M'}, then $M'$ has exactly the same number of connected components as $M_+$. So we can write $M'= \coprod_{j = 1}^c M^{\prime, j}$ with $M_+^j\subseteq M^{\prime, j}$. So we will need to compare the \emph{connected} manifold $M$ to the disconnected manifold $M'$. Similarly, we will need to compare the \emph{connected} manifolds $M^{\prime, j}$ to the disconnected product manifold $\R\times N^j$. This suggests that our adapted notion of a half-isomorphism should allow for the manifold $M_2$ to be disconnected. We therefore introduce the following variation on Definition \ref{def: half iso}.

\begin{definition}
	\label{def: half iso (mult comp)}
	For $i = 1, 2$, let $(M_i, g_i, S_i, h_i)$ be two sets of complete Riemannian manifolds $(M_i, g_i)$ with Hermitian vector bundles $(S_i, h_i)$. Let $D_i$ be differential operators acting on the smooth sections of $S_i$. Suppose that both $M_i$ are partitioned by a simple hypersurface $N_i\subseteq M_i$ and have been given a partition labeling $M_{i\pm}$. Suppose that $M_1$ is connected and that each connected component $M_2^j$ of $M_2$ contains exactly one connected component $M_{2+}^j$ of $M_{2+}$.
	\begin{enumerate}
		\item \label{item: half iso (mult comp)} A \emph{half-isomorphism} is a set of disjoint opens $O^j_i \supseteq M_{i+}^j$ with $O^j_2\subseteq M_2^j$ such that the function $\phi_{+}$ as in Definition \ref{def: simple hypersurface} can be chosen such that $\supp(\phi_{+})\subseteq O_i:=\coprod_{j = 1}^cO_i^j$, and a pair of maps $(\phi, u)$ with $\phi\colon O_1\rightarrow O_2$ an orientation-preserving Riemannian isometry such that $\phi(N_1) = N_2$ and $\phi(M_{1+}) = M_{2+}$, and $u\colon \restr{S_1}{O_1}\rightarrow \restr{S_2}{O_2}$ a fiberwise unitary vector bundle morphism covering $\phi$. We will usually suppress the opens $O^j_i$ in the notation of a half-isomorphism.
		\item \label{item: uniform half iso (mult comp)} A \emph{uniform half-isomorphism} is a half-isomorphism $(\phi, u)$ such that the map $\phi$ extends to a homeomorphism $\phi\colon \ol{O}^j_1\rightarrow \ol{O}^j_2$ that is a uniform equivalence, when $\ol{O}^j_1$ is considered as a metric space with the distance function $\restr{d_{M_1}}{\ol{O}^j_1}$ and $\ol{O}^j_2$ is considered as a metric space with the distance function $\restr{d_{M^j_2}}{\ol{O}^j_2}$.
		\item \label{item: compatible half iso (mult comp)} A half-isomorphism $(\phi, u)$ is \emph{compatible with the differential operators} if 
		\begin{equation}
			\label{eq: compatibility with diff ops (mult comp)}
			u\circ D_1\sigma \circ \phi^{-1} = D_2(u\circ\sigma\circ\phi^{-1})
		\end{equation}
		for all $\sigma \in C^{\infty}(O_1;S_1)$.
	\end{enumerate}
\end{definition}

With this notion of a half-isomorphism, we can construct a partial transition map for each connected component of $M_2$ as a variation on the transition map from Definition \ref{def: transition map}. We adopt all the notation introduced in Section \ref{subsec: half-isos and the transition map} concerning the Hilbert spaces and isometric inclusions. Moreover, denote $H_{i+}^j:=L^2(M^j_{i+}; \restr{S}{M^j_{i+}})$ and write $W^j_i\colon H^j_{i+}\hookrightarrow L^2(M_i;S_i)$ and $X^j\colon H^j_{2+}\hookrightarrow L^2(M_2^j;\restr{S}{M_2^j})$ for the isometric inclusions of Hilbert spaces. It is easy to check that these uniformly cover the inclusion maps on the level of coarse spaces and that conjugation with these isometries also preserves the property of being supported near the hypersurface. Moreover, the unitary $\Gamma\colon H_{1+}\rightarrow H_{2+}$ splits as a direct sum of unitary operators 
\[
\bigoplus_{j = 1}^c\Gamma^j\colon \bigoplus_{j = 1}^cH^j_{1+}\rightarrow \bigoplus_{j = 1}^cH^j_{2+}
\]
and we denote by $\widetilde{\Gamma^j}\colon L^2(M_1;S_1)\rightarrow L^2(M_2^j;\restr{S_2}{M_2^j})$ the extension of $\Gamma^j$ that equals 0 on $(H^j_{1+})^{\bot}$.

\begin{definition}
	\label{def: (partial) transition map (mult comp)}
	\begin{enumerate}
		\item Let $(\phi, u)$ be a uniform half-isomorphism between two partitioned manifolds $N_i\subseteq M_i$. Then the map $\cT^j$ defined by 
		\[\begin{tikzcd}
			{K_0(C^*(N_1\subseteq M_1))} && {K_0(C^*(N_1^j\subseteq M_{1+}^j))} \\
			{K_0(C^*(N_2^j\subseteq M_2^j))} && {K_0(C^*(N_2^j\subseteq M_{2+}^j))}
			\arrow["{K_0(\Ad_{W^j_1})}"', from=1-3, to=1-1]
			\arrow["{\cT^j}", dotted, from=2-1, to=1-1]
			\arrow["{K_0(\Ad_{X^j})^{-1}}"', from=2-1, to=2-3]
			\arrow["{K_0(\Ad_{\Gamma^{j, *}})}"', from=2-3, to=1-3]
		\end{tikzcd}\]
		is the \emph{partial transition map} for the connected component $M^j_2$ of $M_2$ induced by the uniform half-isomorphism $(\phi, u)$.
		\item The resulting map 
		\[
		\cT:= \sum_{j = 1}^c\cT^j\colon \bigoplus_{j =1}^cK_0(C^*(N_2^j\subseteq M_2^j)) \rightarrow K_0(C^*(N_1\subseteq M_1))
		\]
		is the \emph{transition map} induced by the uniform half-isomorphism $(\phi, u)$.
	\end{enumerate}
\end{definition}

The main goals of Sections \ref{subsec: half-isos and the transition map} and \ref{subsec: half-isos and the partitioned index} were to show that Diagram \eqref{diag: transition diagram} commutes (Proposition \ref{prop: transition diagram commutes}) and that the transition map preserves the partitioned index (Proposition \ref{prop: iso on localized Roe algs preserves partitioned index}). If $(\phi, u)$ is a half-isomorphism between partitioned manifolds $N_i\subseteq M_i$, then we denote by $\phi^{j,k_j}\colon N_1^{j, k_j}\rightarrow N_2^{j, k_j}$ the induced isometries on the connected components of the hypersurfaces. Then as a counterpart to Proposition \ref{prop: transition diagram commutes}, we obtain the following result.

\begin{proposition}
	\label{prop: transition diagram commutes (mult comp)}
	Let $(\phi, u)$ be a uniform half-isomorphism between two partitioned manifolds $N_i\subseteq M_i$ with Hermitian bundles $S_i$ and let $\cT$ be the transition map induced by the half-isomorphism. Then the following diagram commutes: 
	\[\begin{tikzcd}
		{\bigoplus_{j = 1}^c\bigoplus_{k_j = 1}^{n_j}K_0(C^*(N_1^{j, k_j}))} && {K_0(C^*(N_1\subseteq M_1))} \\
		{\bigoplus_{j = 1}^c\bigoplus_{k_j = 1}^{n_j}K_0(C^*(N_2^{j, k_j}))} && {\bigoplus_{j = 1}^cK_0(C^*(N_2^j\subseteq M_2^j))}
		\arrow["{\sum_{j = 1}^c\sum_{k_j= 1}^{n_j}\Phi_1^{j, k_j}}", from=1-1, to=1-3]
		\arrow["{\bigoplus_{j = 1}^c\bigoplus_{k_j = 1}^{n_j}\phi^{j, k_j}_*}"', from=1-1, to=2-1]
		\arrow["{\bigoplus_{j = 1}^c(\sum_{k_j = 1}^{n_j}\Phi_2^{j, k_j})}"', from=2-1, to=2-3]
		\arrow["{\cT}"', from=2-3, to=1-3]
	\end{tikzcd}\]
\end{proposition}

Also denote by $D_2^j:= \restr{D_2}{M_2^j}$ the restriction of the Dirac operator $D_2$ to the connected component $M_2^j$ of $M_2$. Then as a counterpart to Proposition \ref{prop: iso on localized Roe algs preserves partitioned index}, we obtain the following result.
\begin{proposition}
	\label{prop: iso on localized Roe algs preserves partitioned index (mult comp)}
	For $i= 1, 2$, let $N_i\subseteq M_i$ be two partitioned manifolds with Hermitian bundles $S_i$ and symmetric, elliptic differential operators $D_i$ with finite propagation speed. Let $(\phi, u)$ be a uniform half-isomorphism between them that is compatible with the differential operators and let $\cT$ be the induced transition map. Then 
	\[
	\cT(\Ind(D^1_2;N^1_2), \ldots, \Ind(D^c_2;N^c_2)) = \sum_{j = 1}^c \cT^j(\Ind(D^j_2;N^j_2))= \Ind(D_1;N_1).
	\]
\end{proposition}

The proof of Proposition \ref{prop: iso on localized Roe algs preserves partitioned index (mult comp)} is very similar as the proof of Proposition \ref{prop: iso on localized Roe algs preserves partitioned index}, but instead uses the following generalization of Lemma \ref{lem: Gamma tilde intertwines differential operators}.

\begin{lemma}
	\label{lem: Gamma tilde intertwines differential operators (mult comp)}
	The following equality holds: 
	\begin{equation*}
		q(P_{1+}(\ol{D}_1+i)^{-1}P_{1+}) = \sum_{j = 1}^cq(\widetilde{\Gamma^j}^*(\ol{D}^j_2+i)^{-1}\widetilde{\Gamma^j}) \in C^*(M_1)/C^*(N_1\subseteq M_1).
	\end{equation*}
\end{lemma}

Finally, most of the results in Sections \ref{subsec: gluing on M'} and \ref{subsec: gluing on S'} hold true without changes, except for Proposition \ref{prop: uniform equivalence 1} and \ref{prop: uniform equivalence 2}. These are replaced by the following two results.

\begin{proposition}
	\label{prop: uniform equivalence 1 (mult comp)}
	For each connected component $M_+^j$ of $M_+$, the identity map 
	\[
	(M^j_+, \restr{d_M}{M^j_+})\rightarrow (M^j_+, \restr{d_{M^{\prime, j}}}{M^j_+})
	\]
	is a uniform equivalence.
\end{proposition}

\begin{proposition}
	\label{prop: uniform equivalence 2 (mult comp)}
	For each connected component $N^{j, k_j}$ of $N$, the identity map
	\[
	(\R_{\leq0}\times N^{j, k_j}, \restr{d_{M^{\prime, j}}}{\R_{\leq0}\times N^{j, k_j}})\rightarrow (\R_{\leq0}\times N^{j, k_j}, \restr{d_{\R\times N^{j, k_j}}}{\R_{\leq0}\times N^{j, k_j}})
	\]
	is a uniform equivalence.
\end{proposition}

Proposition \ref{prop: uniform equivalence 2 (mult comp)} is proven in exactly the same way as Proposition \ref{prop: uniform equivalence 2}. However, for Proposition \ref{prop: uniform equivalence 1 (mult comp)} this is not the case. The proof of Proposition \ref{prop: uniform equivalence 1} uses the fact that we could replace a segment on which a curve leaves $M_+$ with a curve through $N$. This obviously only holds true since $N$ is connected. In general, we cannot do the same for a segment on which a curve leaves a connected component $M_+^j$ of $M_+$ in the disconnected setting, as $N^j = \partial M_+^j$ need not be connected and therefore the curve can leave and enter $M_+^j$ via different connected components of $N^j$. This is where the extra assumption \ref{item: assumption 2 (mult comp)} in Theorem \ref{thm: PMT (mult comp)} is used to replace such a segment with a curve through $M_+^j$.

With these adaptations to the results in Section \ref{section: proof general case}, we can now prove Theorem \ref{thm: PMT (mult comp)} in a very similar way as Theorem \ref{thm: PMT}, reducing a general partitioned manifold to $\sum_{j =1}^cn_j$ product manifolds and applying Theorem \ref{thm: PMT for product} to each product manifold.

\bibliographystyle{plain}
\bibliography{references}
\end{document}